\documentclass[11pt,english,a4paper,leqno]{smfart}
\usepackage[english]{babel}
\usepackage{amssymb}
\usepackage{enumerate}
\usepackage{amsmath}

\newtheorem{theorem}{Theorem}[section]
\newtheorem{lemma}[theorem]{Lemma}

\newtheorem{proposition}[theorem]{Proposition}

\newtheorem*{coro*}{Corollary}

{\theoremstyle{definition}}

{\theoremstyle{definition}\newtheorem{example}[theorem]{Example}}

\theoremstyle{definition}
\newtheorem{definition}[theorem]{Definition}
\newtheorem{question}[theorem]{Question}

\newtheorem*{fact*}{Fact}

\newtheorem*{claim*}{\rm Claim}

{\theoremstyle{definition}\newtheorem{remark}[theorem]{Remark}}

\def\D{\ensuremath{\mathbb D}}

\def\T{\ensuremath{\mathbb T}}

\def\Z{\ensuremath{\mathbb Z}}

\def\C{\ensuremath{\mathbb C}}

\def\N{\ensuremath{\mathbb N}}

\newcommand{\LL}[1]{L^{#1}}
\newcommand{\expo}[2]{{#1}^{#2}}
\newcommand{\kin}{k\in\Z}
\newcommand{\lp}[2]{\ell_{#1}(#2)}
\newcommand{\co}[1]{c_{0}(#1)}
\newcommand{\wn}[1]{w_{1}\dots w_{#1}}
\newcommand{\bw}{B_{w}}

\DeclareMathOperator{\vect}{span}

\newcommand{\hy}{hypercyclic}
\newcommand{\fhy}{frequently hypercyclic}
\newcommand{\ops}{operators}
\newcommand{\op}{operator}

\newcommand{\wrt}{with respect to}

\newcommand{\nz}{non-zero}

\newcommand{\inv}{invariant}

\newcommand{\erg}{ergodic}
\newcommand{\mea}{measure}
\newcommand{\prob}{probability}

\newcommand{\ifff}{if and only if}

\newcommand{\pss}[2]{\ensuremath{{\langle #1,#2\rangle}}}

\newcommand{\To}{\longrightarrow}

\newcommand{\bb}{\mathcal{B}}

\newcommand{\np}{(n+1)}
\newcommand{\en}[1]{\{0,\dots,#1\}}

\newcommand{\ba}[1]{\overline{#1}}

\begin{document}

\date{\today}

\title[New examples of universal hypercyclic operators]{Some new examples of universal hypercyclic operators in the sense of Glasner and Weiss}

\author{Sophie Grivaux}
\address{CNRS,
Laboratoire Ami\'enois de Math\'ematique Fondamentale et Appliqu\'ee, UMR 7352,
Universit\'e de Picardie Jules Verne,
33 rue Saint-Leu,
80039 Amiens Cedex 1,
France}
\email{sophie.grivaux@u-picardie.fr}

\begin{abstract}
A bounded operator $A$ on a real or complex  separable infinite-dimensional Banach space $Z$ is \emph{universal} in the sense of Glasner and Weiss if for every invertible \erg\ \mea-preserving transformation $T$ of a standard Lebesgue \prob\ space $(X,\bb,\mu )$, 
there exists an $A$-\inv\ \prob\ \mea\ $\nu $ on $Z$ with full support such that the two dynamical systems $(X,\bb,\mu ;T)$ and $(Z,\bb_{Z},\nu ;A)$ are isomorphic. We present a general and simple criterion for an operator to be universal, which allows us to characterize universal operators among unilateral or bilateral weighted shifts on $\ell_{p}$ or $c_{0}$, to show the existence of universal operators on a large class of Banach spaces, and to give a criterion for universality in terms of unimodular eigenvectors. We also obtain similar results for operators which are universal for all ergodic systems (not only for invertible ones), and study necessary conditions for an operator on a Hilbert space to be universal.
\end{abstract}

\subjclass{47 A 16, 37 A 35, 47 A 35, 47 B 35, 47 B 37}
\thanks{This work was supported in part by the Labex CEMPI (ANR-11-LABX-0007-01)}

\maketitle
\section{Introduction and main results}\label{Sec1}
Let $G$ be a topological group, and $Z$ a real or complex  separable infinite-dimensional Banach space. We denote by $\mathcal{B}(Z)$ the set of bounded linear \ops\ on $Z$. Let $S:G\rightarrow {Z}$, $g\mapsto S_{g}$ be a representation of the group $G$ by bounded \ops\ on 
$Z$. If $\bb{}_{Z}$ denotes the Borel $\sigma $-field of $Z$, and $\nu $ 
is a Borel \prob\ \mea\ on $Z$ which is $S$-\inv\ (i.e.\ $\nu $ is $S_{g}$-\inv\ for every $g$ in $G$), then $S$ naturally defines a \prob-preserving action of the group $G$ on the 
\prob\ space $(Z,\bb_{Z},\nu )$. Recall that the measure $\nu $ is said to have full support if $\nu (U)>0$ for any non-empty open subset $U$ of $Z$.
\par\smallskip 
Glasner and Weiss introduced in the paper \cite{GW} the following notion 
of a universal representation:

\begin{definition}\label{DefGW}\cite{GW} 
The representation $S=(S_{g})_{g\in G}$ of the group $G$ on the Banach space $Z$ is said to be \emph{universal} if for every \erg\ \prob-preserving free action $T=(T_{g})_{g\in G}$ of $G$ on a standard Lebesgue \prob\ space $(X,\bb,\mu )$, there exists a Borel \prob\ \mea\ $\nu 
$ on $Z$ with full support which is $S$-\inv\  and such that  the two 
actions of $T$ and $S$ of $G$ on  $(X,\bb,\nu )$ and $(Z,\bb_{Z},\nu )$  respectively are isomorphic.
\end{definition}

Recall that $(T_{g})_{g\in G}$ is \emph{free} if for any element $g\in G$ different from the identity, $\mu (\{x\in X;\, T_{g}x=x\})=0$, and \emph{\erg}\ if the following holds true: if $A\in \mathcal{B}$ is such that $T_{g}^{-1}(A)=A$ for every $g\in G$, then $\mu (A)(1-\mu (A))=0$.

A universal representation of $G$ thus simultaneously models every 
possible free \erg\ action of $G$ on a \prob\ space. The existence of a universal representation is shown in \cite{GW} for a large class of groups $G$, including all countable discrete groups and all locally compact, second countable, compactly generated groups.
\par\smallskip 
When $G=\Z$, the main result of \cite{GW} states thus that there exists a 
bounded invertible \op\ $S$ on $H$ which is \emph{universal} in the following sense: for every invertible \erg\ \prob-preserving transformation $T$ of a standard Lebesgue \prob\ space $(X,\bb,\mu )$, 
there exists an $S$-\inv\ \prob\ \mea\ $\nu $ on $H$ with full support such that the two dynamical systems $(X,\bb,\mu ;T)$ and $(H,\bb_{H},\nu ;S)$ are isomorphic. Observe that any invertible ergodic \prob-preserving transformation $T$ of $(X, \mathcal{B},\mu)$ acts freely on $(X, \mathcal{B},\mu)$: for any $n\in\Z$, the set $\{x\in X \textrm;\ T^{n}x=x\}$ is $T$-invariant, and the ergodicity of $T$ implies that it is of $\mu$-measure zero. This definition of a universal operator is thus coherent with Definition \ref{DefGW}.
\par\smallskip 
Any of the systems $(H,\bb_{H},\nu ;S)$ is what is called a \emph{linear 
dynamical system}, i.e.\ a system given by the action of a bounded linear \op\ $A$ on  an infinite-dimensional separable Banach space $Z$. These systems can be studied from both the topological point of view and the ergodic point of view (when one endows the Banach space $Z$ with an $A$-\inv\ \prob\ \mea), and we refer the reader to the 
two books \cite{BM} and \cite{GP} for more on this particular class of dynamical systems.
\par\smallskip 
The result of \cite{GW}, when specialized to the case where $G=\Z$, thus 
says that \emph{any} invertible ergodic \prob-preserving dynamical system can be represented as a linear dynamical system where the underlying space is a Hilbert space, and, moreover, 
the same operator $S$ on $H$ can serve as a model for any such dynamical 
system. Given the apparent rigidity entailed by linearity, the universality result of 
\cite{GW} in the case where $G=\Z$
may seem  rather surprising. It is worth pointing out here that a topological version of this result had been obtained previously by Feldman in \cite{F}: there exists a bounded operator $A$ on
the Hilbert space $\lp{2}{\N}$ 
which has the following property: whenever $\varphi $ is a continuous self-map of a compact metrizable space $K$, there exists a compact subset $L$ of $\lp{2}{\N}$ which is $A$-invariant and an homeomorphism $\Phi :K\rightarrow L$ such that $\varphi =\Phi 
^{-1}\circ A\circ\Phi $. The proof of this topological result is rather straightforward, but it already gives a hint at the richness of the class of linear dynamical systems.
\par\smallskip 
A bounded operator $A$ acting on the Banach space $Z$ is said to be 
\emph{\hy}\ if it admits a vector $z\in Z$ whose orbit $\{A^{n}z;\ n\ge 0\}$ is dense in $Z$, and \emph{\fhy}\ if there exists a vector $z\in Z$ such that for every non-empty open subset $V$ of $Z$, the set $\{n\ge 0;\ A^{n}z\in V\}$ has positive lower density. If $A$ admits an invariant probability measure with full support with respect to which it is \erg, 
then Birkoff's \erg\ theorem is easily seen to imply that almost all vectors of $Z$  
are \fhy\ for $A$. Thus any universal operator is \fhy. Let us now say a few words about the construction of universal operators of \cite{GW}.

\par\smallskip 

The universal operators constructed in \cite{GW} are shift operators on 
certain weighted  $\ell_{p}$-spaces of sequences on $\Z$ for $1<p<+\infty $, or, equivalently, weighted shift operators on $\lp{p}{\Z}$. The proof uses in a crucial way an ergodic theorem for certain random walks of Jones, Rosenblatt, and Tempelman \cite{JRT}. This theorem states 
in particular that whenever $\eta $ is a symmetric strictly aperiodic probability measure on 
$\Z$, the following holds true: for any probability-preserving dynamical system $(X,\bb,\mu ;T)$ and any function $f\in L^{p}(X,\bb,\mu )$, $1<p<+\infty $, the powers $A_{\eta }^{n}f$ of the random walk operator on $\Z$ defined by \[
A_{\eta }f(x)=\sum_{k\in\Z}f(T^{k}x)\,\eta (k)\]
converge for almost every $x\in X$ to the projection $P_{\mathcal{J}}f$ of $f$ onto the subspace $\mathcal{J}$ of $L^{p}(X,\bb,\mu )$ consisting of $T$-invariant functions. This ergodic theorem can be applied for instance starting from the measure $\eta =(\delta _{-1}+\delta _{0}+\delta_{1} )/3$ on $\Z$. If $(p_{n})_{n\ge 1}$ is a sequence of positive real numbers such 
that $\sum_{n\ge 1}p_{n}=1$ and $\sup\, (p_{n}/p_{n+1})<+\infty $, the weights considered in \cite{GW} are defined by setting $w_{k}:=\sum_{n\ge 1}p_{n}\,\eta ^{*n}(k)$ for every $k\in\Z$. If $S$ is the shift operator defined on 
\[
\lp{p}{\Z,w}:=\{\xi =(\xi _{k})_{k\in\Z};\ \sum_{k\in \Z}|\xi 
_{k}|^{p}\,w_{k}<+\infty 
\}
\]
by setting $S\xi =(\xi_{k+1} )_{k\in\Z}$ for each $\xi \in \lp{p}{\Z,w}$, 
then $S$ is shown in \cite{GW} to be bounded, and the ergodic theorem of \cite{JRT} is then used to prove that for any function $f\in L^{2p}(X,\bb,\mu )$ the sequence $(f(T^{k}x))_{k\in\Z}$ belongs to $\lp{p}{\Z,w}$ for $\mu $-almost every $x\in X$. 
Setting
\begin{align*}
 \Phi _{f}:(X,\bb,\mu 
)&\longrightarrow(\lp{p}{\Z,w},\bb_{\lp{p}{\Z,w}},\nu _{f})\\
 x&\longmapsto (f(T^{k}x))_{k\in\Z}
\end{align*}
where $\nu _{f}$ is the measure on $\lp{p}{\Z,w}$ defined by $\nu 
_{f}(B)=\mu (\Phi _{f}^{-1}(B))$ for any Borel subset $B$ of $\lp{p}{\Z,w}$, it is easy to check that $\Phi 
_{f}$ intertwines the actions of $T$ on $(X,\bb,\mu )$ and of $S$ on 
$\lp{p}{\Z,w}$. The last (and most difficult) step of the proof of \cite{GW} is then to construct a function $f$ such that $\Phi _{f}$ is an isomorphism of dynamical systems and $\nu _{f}$ has full support.

\par\smallskip 

Our aim in this paper is to present an alternative construction of 
universal operators, which is elementary in the sense that it avoids the use of an ergodic theorem such as the one of \cite{JRT}. It is also more flexible than the construction of \cite{GW}, yields some rather simple criteria for universality, and allows us to show the existence of universal operators on a large class 
of Banach spaces. Moreover, this construction makes it possible to exhibit operators which are universal for \emph{all} ergodic dynamical systems, not only for invertible ones. As we will often need to make a distinction between these two notions, we introduce the following definition:

\begin{definition}\label{Def0}
 Let $A$ be a bounded operator on a real or complex Banach space $Z$.
 
\begin{enumerate}
 \item [$\bullet$] We say that $A$ is \emph{universal for invertible 
ergodic systems} if for every invertible ergodic dynamical system $(X,\bb,\mu ;T)$ on a standard Lebesgue probability space there exists a probability measure $\nu $ on $Z$ with full support which is $A$-invariant and such that the dynamical systems $(X,\bb,\mu ;T)$ and $(Z,\bb_{Z},\nu ;A)$ are isomorphic.

\item[$\bullet$] We say that $A$ is \emph{universal for ergodic systems} 
if the same property holds true  for all ergodic dynamical systems $(X,\bb,\mu ;T)$ on a standard Lebesgue probability space.
\end{enumerate}
\end{definition}
Universal operators in the sense of Glasner and Weiss are universal for 
invertible ergodic systems. When we use simply the term ``universal operator'' in the rest of the paper, we will mean an operator which is universal either for all ergodic systems or just for invertible ones.
Before stating our main results, we introduce the following intuitive notation: suppose that $A$ is a bounded operator on a real or complex separable Banach space $Z$, and suppose that $(z_{n})_{n\in\Z}$ is a sequence of vectors of $Z$ such that, for every $n\in\Z$, $Az_{n}=z_{n+1}$. We then write $z_{n}=A^{n}z_{0}$ for every $n\in\Z$.
\par\smallskip 
Our first result consists of a general and simple criterion for an 
operator to be universal for invertible ergodic systems.

\begin{theorem}\label{Theo1}
 Let $A$ be a bounded operator on a real or complex separable Banach space $Z$. Suppose that there exists a sequence $(z_{n})_{n\in\Z}$ of vectors of $Z$ such that, for every $n\in\Z$, $Az_{n}=z_{n+1}$, and such  that the following three properties hold true:
 
\begin{enumerate}
 \item [\emph{(a)}] the vector $z_{0}$ is bicyclic, i.e.\ 
$\overline{\vphantom{[}\vect}\,[A^{-n}z_{0};\ n\in\Z\,]$=Z;

\item[\emph{(b)}] there exists a finite subset $F$ of $\Z$ such that 
$\overline{\vphantom{[}\vect}\,[A^{-n}z_{0};\ n\in\Z\setminus F\,]\neq Z$;

\item[\emph{(c)}] the series $\sum_{n\in\Z}A^{-n}z_{0}$ is unconditionally 
convergent in $Z$.
\end{enumerate}

Then $A$ is universal for invertible ergodic systems.
\end{theorem}

There is a very similar criterion which implies that an operator is 
universal for all ergodic systems:

\begin{theorem}\label{Theo1bis}
Let $A$ be a bounded operator on a real or complex separable Banach 
space $Z$. If $A$ satisfies the assumptions of Theorem \ref{Theo1}, and if moreover the sequence $(z_{n})_{n\in \Z}$ is such that $A^{r}z_{0}=0$ for some $r\in\Z$ (or, equivalently, such that $z_{0}=0$), $A$ is universal for ergodic systems.
\end{theorem}

We have already mentioned that a universal operator is necessarily frequently hypercyclic. An operator satisfying the assumptions of either Theorem \ref{Theo1} or Theorem \ref{Theo1bis} is easily seen to satisfy the Frequent Hypercyclicity Criterion of \cite{BoGr} (see also \cite{BM} or \cite{GP}), and so is in particular frequently hypercyclic and chaotic.
\par\smallskip
The proofs of Theorems \ref{Theo1} and \ref{Theo1bis} largely rely 
on the ideas of \cite{GW}, but some extra work is needed, in particular in order to cope with the condition (b) in both theorems. The proofs would be simpler if we assumed that $F=\{0\}$ 
(which is what happens in some of the examples, in particular in those of 
\cite{GW}), but the generality of assumption (b) is needed in several of the examples given in Section \ref{Sec4}.

\par\smallskip 

The proofs of Theorems \ref{Theo1} and 
\ref{Theo1bis} are presented in Section \ref{Sec2}, as well as two 
generalizations of these results (Theorems \ref{Theo1-3} and 
\ref{Theo1-4}) in which assumption (a) is relaxed. The next two sections are devoted to applications and examples. In Section \ref{Sec3}, we characterize universal operators (both for ergodic systems and for invertible ergodic systems) among unilateral or bilateral weighted backward shifts on the spaces $\lp{p}{\N}$, $1\le p<+\infty $ or $\co{\N}$. Recall that if $(e_{n})_{n\ge 0}$ denotes the canonical basis of $\lp{p}{\N}$, 
or $\co{\N}$, and $(w_{n})_{n\ge 1}$ is a bounded sequence of non-zero 
complex numbers, the weighted backward shift $\bw $ is defined
 on $\lp{p}{\N}$ or $\co{\N}$ by setting $\bw e_{0}=0$ and $\bw 
 e_{n}=w_{n}e_{n-1}$ for every $n\ge 1$. In the same way, if $(f_{n})
 _{n\in \Z}$ is the canonical basis of $\lp{p}{\Z}$ or $\co{\Z}$, and 
$(w_{n})_{n\in\Z}$ is again a bounded sequence of non-zero complex 
numbers, the bilateral weighted shift $S_{w}$ on $\lp{p}{\Z}$ or 
$\co{\Z}$ is defined by setting $S_{w}e_{n}=w_{n}e_{n-1}$ for every 
$n\in\Z$. Here is the characterization of universal weighted shifts which 
can be obtained thanks to Theorems \ref{Theo1} and \ref{Theo1bis}:

\begin{theorem}\label{Theo2}
 With the notations above, the unilateral backward weighted shift $\bw $ 
is universal for (invertible) ergodic systems on $\lp{p}{\N}$, $1\le p<+\infty $, if 
and only if the series $$\sum_{n\ge 1}\frac{1}{|w_{1}\dots w_{n}|^{p}}$$ is 
convergent. It is universal for  (invertible) ergodic systems on $\co{\N}$ if and only 
if 
$|w_{1}\dots w_{n}|\longrightarrow 0$ as $n\longrightarrow +\infty $. 
\par\smallskip
In the same way, the bilateral backward weighted shift $S_{w}$ is 
universal for (invertible)  ergodic systems on $\lp{p}{\Z}$, $1\le p<+\infty $, if and 
only if the series $$\sum_{n\ge 1}\frac{1}{|w_{1}\dots w_{n}|^{p}}+\sum_{n\ge 1}{|w_{0}\dots w_{-(n-1)}|^{p}}$$ is convergent. 
It is universal for (invertible) ergodic systems on $\co{\Z}$ if and only if 
$|w_{1}\dots w_{n}|\longrightarrow 0$ and
$|w_{0}\dots w_{-(n-1)}|\longrightarrow +\infty $ as $n\longrightarrow +\infty $.
\end{theorem}
This result shows in particular the existence of universal operators for 
ergodic systems living on any of the spaces $\lp{p}{\N}$, $1\le 
p<+\infty $, or $\co{\N}$. The existence of universal operators for invertible \erg\ systems on $\lp{p}{\N}$, $1< p<+\infty $,  is already proved in \cite{GW}. A natural question, asked in \cite{GW}, is to 
determine which Banach (or Fr\'{e}chet) spaces support a universal 
operator. As a universal operator is necessarily frequently hypercyclic, 
and some Banach spaces (like the hereditarily indecomposable spaces, for 
instance), do not support frequently hypercyclic operators, it follows 
that 
not all Banach spaces support a universal operator. But, as a consequence 
of Theorem \ref{Theo1}, we obtain the existence of such operators on 
Banach spaces with a sufficiently rich structure.

\begin{theorem}\label{Theo4}
Let $Z$ be a separable infinite dimensional Banach space containing  a complemented 
copy of a space  with a sub-symmetric basis. Then $Z$ supports an operator 
which is universal for all ergodic systems.
\end{theorem}

This result implies for example that any separable Banach space containing a 
complemented copy of one of the spaces $\lp{p}{\N}$, $1\le p<+\infty $,
or $c_{0}(\N)$, supports a universal operator. This is the case for all spaces $\LL{p}(\Omega ,\mu )$, where $(\Omega ,\mu )$ is a $\sigma $-finite measured space.

\par\smallskip 

If $A$ is a bounded operator on a complex infinite dimensional separable Hilbert 
space $H$, it is known (see \cite{BG1}, or \cite[Ch.\,5]{BM}) that $A$
admits an invariant measure with respect to which it is ergodic, and which 
additionally has full support and admits a second order moment, if and 
only if its unimodular eigenvectors are \emph{perfectly spanning}: this 
means that  there exists a continuous probability measure $\sigma $ on the 
unit circle $\T=\{\lambda \in\C,\ |\lambda |=1\}$ such that for any Borel 
subset $B$ of $\T$ with $\sigma (B)=1$,
\[
\overline{\vphantom{(}\textrm{span}}\, \bigl[\,\ker(A-\lambda 
),\ \lambda \in B \bigr]=H.
\]
 In this case, the unimodular eigenvectors of $A$ are said to be 
\emph{$\sigma $-spanning}. An eigenvectorfield $E$ of $A$ is a map $E:\T\longrightarrow Z$ such that $A\,E(\lambda )=\lambda \,E(\lambda )$ for every $\lambda \in\T$. We will often be dealing in the rest of the paper with eigenvectorfields $E$ belonging to $\LL{2}(\T,\sigma  ;Z)$ where $\sigma  $ is a certain probability measure on $\T$: this means that $E:\T\longrightarrow Z$ is $\sigma  $-measurable, with 
$$\displaystyle\int_{\T}||E(\lambda 
)||^{2}d\sigma  (\lambda )<+\infty $$ and $A\,E(\lambda )=\lambda \,E(\lambda )$ 
$\sigma   $-almost everywhere. When we write simply that $E$ belongs to $\LL{2}(\T ;Z)$, this means that $\sigma  $ is assumed to be the normalized Lebesgue measure $d\lambda $ on $\T$.
\par\smallskip 
As a universal operator on $H$ is necessarily ergodic with respect to a 
certain invariant measure with full support (although this measure is not 
required to have a second-order moment), it is natural to look for 
conditions involving the unimodular eigenvectors of the operator $A$ which 
imply its universality. This is done in Section \ref{Sec4}, where we prove the two following general results:

\begin{theorem}\label{Theo5}
 Let $A$ be a bounded operator on a complex Banach space $Z$. Suppose that there exists an eigenvectorfield $E\in\LL{2}(\T;Z)$ for $A$ such that

\begin{enumerate}
 \item [\emph{(i)}] whenever $B$ is a Borel subset of $\T$ of full Lebesgue measure, $\overline{\vphantom{(}\textrm{span}}\,[E(\lambda ),
 \ \lambda \in B]=Z$;
 
 \item[\emph{(ii)}] there exists a non-zero functional $z_{0}^{*}\in 
Z^{*}$ and a trigonometric polynomial $p$ such that 
$\pss{z_{0}^{*}}{E(\lambda )}=p(\lambda )$ almost
everywhere on $\T$;

\item[\emph{(iii)}] if we set for every $n\in \Z$ 
$$\widehat{E}(n)=\displaystyle\int_{\T}\lambda ^{-n}E(\lambda )\,d\lambda 
,$$ then the series $\sum_{n\in\Z}\widehat{E}(n)$ is unconditionally 
convergent.
\end{enumerate}
Then the operator $A$ is universal for invertible ergodic systems. If 
moreover $\widehat{E}(-r)=0$ for some integer $r\in\Z$ (so that $\widehat{E}(-n)=0$ for every $n\ge 
r$), then $A$ is universal for ergodic systems.
\end{theorem}

Remark that if $E$ is sufficiently smooth (of class $\mathcal{C}^{1}$ for 
instance), then the sequence of Fourier coefficients 
$(\widehat{E}(n))_{n\in\Z}$ goes to zero sufficiently rapidly for the 
series $\sum_{n\in\Z}||\widehat{E}(n)||$ to be convergent. Hence the 
assumption (iii) is automatically satisfied in this case. If $E$ is 
analytic in a neighborhood of $\T$, it can be renormalized in such a way that assumption (ii) is also 
satisfied, and this yields

\begin{theorem}\label{Theo6}
 Suppose that $A\in \bb(Z)$ admits an eigenvectorfield $E$ which is 
analytic 
in a neighborhood of $\T$, and that 
$\overline{\vphantom{(}\emph{span}}\,\bigl[E(\lambda );\ \lambda \in\T 
\bigr]=Z$. Then $A$ is universal for invertible ergodic systems. If $E$ is 
analytic in a neighborhood of the closed unit disk 
$\overline{\vphantom{(}\D}$ and 
$\overline{\vphantom{(}\emph{span}}\,\bigl[E(\lambda );\ \lambda \in\T 
\bigr]=Z$, then $A$ is universal for ergodic systems.
\end{theorem}

Several applications of these two theorems are given in Section 
\ref{Sec4}, in particular to adjoints of multipliers on $H^{2}(\D)$ 
(Example \ref{Example2}) and to the rather unexpected case of a 
Kalish-type operator on $\LL{2}(\T)$ (Example \ref{Example4}).
\par\smallskip 
In Section \ref{Sec5} we try to exhibit some necessary conditions for an 
operator to be universal. In this generality, and with the present 
definition of universality, this seems to be delicate. But if we restrict 
ourselves to operators acting on a Hilbert space, and if we additionally 
require in the definition of universality that the measure $\nu $ admits a 
moment of order $2$ (see Definition \ref{Definition10}), then we obtain:

\begin{theorem}{\label{Theo7}}
 Suppose that $A\in \bb(H)$ is universal for ergodic systems in the 
modified 
sense presented above. Then the unimodular eigenvalues of $A$ form a 
subset of $\T$ of Lebesgue measure 1.
\end{theorem}

It is a rather puzzling fact that we do not know whether Theorem 
\ref{Theo7} can be extended to operators which are universal for 
invertible ergodic systems only. This point is discussed in Section 
\ref{Sec5}, as well as some open questions. 

\par\bigskip
\textbf{Acknowledgement:} I am grateful to the referee for his/her careful reading of the manuscript and his/her suggestions which enabled me to clarify some points in the presentation of the text and to simplify some arguments.

\section{A criterion for universality: proofs of Theorems \ref{Theo1} and 
\ref{Theo1bis}}\label{Sec2}

The proofs of Theorems \ref{Theo1} and \ref{Theo1bis} being very similar, 
we concentrate on the proof of Theorem \ref{Theo1}, and will indicate briefly afterwards the modifications needed for proving Theorem \ref{Theo1bis}.

\subsection{General pattern of the proof of Theorem 
\ref{Theo1}}\label{SousSec2a}
Let $A$ be a bounded operator on the infinite-dimensional 
separable Banach space $Z$ satisfying the assumptions of Theorem \ref{Theo1}. We will suppose in the rest of the proof that $Z$ is a complex Banach space, but the proof obviously holds true for real spaces as well. Let $(X,\bb,\mu ;T)$ be an invertible ergodic dynamical system on a standard probability space. For any complex-valued function $f\in\LL{\infty }(X,\bb,\mu )$ let $\Phi _{f}$ be the map from 
$(X,\bb,\mu  )$ into $Z$ defined by setting \[
\Phi _{f}(x)=\sum_{\kin}f(\expo{T}{k}x)\,\expo{A}{-k}z_{0}.
\]
Since $f$  is essentially bounded and the series 
$\sum_{\kin}\expo{A}{-k}z_{0}$ is unconditionally convergent, $\Phi _{f}(x)$ is well-defined for $\mu $-almost every $x\in 
X$. Also we have 
\[
\Phi_{f}(Tx)=\sum_{\kin}f(\expo{T}{k+1}x)\,\expo{A}{-k}z_{0}
=\sum_{\kin}f(\expo{T}{k}x)\,\expo{
A}{-k+1}z_{0}=A\,\Phi _{f}(x)
\]
since $\expo{A}{-k+1}z_{0}=z_{-k+1}=Az_{-k}=A.\expo{A}{-k}z_{0}$.
\par\smallskip 
If we denote by $\bb_{Z}$ the Borel $\sigma $-field of $Z$,
and by
$\nu _{f}$ the Borel probability measure on $Z$ which is 
the image of $\mu $ under the map $\Phi _{f}$ (i.e.\ $\nu _{f}(B)=\mu (\Phi 
_{f}^{-1}(B))$ for every Borel subset $B$ of $Z$), it then follows that $\Phi _{f}:(X,\bb,\mu ;T)\longrightarrow(Z,\bb_{Z},\nu _{f};A)$ is a factor map. This first argument is rather similar to the one 
employed in \cite{GW}, but the map $\Phi _{f}$ is defined differently in 
\cite{GW}, and for any $f\in\LL{4}(X,\bb,\mu )$, thanks to the ergodic theorem of \cite{JRT}. The goal in \cite{GW} is then to construct a function $f\in\LL{4}(X,\bb,\mu )$ as a limit of certain finitely-valued functions $f_{n}\in\LL{\infty}(X,\bb,\mu )$, in such a way that $\Phi 
_{f}$ becomes an isomorphism of dynamical systems and the measure $\nu 
_{f}$ has full support. As $\Phi _{f}$ is not necessarily well-defined here when $f\in\LL{4}(X,\bb,\mu )$ (or $f\in\LL{p}(X,\bb,\mu )$, $1<p<+\infty$), we will construct, in the same spirit as in \cite {GW}, a sequence of finitely-valued functions $f_{n}\in\LL{\infty }(X,\bb,\mu )$ such that $\Phi _{f_{n}}$ converges in $\LL{2}(X,\bb,\mu;Z )$ to a certain function $\Phi \in\LL{2}(X,\bb,\mu ;Z)$ 
which will be an isomorphism between the two systems $(X,\bb,\mu ;T)$ and 
$(Z,\bb_{Z},\nu ;A)$, where  
$\nu $ is the image of $\mu $ under the map $\Phi $.
\par\smallskip 
Let $(Q_{j})_{j\ge 0}$ be a sequence of Borel subsets of $X$ which is dense in 
$(\bb,\mu )$ (i.e.\ for every $\varepsilon>0$  and every $B\in\bb$, there exists a $j\ge 0$ such that $\mu 
(Q_{j}\vartriangle B )<\varepsilon $) with $Q_{0}=X$. Moreover, we suppose that for any $i\ge 0$, the set $J_{i}=\{j\ge i\; ;\; Q_{j}=Q_{i}\}$ is infinite. Since the span of the vectors $\expo{A}{k}z_{0}$, $\kin$, is dense in $Z$, there exists a sequence $(u_{n})_{n\ge 1}$ of vectors of $Z$ of the form \[
u_{n}=\sum_{|k|\le d_{n}}\expo{a}{(n)}_{k}\,\expo{A}{-k}z_{0},\quad 
\expo{a}{(n)}_{k}\in\C,\quad \max_{|k|\le d_{n}}|\expo{a}{(n)}_{k}|>0
\]
which is dense in $Z$. Let, for each $n\ge 1$, $r_{n}$ be a positive 
number such that  the open balls $U_n=B(u_{n},r_{n})$ centered at $u_{n}$ and of radius $r_{n}$ form a basis of the topology of $Z$. We set 
$U_{0}=Z$. 
Lastly, by assumption (b) of Theorem \ref{Theo1} there exists a finite 
subset $F$ of $\Z$ and a non-zero functional $z_{0}^{*}\in Z^{*}$ such that $\pss{z_{0}^{*}}{\expo{A}{-n}z_{0}}=0$ for all $n\in\Z\setminus F$ and  $\pss{z_{0}^{*}}{\expo{A}{-n}z_{0}}\neq 0$ for all $n\in F$ (we may have to modify the initial set $F$ to obtain this property). 
If we replace the vector $z_{0}$ by the vector $z_{0}'=A^{-p}z_{0}$, then $\pss{z_{0}^{*}}{\expo{A}{-n}z_{0}'}\neq 0$ if and only if $n\in F'=F-p$. If we choose $p\in F$, $\pss{z_{0}^{*}}{z_{0}'}\neq 0$.
So, replacing $z_{0}$ by $z_{0}'$ and $F$ by $F'$,
we can suppose that $0\in F$ and that $z_{0}$ and $z_{0}^{*}$ are such that
$c_{n}=\pss{z_{0}^{*}}{\expo{A}{-n}z_{0}}$ is non-zero \ifff\ 
$n\in F$. We let $d=\max |F|$.

\subsection{Construction of the functions $f_{n}$, $n\ge 
0$}\label{SousSec2b}
We are now ready to start the construction of the functions $f_{n}$. This 
construction is very much inspired from that of \cite{GW}, but many technical details need to be adjusted to the present situation. For any $z\in\C$ and $r>0$, $D(z,r)$ denotes the open disk centered at 
$z$ of radius $r$.
\par\smallskip 
We construct by induction
\begin{enumerate}
 \item [--] a sequence $(f_{n})_{n\ge 0}$ of functions of 
$\LL{\infty }(X,\bb,\mu ;\C  )$;

\item[--] sequences $(\alpha  
_{n})_{n\ge 0}$, $(\beta  _{n})_{n\ge 0}$, $(\gamma  _{n})_{n\ge 0}$, $(\delta  _{n})_{n\ge 0}$, and $(\eta  _{n})_{n\ge 0}$ of positive real numbers, decreasing to zero extremely fast;

\item[--] for each $n\ge 0$, families $(D_{i,\,0}^{(n)})_{0\le i\le n}$ and $(D_{i,\,1}^{(n)})_{0\le i\le n}$, 
$(E_{i,\,0}^{(n)})_{0\le i\le n}$ and $(E_{i,\,1}^{(n)})_{0\le i\le n}$, $(F_{i,\,0}^{(n)})_{0\le i\le n}$ and $(F_{i,\,1}^{(n)})_{0\le i\le n}$ of Borel subsets of $\C$;

\item[--] for each $n\ge 0$, families $(G_{i}^{(n)})_{0\le i\le n}$ and 
$(H_{i}^{(n)})_{0\le i\le 
n}$ 
of $\mu $-measurable subsets of $X$, and two measurable subsets $B_{n}$ and 
$C_{n}$ of $X$
\end{enumerate}

such that
\begin{enumerate}
\par\smallskip
\item [(1)] the sets $E_{i\,, 0}^{(n)}$ and $E_{i\,, 1}^{(n)}$, $0\le i\le n$,
are finite, 
and 
the range of $f_{n}$ is equal to
\[ 
\left(\bigcup_{i=0}^{n}E_{i,\,0}^{(n)}\right)\cup\left(\bigcup_{i=0}^{n}E_{
i,\,1}^{(n)}
\right)\!;\label{Property1}
\]
moreover, for every $i\in\{0,\dots, n\}$, $E_{i,\,0}^{(n)}\cap 
E_{i,\,1}^{(n)}=\varnothing$;

\par\smallskip
\item[(2)] we have \label{Property2}
\par\smallskip
\begin{enumerate}
\par\smallskip
\item[(2a)] \label{Property2a}$\mu (C_{n})>1-\eta _{n}$;
\par\smallskip
\item[(2b)] \label{Property2b}
if we set, for every $i\in\{0,\dots, n\}$

\begin{align*}
 D_{i,\,0}^{(n)}&=\Bigl\{\sum_{p\in F}c_{p}\,f_{n}(\expo{T}{p}x);\ 
  x\in 
C_{n}\ \textrm{and}\ f_{n}(x)\in E_{i,\,0}^{(n)}\Bigr\}
\end{align*}
and
\begin{align*}
 D_{i,\,1}^{(n)}&=\Bigl\{\sum_{p\in F}c_{p}\,f_{n}(\expo{T}{p}x);\ 
  x\in 
C_{n}\ \textrm{and}\ f_{n}(x)\in E_{i,\,1}^{(n)}\Bigr\},
\end{align*}

then 
$D_{i,\,0}^{(n)}\cap D_{i,\,1}^{(n)}=\varnothing$;

\par\smallskip
\item[(2c)] \label{Property2c}
if we set, for every $i\in\{0,\dots, n\}$
\[
F_{i,\,0}^{(n)}=D_{i,\, 0}^{(n)}+D(0,\beta _{n})\quad\textrm{and}\quad 
F_{i,\,1}^{(n)}=D_{i,\, 1}^{(n)}+D(0,\beta _{n}),
\]
then $F_{i,\,0}^{(n)}\cap F_{i,\,1}^{(n)}=\varnothing$;
\end{enumerate}
\par\smallskip
\item[(3)] \label{Property3} for every $i\in\{0,\dots, n\}$,
\par\smallskip
\begin{enumerate}

\par\smallskip
\item[(3a)] \label{Property3a}$\mu (H_{i}^{(n)})<\alpha _{i}(1-2^{-n})$;
\par\smallskip
\item[(3b)] \label{Property3b}$H_{i}^{(n-1)}\subseteq H_{i}^{(n)}$ for 
every 
$i\in\{0,\dots, n-1\}$;
\par\smallskip
\item[(3c)] \label{Property3c} for every $x\in Q_{i}\setminus H_{i}^{(n)}$, 
$f_{n}(x)\in E_{i,\,0}^{(n)}$, 

and 

for every $x\in (X\setminus Q_{i})\setminus H_{i}^{(n)} $, 
$f_{n}(x)\in 
E_{i,\,1}^{(n)}$;
\end{enumerate}
\par\smallskip
\item[(4)] \label{Property4} we have
\par\smallskip
\begin{enumerate}
\par\smallskip
\item[(4a)] \label{Property4a}$\mu (B_{n})>1-\eta _{n}$;
\par\smallskip
\item[(4b)] \label{Property4b}for every $x\in B_{n}$,
$|f_{n}(x)-f_{n-1}(x)|<\gamma 
_{n}$;
\par\smallskip
\item[(4c)] \label{Property4c}for every $x\in B_{n}$, $||\Phi 
_{f_{n}}(x)-\Phi 
_{f_{n-1}}(x)||<\gamma _{n}$;
\end{enumerate}
\par\smallskip
\item[(5)] \label{Property5}we have
\par\smallskip
\begin{enumerate}
\item[(5a)]\label{Property5a} $ ||f_{n}-f_{n-1}||_{\LL{2}(X,\bb,\mu 
)}<2^{-n}$;
\par\smallskip
\item[(5b)] \label{Property5b}$||\Phi _{f_{n}}-\Phi 
_{f_{n-1}}||_{\LL{2}(X,\bb,\mu;Z 
)}<2^{-n}$;
\end{enumerate}
\par\smallskip
\item[(6)] \label{Property6}for every $i\in\{0,\dots,n\}$,
\par\smallskip
\begin{enumerate}
\item[(6a)] \label{Property6a}$\mu (G_{i}^{(n)})\ge \delta _{i}(1+2^{-n})$;
\par\smallskip
\item[(6b)] \label{Property6b}$G_{i}^{(n)}\subseteq G_{i}^{(n-1)}$ for 
every 
$i\in\{0,\dots,n-1\}$;
\par\smallskip
\item[(6c)] \label{Property6c}$\Phi _{f_{n}}(x)\in U_{i}$ for every $x\in 
G_{i}^{(n)}$.
\end{enumerate}
\end{enumerate}
\par\medskip
We start the construction by setting (recall that $Q_{0}=X$ and $U_{0}=Z$):
$E_{0,\,0}^{(0)}=\{0\}$,  $E_{0,\, 1}^{(0)}=\varnothing$, $B_{0}=C_{0}=X$,
$\alpha _{0}=\beta _{0}=\gamma _{0}=\delta _{0}=\eta 
_{0}=1/8$, $G_{0}^{(0)}=X$, $H_{0}^{(0)}=\varnothing$ and 
$f_{0}=0$.
\par\smallskip
Suppose now that the construction has been carried out until step $n$. At 
step $n+1$, we start by introducing
\par\smallskip
\begin{enumerate}
 \item [--] an integer $N\ge 1$, which will be chosen very large at the end of the construction;
\item[--] two positive numbers $\eta $ and $\gamma $, independent 
of each other, which will be chosen very small at the end of the construction.
\end{enumerate}
\par\smallskip
As $T$ is invertible and ergodic, there exists a measurable subset $E$ of 
$X$ such that
 $\mu (E)>0$, $\mu \left(\bigcup_{|k|\le N}\expo{T}{k}E\right)<\eta $ and the sets $\expo{T}{k}E$, $|k|\le N$, are pairwise 
disjoint.
\par\smallskip
Recall that 
\[
u_{n+1}=\sum_{|k|\le d_{n+1}}a_{k}^{\np 
}\,\expo{A}{-k}z_{0}\quad\textrm{and}\quad 
U_{n+1}=B(u_{n+1},r_{n+1}).
\]
We suppose that $N\ge d_{n+1}$.
\par\smallskip 
\textbf{Step 1:}
We first define an auxiliary function $g_{n+1}$ on $X$ in the following 
way:
\[
g_{n+1}(x)=
\begin{cases}
 a_{k}^{\np }&\quad\textrm{if}\ x\in\expo{T}{k}E,\quad |k|\le d_{n+1}\\
 0&\quad\textrm{if}\ x\in\expo{T}{k}E,\quad d_{n+1}<|k|\le N\\
 f_{n}(x)&\quad\textrm{if}\ x\not\in  \bigcup\limits_{|k|\le 
N}\expo{T}{k}E.
\end{cases}
\]
The function $g_{n+1}$ thus defined is finite-valued and it coincides with 
$f_{n}$ on 
the 
set $$B=X\setminus\bigcup\limits_{|k|\le N}\expo{T}{k}E,$$ which has $\mu $-measure 
larger than 
$1-\eta $. The range of $g_{n+1}$ is equal to 
\[
\textrm{Ran}(f_{n}\restriction_{B})\cup\{0\}\cup\{a_{k}^{\np };\ |k|\le d_{n+1}\}, 
\]
and we write this 
finite set as $\{c_{l}^{\np };\ 0\le l\le l_{n+1}\}$, with all numbers 
$c_{l}^{\np }$ 
distinct.
\par\smallskip 
By the Rokhlin Lemma, we can choose a subset
$E'\in\bb$ of $X$  and an integer $M\ge d$ 
such 
that the sets $\expo{T}{k}E'$, $|k|\le M$, are pairwise disjoint, and 
$$\mu \bigl(\bigcup_{|k|\le M-d}\expo{T}{k}E' \bigr)>1-\eta .$$
\par\smallskip 
\textbf{Step 2:}
We state and prove in this step a simple abstract lemma, which will be used in the forthcoming Steps 3 and 4 in order to approximate certain finite families of scalars (like the family $(c_{l}^{(n+1)})_{0\le l\le l_{n+1}}$) by other families of scalars with further additional properties.

\begin{lemma}\label{Lemme0}
 Let $r\ge 1$ and let $\mkern 1.5 mu\pmb{d}=(d_{1},\ldots, d_{r})$ be an $r$-tuple of positive integers. We denote by $E_{\mkern 1.5 mu\pmb{d}}$ the subset of $\Z^{r}$ defined by 
 $$E_{\mkern 1.5 mu\pmb{d}}=\{\mkern 1.5 mu\pmb{u}=(u_{1},\ldots, u_{r});\ 0\le u_{i}\le d_{i}\; \textrm{for every } i=1,\ldots, r\}.$$
 For every $i=1,\ldots, r$, let $\lambda _{i}$ be a map from $F$ into $\{0,\ldots, d_{i}\}$. 
We denote by $\mkern 1.5 mu\pmb{\lambda}$
  the $r$-tuple of maps 
$\mkern 1.5 mu\pmb{\lambda}=(\lambda_1,\ldots, \lambda_{r})$ from $F$ into
 $\{0,\ldots, d_{1}\}\times\ldots\times \{0,\ldots, d_{r}\}$. For any such $\mkern 1.5 mu\pmb{\lambda }$, let $\sigma  _{\mkern 1.5 mu\pmb{\lambda }}$ be the functional on the vector space of functions from $E_{\mkern 1.5 mu\pmb{d}}$ into $\C$, identified with $\C^{\#E_{\mkern 1.5 mu\pmb{d}}}$, defined by
 \begin{align*}
 \sigma _{\mkern 1.5 mu\pmb{\lambda } }\,:\qquad \C^{\#E_{\mkern 1.5 mu\pmb{d}}}&\longrightarrow\C\\ \ \ \bigl( \gamma _{\mkern 1.5 mu\pmb{u}}\bigr)_{\mkern 1.5 mu\pmb{u }\in E_{\mkern 1.5 mu\pmb{d }}}&\longmapsto \sum_{p\in F}c_{p}\, \gamma _{\mkern 1.5 mu\pmb{\lambda }(p)}.
\end{align*}
There exists a dense subset of $\C^{\#E_{\mkern 1.5 mu\pmb{d}}}$ consisting of elements $(\gamma _{\mkern 1.5 mu\pmb{u }})_{\mkern 1.5 mu\pmb{u }\in E_{\mkern 1.5 mu\pmb{d }}}$ with the following property: 
$$\sigma _{\mkern 1.5 mu\pmb{\lambda } }((\gamma _{\mkern 1.5 mu\pmb{u }})_{\mkern 1.5 mu\pmb{u }\in E_{\mkern 1.5 mu\pmb{d }}})\not = \sigma _{\mkern 1.5 mu\pmb{\lambda' } }((\gamma _{\mkern 1.5 mu\pmb{u }})_{\mkern 1.5 mu\pmb{u }\in E_{\mkern 1.5 mu\pmb{d }}})$$
for every maps $\mkern 1.5 mu\pmb{\lambda }$ and $\mkern 1.5 mu\pmb{\lambda' }$ such that $\mkern 1.5 mu\pmb{\lambda }(F)\not = \mkern 1.5 mu\pmb{\lambda' }(F)$.
\end{lemma}

\begin{proof}
Let us first observe that if  $\mkern 1.5 mu\pmb{\lambda }(F)\not = \mkern 1.5 mu\pmb{\lambda' }(F)$, $\sigma _{\mkern 1.5 mu\pmb{\lambda }}\not = \sigma _{\mkern 1.5 mu\pmb{\lambda' }}$. This follows from the fact that all coefficients $c_{p}$, $p\in F$, are distinct. Let then
$$\mkern 1.5 mu\pmb{\Sigma  }=\{(\mkern 1.5 mu\pmb{\lambda },\mkern 1.5 mu\pmb{\lambda' });\ \mkern 1.5 mu\pmb{\lambda }(F)\not = \mkern 1.5 mu\pmb{\lambda' }(F)\}.$$
For each $(\mkern 1.5 mu\pmb{\lambda },\mkern 1.5 mu\pmb{\lambda' })\in\mkern 1.5 mu\pmb{\Sigma  }$, the kernel $\ker(\sigma _{\mkern 1.5 mu\pmb{\lambda }}-\sigma _{\mkern 1.5 mu\pmb{\lambda' }})$ is different from the whole space $\C^{\#E_{\mkern 1.5 mu\pmb{d}}}$. The set $\mkern 1.5 mu\pmb{\Sigma  }$ being finite, the Baire Category Theorem yields that
$$\bigcup_{(\mkern 1.5 mu\pmb{\lambda },\mkern 1.5 mu\pmb{\lambda' })\in\mkern 1.5 mu\pmb{\Sigma  }}(\sigma _{\mkern 1.5 mu\pmb{\lambda }}-\sigma _{\mkern 1.5 mu\pmb{\lambda' }})^{-1}(\C^{*})$$ is dense in $\C^{\#E_{\mkern 1.5 mu\pmb{d}}}$, which proves our claim.
\end{proof}

\par\smallskip 
\textbf{Step 3:}
We define a second auxiliary function $h_{n+1}$ on $X$ by setting
\[
h_{n+1}(x)=
\begin{cases}
 c_{l,\,k}^{\np }&\quad\textrm{if}\ g_{n+1}(x)=c_{l}^{\np }\ \textrm{and}\ 
 x\in\expo{T}{k}E'\ \textrm{for some}\ |k|\le M,\\
 c_{l}^{\np }&\quad\textrm{if}\ g_{n+1}(x)=c_{l}^{\np }\ \textrm{and}\ 
 x\not\in\bigcup_{ |k|\le M}\expo{T}{k}E',
\end{cases}
\]
where for every $l\in\{0,\dots,l_{n+1}\}$, $c_{l,\,k}^{(n)}$ is so close 
to 
$c_{l}^{\np }$ for each $ |k|\le M$ that 
\[
||h_{n+1}-g_{n+1}||_{\infty }\le \frac{\gamma }{2},
\]
 all the numbers $c_{l,\,k}^{\np }$, $l\in\{0,\dots,l_{n+1}\}$, $ 
|k|\le 
M$, and 
$ c_{l}^{\np }$, $l\in\{0,\dots,l_{n+1}\}$, are distinct,
and, moreover, the numbers $c_{l,\,k}^{(n+1)}$, $l\in\{0,\dots,l_{n+1}\}$, $|k|\le M$, have the following property:
\par\smallskip
whenever $\tau ,\,\tau '$ are two maps from $F$ into $\{0,\dots,l_{n+1}\}$, and $k,\, k'$ are two integers with $|k|,\,|k'|\le M-d$,  we have
\[
\sum_{p\in F}c_{p}\,c_{\tau (p),\,k+p}^{(n+1)}\neq\sum_{p\in F}c_{p}\,c_{\tau '(p),\,k'+p}^{(n+1)}
\]
as soon as there exists a $p\in F$ such that  $(\tau (p),k+p)\neq(\tau '(p),k'+p)$.
\par\smallskip
Observe that since $|k|,\,|k'|\le M-d$ and $d=\max |F|$, $|k+p|,\,|k'+p|\le M$ for every $p\in F$, so that the quantities $c_{\tau (p),\,k+p}^{(n+1)}$ and $c_{\tau '(p),\,k'+p}^{(n+1)}$ in the expression above are well-defined.
\par\smallskip
That the scalars $c_{l,\,k}^{(n+1)} $ can indeed be chosen so as to satisfy these properties is a consequence of Lemma \ref{Lemme0}. Denote by ${\Sigma  }$ the set of all $4$-tuples $(\tau,\tau',k,k')$, where $\tau,\tau'$ are maps from $F$ into $\{0,\ldots, l_{n+1}\}$ and $k,k'$ are integers with $|k|,\,|k'|\le M-d$, such that there exists a $p\in F$ with $(\tau (p),k+p)\neq(\tau '(p),k'+p)$. For any map $\tau : F \To\{0,\dots, l_{n+1}\}$ and any integer $k$ with $|k|\le M-d$, let 
\begin{align*}
 \mkern 1.5 mu\pmb{\lambda  } _{\tau ,\,k}\,:\qquad F &\longrightarrow
 \{0,\dots, l_{n+1}\}\times \{-M,\dots, M\}.\\ \ \ p&\longmapsto (\tau(p), p+k)
\end{align*}
Let us check that if $(\tau,\tau',k,k')$ belongs to 
${\Sigma  }$, $\mkern 1.5 mu\pmb{\lambda  } _{\tau ,\,k}(F)\not =\mkern 1.5 mu\pmb{\lambda  } _{\tau' ,\,k'}(F)$.
If $\mkern 1.5 mu\pmb{\lambda  } _{\tau ,\,k}(F) =\mkern 1.5 mu\pmb{\lambda  } _{\tau' ,\,k'}(F)$, then 
$$\bigl\{(\tau (p),\,k+p);\ p\in F\bigr\}=\bigl\{(\tau '(p), k'+p);\ p\in F\bigr\},$$ so that $k+F=k'+F$. As the set $F$ is finite, $k=k'$, and thus for every $p\in F$ there exists a $p'\in F$ such that $(\tau (p), k+p)=(\tau' (p'), k+p')$. So $p=p'$ and $\tau (p)=\tau' (p)$. Thus $(\tau(p), k+p)=(\tau' (p), k+p)$ for every $p\in F$,
which is
contrary to our assumption. So  $\mkern 1.5 mu\pmb{\lambda  } _{\tau ,\,k}(F)\not =\mkern 1.5 mu\pmb{\lambda  } _{\tau' ,\,k'}(F)$ as soon as $(\tau,\tau',k,k')$ belongs to 
${\Sigma  }$. 

Applying Lemma \ref{Lemme0}, it follows from the observation above that we can choose a family of scalars
 $\bigl( c_{l,\,k}^{(n+1)}\bigr)_{0\le l\le l_{n+1},\,|k|\le M}$ such that $$\bigl|c_{l,\,k}^{(n+1)}-c_{l}^{(n+1)}\bigr|<\frac{\gamma }{2}\quad \textrm{ for every } l\in\{0,\ldots, l_{n+1}\} \textrm{ and }|k|\le M,$$ all the numbers $c_{l,\,k}^{n+1}$ and $c_{l}^{(n+1)}$ are distinct, and 
\[
\bigl( \sigma _{\mkern 1.5 mu\pmb{\lambda  } _{\tau ,\,k}}-\sigma _{\mkern 1.5 mu\pmb{\lambda  } _{\tau' ,\,k'}}\bigr)\,\bigl(\bigl( c_{l,\,k}^{(n+1)}\bigr)_{0\le l\le l_{n+1},\,|k|\le M}\bigr)\neq 0\quad \textrm{for every}\ (\tau ,\,\tau',\,k,\,k')\in {{\Sigma   }} ,
\] i.e. $$\displaystyle \sum_{p\in F}c_{p}\,c_{\tau (p),\,k+p}^{(n+1)}\neq\sum_{p\in F}c_{p}\,c_{\tau '(p),\,k'+p}^{(n+1)}
\quad \textrm{for every}\ (\tau ,\,\tau ',\,k,\,k')\in {{\Sigma   }} .$$
\par\smallskip 
Now the funtion $h_{n+1}$ has been defined, we observe that it is finite-valued, and we write its range as $$\{b_{j}^{\np };\ 0\le j\le 
j_{n+1}\}$$
where all the numbers $b_{j}^{\np }$ are distinct.
We also set $$C_{n+1}=\bigcup_{|k|\le M-d }\expo{T}{k}E'.$$ By our assumptions on $M$ and $E'$, $\mu 
(C_{n+1})>1-\eta$.
\par\smallskip 
\textbf{Step 4:}
We construct in this step complex numbers $b_{j,\,0}^{(n+1)}$ and $b_{j,\,1}^{(n+1)}$, 
$0\le j\le j_{n+1}$, which are such that
all the numbers $b_{j,\,0}^{(n+1)}$ and $b_{j,\,1}^{(n+1)}$ are distinct, and both $b_{j,\,0}^{(n+1)}$ and $b_{j,\,1}^{(n+1)}$ are so close to $b_{j}^{(n+1)}$ for each $j\in\{0,\dots, j_{n+1}\}$ that 
\[
\sup_{j\in\{0,\dots, j_{n+1}\}}\Bigl(\, |b_{j,\,0}^{(n+1)}-b_{j}^{(n+1)}|+|b_{j,\,1}^{(n+1)}-b_{j}^{(n+1)}|\,\Bigr)<\dfrac{\gamma }{2}\cdot
\]
Moreover, if $b_{j}^{(n+1)}=c_{l,\,k}^{(n+1)}$ for some $l\in\{0,\dots,l_{n+1}\}$ and $|k|\le M$, we write $$b_{j,\,0}^{(n+1)}=c_{l,\,k,\,0}^{(n+1)}\quad \textrm{ and }\quad b_{j,\,1}^{(n+1)}=c_{l,\,k,\,1}^{(n+1)},$$ and we require that the following holds true: 
\par\smallskip
for any maps $\theta ,\,\theta '\,:\ F\longrightarrow \{0,1\}$,  $\tau ,\,\tau '\,:\ F\longrightarrow \{0,\dots,l_{n+1}\}$ and any integers $k,\,k'$ with $|k|,|k'|\le M-d$,  
\[
\sum_{p\in F}c_{p}\,c_{\tau (p),\,k+p,\,\theta (p)}^{(n+1)}\neq\sum_{p\in F}c_{p}\,c_{\tau'(p),\,k'+p,\,\theta '(p)}^{(n+1)}
\]
as soon as there exists a $p\in F$ such that $(\tau (p),k+p,\theta (p))\neq (\tau' (p),k'+p,\theta '(p))$.
\par\medskip 
The proof of the existence of such numbers again relies on Lemma \ref{Lemme0}. 
Denote by $\mathcal{F}$ the set of all $6$-tuples $(\tau,\tau',k,k',\theta 
,\theta ')$, where $\tau,\tau'$ are maps from $F$ into $\{0,\ldots, l_{n+1}\}$,
$\theta ,\theta '$ maps from 
$F$ into $\{0,1\}$,  and $k,k'$ integers with $|k|,|k'|\le M-d$, such that there 
exists a $p\in F$ with $(\tau (p),k+p,\theta (p))\neq (\tau' (p),k'+p,\theta 
'(p))$. For any maps $\tau \,:\ F\longrightarrow \{0,\ldots, l_{n+1}\}$, 
$\theta \,:\ F\longrightarrow \{0,1\}$,
 and any integer $k$ with $|k|\le M-d$, let
\begin{align*}
 \mkern 1.5 mu\pmb{\lambda}  _{\tau ,\,k,\,\theta }\,:\qquad F&\longrightarrow \{0,\ldots, l_{n+1}\}\times \{-M,\ldots, M\}\times\{0,1\}\\ \ \ p&\longmapsto (\tau(p), k+p,\theta (p)).
\end{align*}
We claim that if $(\tau ,\,\tau ',\,k,\,k',\,\theta ,\,\theta ')$ belongs to 
$\mathcal{F}$, then $\mkern 1.5 mu\pmb{\lambda}  _{\tau ,\,k,\,\theta }(F)\not = 
\mkern 1.5 mu\pmb{\lambda}  _{\tau' ,\,k',\,\theta' }(F)$. Indeed, if these two 
sets were were equal, we would have 
\[
\bigl\{\bigl( \tau (p),\,k+p,\,\theta (p)\bigr);\ p\in F\bigr\}=\bigl\{\bigl( \tau '(p),\,k'+p,\,\theta' (p)\bigr);\ p\in F\bigr\}\cdot 
\]
Hence $k+F=k'+F$, so that $k=k'$. Thus for every $p\in F$ there exists $p'\in F$ such that $\bigl( \tau (p),\,k+p,\,\theta (p)\bigr)=\bigl( \tau '(p'),\,k+p',\,\theta '(p')\bigr)$. Necessarily, $p=p'$, so that $\tau (p)=\tau '(p)$ and $\theta (p)=\theta '(p)$. Hence
$\bigl( \tau (p),\,k+p,\,\theta (p)\bigr)=\bigl( \tau '(p),\,k+p,\,\theta '(p)\bigr)$ for every $p\in F$, and this contradicts our initial assumption. So $\mkern 1.5 mu\pmb{\lambda}  _{\tau ,\,k,\,\theta }(F)\not = \mkern 1.5 mu\pmb{\lambda}  _{\tau' ,\,k',\,\theta' }(F)$. It thus follows from Lemma \ref{Lemme0} that 
numbers $c_{l,\,k,\,0}^{(n+1)}$ and $c_{l,\,k,\,1}^{(n+1)}$ can be chosen as close to $c_{l,\,k}^{(n+1)}$ as we wish, all distinct, and such that 
\[
\sum_{p\in F}c_{p}\,c_{\tau (p),\,k+p,\,\theta (p)}^{(n+1)}\neq\sum_{p\in F}c_{p}\,c_{\tau '(p),\,k'+p,\,\theta '(p)}^{(n+1)}\quad \textrm{ for every } (\tau ,\,\tau ',\,k,\,k',\,\theta ,\,\theta ')\in\mathcal{F}.
\]
\par\smallskip
This defines $b_{j,\,0}^{(n+1)}$ and $b_{j,\,1}^{(n+1)}$ when $b_{j}^{(n+1)}=c_{l,\,k}^{(n+1)}$ for some $l\in\{0,\dots,l_{n+1}\}$ and $|k|\le M$. It is then easy to define the numbers $b_{j,\,0}^{(n+1)}$ and $b_{j,\,1}^{(n+1)}$ for the remaining indices in such a way that they are sufficiently close to $b_{j}^{(n+1)}$, distinct, and distinct from all the numbers $c_{l,\,k,\,0}^{(n+1)}$ and $c_{l,\,k,\,1}^{(n+1)}$.
\par\smallskip
\textbf{Step 5:}
We can now define the function $f_{n+1}$ on $X$ by setting
\[
f_{n+1}(x)=
\begin{cases}
 b_{j,\,0}^{\np }&\quad\textrm{if}\ h_{n+1}(x)=b_{j}^{\np }\ \textrm{and}\ 
x\in
 Q_{n+1}\\
 b_{j,\,1}^{\np }&\quad\textrm{if}\ h_{n+1}(x)=b_{j}^{\np }\ \textrm{and}\ 
x\in 
 X\setminus Q_{n+1}. 
\end{cases}
\]
Obviously
$$||h_{n+1}-f_{n+1}||_{\infty }
<\frac{\gamma }{2}\cdot$$
If $x$ belongs to $C_{n+1}$, then there exists an integer $k$ with $|k|\le M-d$ such that $x\in T^{k}E'$. Hence $T^{p}x\in T^{p+k}E'$ for every $p\in F$, and $|k+p|\le M$. It follows that there exists a map $\tau \,:\ F \To\{0,\dots,l_{n+1}\}$ such that $$h_{n+1}(T^{p}x)=c_{\tau (p),\,k+p}^{(n+1)}\quad \textrm{ for every } p\in F.$$ By the definition of the function $f_{n+1}$, there exists a map $\theta :\ F\longrightarrow \{0,1\}$ such that $$f_{n+1}\bigl( T^{p}x\bigr)=c_{\tau (p),\,k+p,\,\theta (p)}^{(n+1)}\quad \textrm{ for every } p\in F.$$ This map $\theta $ satisfies $\theta (0)=0$ if $x\in Q_{n+1}$ and $\theta (0)=1$ if $x\in X\setminus Q_{n+1}$. We have 
\[
\sum_{p\in F} c_{p} f_{n+1}\bigl( T^{p}x\bigr)=\sum_{p\in F}c_{p}\,c_{\tau (p),\,k+p,\,\theta (p)}^{(n+1)}.
\]

\par\smallskip
\textbf{Step 6:}
For every $i\in\{0,\dots,n\}$, let
\begin{align*}
  J_{i,\,0}^{\np }=\Bigl\{j\in\{0,\dots,j_{n+1}\}\ \textrm{ ; } &\textrm{there exists}\ l\in\{0,\dots,l_{n+1}\}\textrm{ such that} \ c_{l}^{\np }\in E_{i,\,0}^{(n)}\ \\
 &\textrm{with 
either}\\
 &b_{j}^{\np }=c_{l}^{\np }\ \textrm{or}\ b_{j}^{\np }=c_{l,\,k}^{\np }\ 
 \textrm{for some}\ |k|\le M\Bigr\}\\[1ex]
  J_{i,\,1}^{\np }=\Bigl\{j\in\{0,\dots,j_{n+1}\}\ \textrm{ ; } &\textrm{there exists}\ l\in\{0,\dots,l_{n+1}\}\textrm{ such that}\ c_{l}^{\np }\in E_{i,\,1}^{(n)}\\\
 & \textrm{with 
either}\\
 &b_{j}^{\np }=c_{l}^{\np }\ \textrm{or}\ b_{j}^{\np }=c_{l,\,k}^{\np }\ 
 \textrm{for some}\ |k|\le M\Bigr\} .
\end{align*}

Set for $i\in\{0,\dots,n\}$

\begin{align*}
 E_{i,\,0}^{\np }&=\bigl\{b_{j,\,0}^{\np };\ j\in 
J_{i,0}^{\np }\bigr\}\cup\bigl\{b_{j,\,1}^{\np };\ j\in J_{i,0}^{\np 
}\bigr\},\\
E_{i,\,1}^{\np }&=\bigl\{b_{j,\,0}^{\np };\ j\in 
J_{i,1}^{\np }\bigr\}\cup\bigl\{b_{j,\,1}^{\np };\ j\in J_{i,1}^{\np 
}\bigr\},\\
E_{n+1,\,0}^{\np }&=\bigl\{b_{j,\,0}^{\np };\ j\in 
\{0,\dots,j_{n+1}\}\bigr\},\\
E_{n+1,\,1}^{\np }&=\bigl\{b_{j,\,1}^{\np };\ j\in 
\{0,\dots,j_{n+1}\}\bigr\}.
\end{align*}
\par\smallskip
\textbf{Step 7:} With these definitions, let us check that property (1) holds true. Of course, all the sets
$E_{i,\,0}^{\np }$ and $E_{i,\,1}^{\np }$ are finite and 
$$\textrm{ran}(f_{n+1})=\Bigl( 
\bigcup_{i=0}
^{n+1}E_{i,\,0}^{\np }\Bigr)\cup\Bigl( \bigcup_{i=0}^{n+1 }E_{i,\,1}^{\np 
}\Bigr)=E_{n+1,\,0}^{\np }\cup E_{n+1,\,1}^{\np }.$$
All the numbers $b_{j,\,0}^{\np }$ and $b_{j,\,1}^{\np }$, 
$j\in\{0,\dots,j_{n+1}\}$ are distinct, and for every index $i\in\en{n}$, $J_{i,\,0}^{\np }\cap 
J_{i,\,1}^{\np }=\varnothing$ (because $E_{i,\,0}^{(n)}\cap 
E_{i,\,1}^{(n)}=\varnothing$). So  
$E_{i,\,0}^{\np}\cap E_{i,\,1}^{\np}=\varnothing$ for all 
$i\in\{0,\dots,n\}$. Also 
clearly $E_{n+1,\,0}^{\np}\cap E_{n+1,\,1}^{\np}=\varnothing$. So property 
(1) holds true.
\par\smallskip 
\textbf{Step 8:}
In order to check property (2), let us fix
$i\in\{0,\dots,n+1\}$, 
and $x,\,y\in C_{n+1}$ such that $f_{n+1}(x)\in E_{i,\,0}^{\np}$ and
$f_{n+1}(y)\in E_{i,\,1}^{\np}$. 
By Step 5 above, there exist  maps $\tau ,\,\tau '\,:F\ \To\{0,\dots, l_{n+1}\}$, integers $k,\,k'$ with $|k|,\,|k'|\le M-d$, and maps $\theta ,\,\theta '\,:\ F\longrightarrow\{0,1\}$ such that 
\[
f_{n+1}\bigl( T^{p}x\bigr)=c_{\tau (p),\,k+p,\,\theta (p)}^{(n+1)}\quad\textrm{and}\quad
f_{n+1}\bigl( T^{p}y\bigr)=c_{\tau '(p),\,k'+p,\,\theta' (p)}^{(n+1)} \quad \textrm{for every } p\in F.
\]
Recall that $0\in F$.
  Since $E_{i,\,0}^{(n+1)}\cap E_{i,\,1}^{(n+1)}=\varnothing$, $f_{n+1}(x)\neq f_{n+1}(y)$, so that $(\tau (0),\,k,\,\theta (0))\neq(\tau '(0),\,k',\,\theta '(0))$. Hence 
$(\tau ,\,\tau ',\,k,\,k',\, \theta ,\,\theta ')$ belongs to $\mathcal{F}$, and 
\begin{align*}
\sum_{p\in F}c_{p}\,c_{\tau (p),\,k+p,\,\theta (p)}^{(n+1)}&\neq \sum_{p\in F}c_{p}\,c_{\tau '(p),\,k'+p,\,\theta' (p)}^{(n+1)}, 
\end{align*}
i.e.
\begin{align*}
\sum_{p\in F}c_{p}\,f_{n+1}\bigl( T^{p}x\bigr)&\neq \sum_{p\in F} c_{p}\,f_{n+1}\bigl( T^{p}y\bigr).
\end{align*}

Thus $D_{i,\,0}^{\np}\cap D_{i,\,1}^{\np}=\varnothing$ for every 
$i\in\{0,\dots,n+1\}$. 
Once $\eta _{n+1}$ is fixed (and this will be done only later on in the 
construction), 
one can choose $\eta <\eta _{n+1}$, and then
$\beta _{n+1}$ so 
small that properties (2a), (2b), and (2c) 
hold true.
\par\smallskip 
\textbf{Step 9:}
Our next step is to define the sets $H_{i}^{\np}$ for 
$i\in\en{n+1}$ and to 
prove property (3). We set 
\begin{align*}
 H_{i}^{\np}&=H_{i}^{(n)}\cup\Bigl(\bigcup_{|k|\le N}\expo{T}{k}E 
\Bigr)\quad \textrm{ for every } i\in\en{n}\\
 H_{n+1}^{\np}&=\varnothing.
\end{align*}
Then for every $i\in\en{n}$, $\mu (H_{i}^{\np})\le \mu (H_{i}^{(n)})+\eta 
< \alpha 
_{i}(1-2^{-n})+\eta $. So, if $\eta $ is chosen sufficiently small, $$\mu 
(H_{i}^{\np})<\alpha _{i}(1-2^{-\np}).$$ Also, $\mu 
(H_{n+1}^{\np})=0<\alpha 
_{n+1}
(1-2^{-\np})$ whatever the value of $\alpha _{n+1}$. So (3a) holds. As (3b) is obvious, it remains to check 
(3c).
\par\smallskip 
Let $i\in\en{n}$ and $x\in Q_{i}\setminus H_{i}^{\np}$. Then  $x\in X\setminus \bigl(\bigcup_{|k|\le N}\expo{T}{k}E 
\bigr)$ so that 
$$g_{n+1}(x)=f_{n}(x)=c_{l}^{\np} \textrm{ for some } l\in\en{l_{n+1}}.$$ Also $x\in Q_{i}\setminus H_{i}^{(n)}$, so $f_{n}(x)\in E_{i\,,0}^{(n)}$ by the induction assumption, that is  
$c_{l}^{\np}
\in E_{i,\,0}^{(n)}$. We have either
$$h_{n+1}(x)=c_{l}^{\np} \quad \textrm{or}\quad h_{n+1}(x)=c_{l,\,k}^{\np} \textrm{ for some }  |k|\le M.$$ If we write 
$h_{n+1}(x)=b_{j}^{\np}$ for some $j\in\en{j_{n+1}}$, then $j$ 
belongs 
to 
$J_{i,\,0}^{\np}$. So $b_{j,\,0}^{\np}$ and $b_{j,\,1}^{\np}$ belong to 
$E_{i,\,0}^{\np}$. Since $f_{n+1}(x)$ is equal to either $b_{j,\,0}^{\np}$ 
or 
$b_{j,\,1}^{\np}$, it follows that $f_{n+1}(x)$ belongs to 
$E_{i,\,0}^{\np}$. In the 
same way, if $x\in\bigl(X\setminus Q_{i} \bigr)\setminus H_{i}^{\np}$, 
then 
$f_{n+1}
(x)$ belongs to $E_{i,\,1}^{\np}$.
\par\smallskip 
Let now $i=n+1$. Let $x\in Q_{n+1}$, and let $j\in\en{j_{n+1}}$ be such 
that $h_{n+1}
(x)=b_{j}^{\np}$. Then $f_{n+1}(x)=b_{j,\,0}^{\np}$, and so by definition 
of the set 
$E_{n+1,\,0}^{\np}$, $f_{n+1}(x)$ belongs to $E_{n+1,\,0}^{\np}$. 
Similarly, if 
$x\in X\setminus Q_{n+1}$, then $f_{n+1}(x)$ belongs to 
$E_{n+1,\,1}^{\np}$. This 
proves property (3c).
\par\smallskip 
\textbf{Step 10:}
We now have to check properties (4) and (5). We 
have $g_{n+1}(x)=f_{n}(x)$ for every $x\in B$ and $\mu (B)>1-\eta$ (the set $B$ has been defined in Step 1). If 
$\gamma 
>0$ 
is an 
arbitrarily small positive number, the numbers $c_{l,\,k}^{\np}$, 
$b_{j,\,0}^{\np}$ and $b_{j,\,1}^{\np}$ have been chosen so close to 
$c_{l}^{\np}$ and $b_{j}^{\np}$ respectively that 
$||f_{n+1}-g_{n+1}||_{\infty 
}<\gamma $, so that in particular $$|f_{n+1}(x)-f_{n}(x)|<\gamma \quad \textrm{ for 
every }
x\in B.$$ Moreover,
\begin{eqnarray*}
 ||f_{n+1}-f_{n}||_{\infty }&\le& ||f_{n+1}||_{\infty }+||f_{n}||_{\infty }\\
 &\le& ||g_{n+1}||_{\infty }
 +\gamma +||f_{n}||_{\infty }\\
 &\le&
 ||f_{n}||_{\infty }+\max_{|k|\le 
d_{n+1}}|a_{k}^{\np}|+\gamma +||f_{n}||_{\infty }\\
&\le&
 2\bigl(||f_{n}||_{\infty }+\max_{|k|\le 
d_{n+1}}|a_{k}^{\np}|\bigr)
\end{eqnarray*}
if $\gamma <\max_{|k|\le 
d_{n+1}}|a_{k}^{\np}|$. The quantity on the righthand side depends only on the 
construction until step $n$ and on the vector $u_{n+1}$, but not on the 
rest of the construction at step $n+1$. In particular, it does not depend 
on 
$\gamma $ nor on  $\eta $. We have
\[
\int_{X}||f_{n+1}(x)-f_{n}(x)||^{2}d\mu (x)<\gamma ^{2}+\mu (X\setminus 
B)\,||f_{n+1}-f_{n}||_{\infty }^{2}\le \gamma ^{2}+\eta 
\,||f_{n+1}-f_{n}||_{\infty }^{2}.
\]
If both $\gamma $ and $\eta $ are chosen sufficiently small, we can ensure 
that 
for instance $$||f_{n+1}-f_{n}||_{\LL{2}(X,\bb,\mu )}<2^{-\np}.$$ This 
proves 
 (5a).
\par\smallskip 
Let us now estimate, for $x\in X$,
\[
||\Phi _{f_{n+1}}(x)-\Phi _{f_{n}}(x)||=\Bigl|\!\Bigl|\sum_{k\in\Z}
\bigl(f_{n+1}(\expo{T}{k}x)-f_{n}(\expo{T}{k}x) 
\bigr)\expo{A}{-k}z_{0}\Bigr|\!\Bigr|.
\]
The series $\sum_{k\in\Z}\expo{A}{-k}z_{0}$ being unconditionally 
convergent in 
 $Z$, there exists for every $\rho >0$ a positive integer $k_{\rho }$ such that, 
for 
every bounded sequence $(a_{k})_{k\in\Z}$ of complex numbers,
$$
 \Bigl|\!\Bigl|\sum_{|k|\ge k_{\rho}}a_{k}\,\expo{A}{-k}z_{0}\Bigr|\!\Bigr|\le 
\rho 
\,\sup_{|k|\ge k_{\rho}}|a_{k}|.$$
So
\begin{align*}
||\Phi _{f_{n+1}}(x)-\Phi _{f_{n}}(x)||&\le
\sup_{|k|<k_{\rho 
}}\bigl|f_{n+1}(\expo{T}{k}x)-f_{n}(\expo{T}{k}x)\bigr|\,.\,
\sum_{|k|<k _{\rho }}||\expo{A}{-k}z_{0}||\\
&+\rho \,.\,\sup_{|k|\ge 
k_{\rho}}\bigl|f_{n+1}(\expo{T}{k}x)-f_{n}(\expo{T}{k}x)\bigr|\\
&\le C_{\rho 
}\sup_{|k|<k_{\rho}}\bigl|f_{n+1}(\expo{T}{k}x)-f_{n}(\expo{T}{k}x)
\bigr|+\rho \,||f_{n+1}-f_{n}||_{\infty},
\end{align*}
where $C_{\rho}=\sum_{|k|<k_{\rho }}||\expo{A}{-k}z_{0}||$. We have seen 
already that $||f_{n+1}-f_{n}||_{\infty}$ does not depend on the 
quantities 
introduced at step $n+1$ of the construction, so let us fix $\rho =
\rho _{n+1}>0$ so small that $\rho _{n+1}
||f_{n+1}-f_{n}||_{\infty}<\gamma $. Then $k_{n+1}=k_{\rho _{n+1}}$
depends on the construction until step $n$, but not on $\eta $. We set
$$B_{n+1}=\bigcap_{|k|<k_{n+1}}\expo{T}{-k}B.$$ Then $\mu 
(B_{n+1})>1-(2k_{n+1}-1)\eta $ which can be made as close to 
$1$ as we wish provided $\eta $ is small enough. We also have
$$\sup_{|k|<k_{n+1}} 
\bigl|f_{n+1}(\expo{T}{k}x)-f_{n}(\expo{T}{k}x)\bigr|<\gamma \quad \textrm{ for every }
x\in B_{n+1}.$$ It follows that for every $x\in 
B_{n+1}$, $||\Phi _{f_{n+1}}(x)-\Phi _{f_{n}}(x)||\le C_{\rho 
_{n+1}}\,.\,\gamma +\gamma $. So if $\gamma '$ is any positive 
number, we can ensure by taking $\gamma $ sufficiently small 
that $$||\Phi _{f_{n+1}}(x)-\Phi _{f_{n}}(x)||\le \gamma '\quad  \textrm{ for 
every } x\in B_{n+1}.$$ Thus if $\gamma _{n+1}$ and $\eta_{n+1}$ are any fixed positive numbers, taking $\gamma $ and $\eta$ sufficiently small yields that properties (4a), (4b) and (4c) are true.
\par\smallskip
Also, we have
\begin{align*}
 \int_{X}||\Phi _{f_{n+1}}(x)-\Phi_{f_{n}}(x)||^{2}d\mu 
(x)&<\gamma'{} ^{2}+\mu (X\setminus 
B_{n+1})\,\bigl(C_{\rho_{n+1}}+\rho _{n+1} \bigr)^{2}
\,.\,||f_{n+1}-f_{n}||_{\infty}^{2}\\
&\le \gamma '{}^{2}+(2k_{n+1}-1)\,\eta \,\bigl(C_{\rho_{n+1}}+\rho 
_{n+1} \bigr)\,.\,||f_{n+1}-f_{n}||_{\infty}^{2}.
\end{align*}
Since $\eta $ can be chosen as small as we wish compared to 
$k_{n+1}$, $\rho _{n+1}$, and $||f_{n+1}-f_{n}||_{\infty}$, 
we can make the bound above as small as we wish provided $\gamma 
$ and $\eta $ are sufficiently small. So we can in particular 
ensure that $$||\Phi _{f_{n+1}}-\Phi_{f_{n}}||_{\LL{2}(X,\bb,\mu 
;Z )}<2^{-(n+1)},$$ which is (5b).
\par\smallskip 
\textbf{Step 11:}
It remains to construct the sets $G_{i}^{\np}$, 
$i\in\en{n+1}$, in such a way that property (6) holds true.
\par\smallskip 
We know that for every $i\in\en{n}$ and every $x\in G_{i}^{(n)}$, $\Phi 
_{f_{n}}(x)$ belongs to $ U_{i} $, i.e.\ $||\Phi 
_{f_{n}}(x)-u_{i}||<r_{i}$, and that $\mu (G_{i}^{(n)})\ge \delta _{i}(1+2^{-n})$.  Let $0<\kappa <r_{i}$ be so small that for 
every $i\in\en{n}$,
\[
\mu \bigl(\bigl\{x\in G_{i}^{(n)} \textrm{ ; } ||\Phi 
_{f_{n}}(x)-u_{i}||<r_{i}-\kappa \bigr\}\bigr)\ge\delta 
_{i}\bigl(1+\frac{3}{4}2^{-n} \bigr) .
\]
This number $\kappa $ only depends on the construction until step $n$.
We set for $i\in\en{n}$
\[
G_{i}^{\np}=\bigl\{x\in G_{i}^{(n)} \textrm{ ; } ||\Phi 
_{f_{n}}(x)-u_{i}||<r_{i}-\kappa\bigr\}\cap B_{n+1}
\]
and $G_{n+1}^{\np}=E$. Then obviously $G_{i}^{\np}\subseteq 
G_{i}^{(n)}$ for $i\in\en{n}$ and 
\[
\mu (G_{i}^{\np})\ge \delta 
_{i}\,(1+\frac{3}{4} 2^{-n})-(2k_{n+1}-1)\,\eta \ge \delta 
_{i}\,(1+2^{-\np})
\]
if $\eta $ is small enough. Also $\mu (G_{n+1}^{\np})=\mu (E)\ge
\delta _{n+1}(1+2^{-\np})$ if $\delta _{n+1}$ is small enough. 
So properties (6a) and (6b) are 
true. 

If $i\in\en{n}$ and $x\in G_{i}^{\np}$, then $||\Phi 
_{f_{n}}(x)-u_{i}||<r_{i}-\kappa$. Also $x\in B_{n+1}$ so that 
$||\Phi 
_{f_{n+1}}(x)-\Phi 
_{f_{n}}(x)||<\gamma '$. If $\gamma '$ is chosen less than 
$\kappa $ (which is possible since $\kappa$ depends only on the construction until step $n$), we have $||\Phi 
_{f_{n+1}}(x)-u_{i}||<r_{i}$, i.e.\ $\Phi _{f_{n+1}}(x)\in U_{i}$. 
If $i=n+1$, and $x\in G_{n+1}^{\np}=E$, then 
$\expo{T}{k}x\in\expo{T}{k}E$ for every $k\in\Z$. Hence 
\[
\Phi _{g_{n+1}}(x)=\sum_{|k|\le d_{n+1}}a_{k}^{\np}\,\expo{A}{-k}z_{0}
+\sum_{|k|> N}g_{n+1}(\expo{T}{k}x)\,\expo{A}{-k}z_{0}
\]
by the definition of the function $g_{n+1}$.
Thus $$||\Phi _{g_{n+1}}(x)-u_{n+1}||\le \bigl|\bigl|\sum_{|k|> 
N}g_{n+1}(\expo{T}{k}x)\,\expo{A}{-k}z_{0}\bigr|\bigr|.$$ Let now be $\rho 
>0$ such that $\rho \,||g_{n+1}||_{\infty }<r_{n+1}/2$. Since $$||g_{n+1}||_{\infty}\le ||f_{n}||_{\infty}+\max_{|k|\le d_{n+1}}|a_{k}^{(n)}|,$$ the number $\rho $ only 
depends on the construction until step $n$, and we can choose $N$ so large 
that $N>k_{\rho }$. Then we have for every $x\in E$
\[
||\Phi _{g_{n+1}}(x)-u_{n+1}||\le\rho \,\sup_{|k|>N}|g_{n+1}(\expo{T}
{k}x)|\le \rho \,||g_{n+1}||_{\infty }<\frac{r_{n+1}}{2}\cdot
\]
Recall that there exists a positive constant $C$ such that 
$$\bigl|\bigl|\sum_{k\in 
\Z}a_{k}\,\expo{A}{-k}z_{0}\bigr|\bigr|\le C\,\sup_{k\in\Z}|a_{k}|$$ for 
all bounded sequences $(a_{k})_{k\in\Z}$ of complex numbers. Since 
$||f_{n+1}-g_{n+1}||_{\infty}<\gamma $, we can assume by taking 
$\gamma $ sufficiently small that $||f_{n+1}-g_{n+1}||_{\infty 
}<r_{n+1}/(2C)$. Then we have, for every $x\in X$,
$$||\Phi _{f_{n+1}}(x)-\Phi _{g_{n+1}}(x)||\le C\,||f_{n+1}-g_{n+1}||_{\infty 
}<\frac{r_{n+1}}{2}\cdot$$ Hence $||\Phi _{f_{n+1}}(x)-u_{n+1}||<r_{n+1}$ for every 
 $x\in E$, i.e.\ $\Phi _{n+1}(x)$ belongs to $ U_{n+1}$ for every $x\in E=G_{n+1}
 ^{\np}$. Thus property (6c) is satisfied.
 \par\smallskip 
 This finishes the construction by induction of the functions $f_{n}$.
 
 \subsection{Construction of the isomorphism $\Phi $ and proof of Theorem
 \ref{Theo1}}\label{SousSec2c}
 By property (5b), the sequence $(\Phi _{f_{n}} )_{n\ge 0}$ 
converges in $\LL{2}(X,\bb,\mu ;Z)$ to a function 
$\Phi$ which belongs to $\LL{2}(X,\bb,\mu ;Z)$. Our aim is now to prove  that the probability
measure $\nu $ on $Z$ defined by $\nu (B)=\mu (\Phi ^{-1}(B))$ for any Borel 
subset $B$ of $Z$ has full support, and that $\Phi $ is an isomorphism 
between the two dynamical systems $(X,\bb,\mu ;T)$ and $(Z,\bb_{Z},\nu 
;A)$.
\par\smallskip 
Observe that the property (5a) of the 
sequence of functions $(f_{n})_{n\ge 0 }$ implies that $(f_{n})_{n\ge 0}$ 
converges in $\LL{2}(X,\bb,\mu )$ to a certain function 
$f\in\LL{2}(X,\bb,\mu )$. Obviously, $f$ does not belong to $\LL{\infty 
}(X,\bb,\mu  )$, and thus it makes no sense to speak of the map $\Phi 
_{f}$ (this is a difference with what happens in the construction of 
\cite{GW}). But a link between $\Phi $ and $f$ can be obtained thanks to 
the assumption (b) of Theorem \ref{Theo1}: for $\mu$-almost every $x\in X$ 
and every $n\ge 0$ we have
$\pss{z_{0}^{*}}{\Phi _{f_{n}}(x)}=\sum_{p\in 
F}c_{p}\,f_{n}(\expo{T}{p}x)$, where the set $F$ is finite. Since there 
exists a strictly increasing sequence $(n_{k})_{k\ge 0}$ of integers such 
that $\Phi _{f_{n_{k}}}\longrightarrow\Phi $ and $f_{n_{k}}\longrightarrow 
f$ $\mu $-almost everywhere, it follows that 
$$\pss{z_{0}^{*}}{\Phi(x)}=\sum_{p\in 
F}c_{p}\,f(\expo{T}{p}x)\quad \mu \textrm{-almost everywhere}.$$
\par\smallskip 
$\bullet$ \textbf{Let us first show that $\nu $ has full support.} By property 
(4c) we have that for every $n\ge 0 $ and 
every $x\in \widetilde{B}_{n}=\bigcap_{k\ge n}B_{k}$,
$||\Phi _{f_{k}}(x)-\Phi _{f_{k-1}}(x)||<\gamma _{k}$ for every $k\ge n$. Observe that 
$$\mu (\widetilde{B}_{n})\ge 1-\sum_{k\ge n}\eta _{k}>1-2\eta_{n}$$ if the sequence 
$(\eta_{n} )_{n\ge 0}$ decreases sufficiently fast, so that $\mu (\widetilde{B}_{n})\rightarrow 1$ as $n\rightarrow +\infty$. For every $n\ge 0$ and 
every $x\in \widetilde{B}_{n}$ we have in particular
\[
\bigl|\bigl|\Phi _{f_{k}}(x)-\Phi _{f_{n}}(x)\bigr|\bigr|\le
\sum_{j=n+1}^{k}\bigl|\bigl|\Phi _{f_{j}}(x)-\Phi _{f_{j-1}}(x)\bigr|\bigr|
\le \sum_{j\ge n+1}\gamma _{j}<2\gamma _{n+1}
\]
if the sequence $(\gamma _{k})_{k\ge 0}$ is sufficiently rapidly 
decreasing. As $(\Phi  _{f_{k}})_{k\ge 0}$ converges to $\Phi $ 
in $\LL{2}(X,\bb,\mu ;Z)$, there exists a subsequence of $(\Phi  
_{f_{k}})_{k\ge 0}$ which converges to $\Phi $ $\mu $-almost everywhere on 
 $X$ and thus we have: 
 $$
 \textrm{for } \mu \textrm{-almost every } x\in \widetilde{B}_{n},\quad
 ||\Phi (x)-\Phi _{f_{n}}(x)||\le 2\gamma _{n+1}.
 $$
 Let us now fix $i\ge 
0$. By property (6c), we have that, for 
every $n\ge i$ and every $x\in G_{i}^{(n)}$, $||\Phi_{f_{n}}(x)-u_{i} 
||<r_{i}$. Let $$G_{i}=\bigcap_{n\ge i}G_{i}^{(n)}.$$ As the sequence of sets 
$\bigl( G_{i}^{(n)}\bigr)_{n\ge i}$ is decreasing by (6b), we have $\mu (G_{i})=\lim_{n\to+\infty }\mu 
(G_{i}^{(n)})$ so that $\mu (G_{i})\ge\delta _{i}>0$ by property (6a). Let now $n\ge i$ be so large that 
$\mu (G_{i}\cap\widetilde{B}_{n})>0$. If $x\in G_{i}\cap 
\widetilde{B}_{n}$, then 
\[
||\Phi (x)-u_{i}||\le||\Phi (x)-\Phi _{f_{n}}(x)||+||\Phi 
_{f_{n}}(x)-u_{i}||<2\gamma _{n+1}+r_{i}.
\]
Hence $\nu \bigl(B(u_{i},r_{i}+2\gamma _{n+1}) \bigr)>0$ for all $n$ 
sufficiently large, so that it follows in particular that $\nu \bigl(B(u_{i},2r_{i}) 
\bigr)>0$. This being true for all $i\ge 0$, the measure $\nu$ has full 
support.
\par\smallskip 
Let us observe at this point that the measure $\nu $ admits a moment of 
order $2$. Indeed
\[
\int_{Z}||z||^{2}\,d\nu (z)=\int_{X}||\Phi (x)||^{2}\, d\mu (x)<+\infty 
\]
since $\Phi \in\LL{2}(X,\bb,\mu ;Z)$.
\par\smallskip 
$\bullet$ \textbf{It remains to show that $\Phi $ is an isomorphism between 
$(X,\bb,\mu ;T) $ and $(Z,\bb_{Z},\nu ;A)$.} For this it suffices to prove 
(see for instance \cite[p.\,59--60]{Wa}) that the two transformations $T$ 
and $A$ are conjugated via the map $\Phi $. First of all, we need to check 
that $\Phi (Tx)=A\Phi (x)$ for $\mu $-almost every $x\in X$. Since $\Phi 
_{f_{n}}$ tends to $\Phi $ in $\LL{2}(X,\bb,\mu ;Z)$ as $n$ tends to 
infinity, there exists a subsequence $\bigl(\Phi _{f_{n_{k}}} \bigr)
_{k\ge 0 }$ of $\bigl(\Phi _{f_{n}}\bigr)$ which tends  to $\Phi $ 
$\mu $-almost everywhere. As $\Phi _{f_{n}}(Tx)=A\Phi _{f_{n}}(x)$ 
$\mu $-almost everywhere for each $n\ge 0$, it follows that 
$\Phi (Tx)=A\Phi (x)$ $\mu $-almost everywhere.
\par\smallskip 
The second point is to check that for every subset $Q$ of $X$, $Q\in \bb$, there exists a subset $B$ of $Z$, $B\in 
\bb_{Z}$, such that $\mu \bigl(Q\vartriangle \Phi ^{-1}(B) \bigr)=0$.
So let $Q\in\bb$. Suppose that we are able to exhibit a Borel subset $C$
of $\C$ such that $Q=\bigl\{x\in X;\ \sum_{p\in 
F}c_{p}\,f(\expo{T}{p}x)\in C\bigr\}$ up to a set of measure zero. Setting 
$B=\bigl\{z\in Z;\ \pss{z_{0}^{*}}{z}\in C\bigr\}$, and remembering that
$\pss{z_{0}^{*}}{\Phi(x)}=\sum_{p\in 
F}c_{p}\,f(\expo{T}{p}x)$ $\mu $-almost everywhere,
we obtain that $B$ is a 
Borel subset of $Z$ such that $\Phi ^{-1}(B)=Q$ (up to a set of $\mu $-measure zero). So it suffices to find 
$C$ with the property above.
\par\smallskip 
Let us first introduce some notation: for each $n\ge 0$ we define the 
function $F_{n}$ on $X$ by setting $F_{n}(x)=\sum_{p\in F}c_{p}\,
f_{n}(\expo{T}{p}x)$ and the function $F$ by setting $F(x)=\sum_{p\in 
F}c_{p}\,f(\expo{T}{p}x)$. For every $i\ge 0$, let 
$$C_{i,\,0}=\bigcap_{n\ge i}F_{i,\,0}^{(n)}\quad \textrm{and} \quad
C_{i,\,1}=\bigcap_{n\ge i}F_{i,\,1}^{(n)}.$$ These are Borel subsets of 
$\C$. Recall that $J_{i}=\{j\ge i;\ Q_{j}=Q_{i}\}$ is supposed to be 
infinite. Let 
\[
\Gamma _{i,\,0}=\bigcup_{j\ge 0}\bigcap_{\genfrac{}{}{0pt}{}{k\ge 
j}{k\in J_{i}}}C_{k,\,0}\quad \textrm{and}\quad
\Gamma _{i,\,1}=\bigcup_{j\ge 0}\bigcap_{\genfrac{}{}{0pt}{}{k\ge 
j}{k\in J_{i}}}C_{k,\,1}.
\]
These two sets are Borel in $\C$. Moreover $\Gamma _{i,\,0}\cap \Gamma 
_{i,\,1}=\varnothing$. Indeed, if it is not the case, we have
\[
\Bigl( \bigcap_{\genfrac{}{}{0pt}{}{k\ge j_{1}}{k\in 
J_{i}}}C_{k,\,0}\Bigr)\cap \bigl( \bigcap_{\genfrac{}{}{0pt}{}{k\ge 
j_{2}}{k\in J_{i}}}C_{k,\,1}\bigr)\neq\varnothing
\]
for some integers $j_{1}$ and $j_{2}$. In particular, if $k\in J_{i}$ is such 
that $k\ge \max(j_{1},j_{2})$ (and such a $k$ does exist because $J_{i}$ is infinite), $C_{k,\,0}\cap C_{k,1}\neq \varnothing$. So 
for all $n\ge k$, $F_{k,\,0}^{(n)}\cap F_{k,\,1}^{(n)}\neq\varnothing$, 
which is a contradiction with (2c).
\par\smallskip 
For every $i\ge 0$, we consider the subsets of $X$ 
\begin{align*}
 D_{i}&=\bigcap_{p\in F}\Bigl(\bigcap_{k\ge i}\expo{T}{-p}B_{k} 
 \Bigr)\cap\Bigl( \bigcap_{k\ge i}C_{k}\Bigr)\\
 H_{i}&=\bigcup_{n\ge i}H_{i}^{(n)}\\
 \Omega _{i}&=\bigcap_{k\ge i}\bigl(D_{k}\setminus H_{k} \bigr)\\
 \Omega &=\bigcup_{i\ge 0}\Omega _{i}.
\end{align*}
We have $\mu (D_{i})\ge 1-(2d+2)\sum_{k\ge i}\eta _{k}\ge 1-4(d+1)\eta 
_{i}$ if the sequence $(\eta _{i})_{i\ge 0}$ decreases sufficiently fast. 
Also, (3a) and (3b) imply that $\mu 
(H_{i})\le \alpha  _{i}$ so that
\[
\mu (\Omega _{i})\ge 1-\sum_{k\ge i}\bigl(\mu (X\setminus D_{k})
+\mu (H_{k}) \bigr)\ge 1-4(d+1)\sum_{k\ge i}\eta _{k}-\sum_{k\ge i}\alpha  
 _{k}\ge 1-8(d+1)(\alpha  _{i}+\eta _{i})
\]
if the sequences $(\eta _{i})_{i\ge 0}$ and $(\alpha  _{i})_{i\ge 0}$ 
decrease to zero sufficiently fast. Thus $(\Omega _{i})_{i\ge 0}$ is an increasing 
sequence of sets such that $\mu (\Omega _{i})\longrightarrow 1$ as 
$i\longrightarrow +\infty $. So $\mu (\Omega )=1$.
\par\smallskip 
For every $n\ge 0$, every $x\in D_{n}$, every $p\in F$, and every $k\ge 
n$, by property (4b), 
$|f_{k+1}(\expo{T}{p}x)-f_{k}(\expo{T}{p}x)|<\gamma _{k+1}$. Hence 
$|F_{k+1}(x)-F_{k}(x)|\le(\sup_{p\in F}|c_{p}|)\,\gamma _{k+1}$.
Since $(F_{k})_{k\ge 0}$ converges to $F$ in $L^{2}(X,\mathcal{B}, \mu )$, there exists a subsequence $(F_{k_{j}})_{j\ge 0}$ of $(F_{k})_{k\ge 0}$ which converges to $F$ $\mu $-almost everywhere. Hence for ($\mu $-almost) every $x\in D_{n}$ we have
$$|F(x)-F_{n}(x)|\le\sup_{p\in F}|c_{p}|\sum_{k\ge n}\gamma _{k+1}<\beta _{n}$$ if 
$\gamma _{n+1},\gamma _{n+2},\dots$ are chosen at steps $n+1,\,n+2,\dots$ 
sufficiently small with respect to $\beta _{n}$. Thus $F(x)\in D(F_{n}
(x),\beta _{n})$ for every $x\in D_{n}$.
\par\smallskip 
After these preliminaries, our aim is now to show that for every $i\ge 0$, 
$F^{-1}(\Gamma _{i,\,0})=Q_{i}$ and $F^{-1}(\Gamma 
_{i,\,1})=X\setminus Q_{i}$. So, let us fix $i\ge 0$ and $k\in J_{i}$. 
Suppose that $x\in Q_{i}\cap\, \Omega _{k}$. Then for every $n\ge k$, 
$x\in Q_{i}\cap (D_{k}\setminus H_{k})\cap(D_{n}\setminus H_{n})$ so that
$
x\in Q_{i}\cap\bigl(D_{n}\setminus H_{k}^{(n)} \bigr)=Q_{k}\cap
(D_{n}\setminus H_{k}^{(n)} \bigr)
$. By property (3c), $f_{n}(x)\in
E_{k,\,0}^{(n)}$. Since $x\in D_{n}\subseteq C_{n}$, this implies that $F_{n}(x)\in 
 D_{k,\,0}^{(n)}$. It follows that for every $n\ge k$, $F(x)\in 
 F_{k,\,0}^{(n)}$, so that $F(x)\in\bigcap_{n\ge k}F_{k,\,0}^{(n)}
 =C_{k,\,0}$. We have thus proved that if $k\in J_{i}$ and $x\in Q_{i}\cap \Omega _{k}$, 
then $F(x)\in C_{k,\,0}$. 
\par\smallskip 
Let now $j\ge 0$.
If $k\ge j$ and $k\in J_{i}$, then, since $\Omega _{j}\subseteq \Omega 
_{k}$, we have $Q_{i}\cap \Omega _{j}\subseteq Q_{i}\cap \Omega _{k}$. It follows 
that if $x\in Q_{i}\cap \Omega _{j}$, 
$F(x)\in C_{k,\,0}$ for every $k\in J_{i}$, $k\ge j$, and so
$$F(x)\in\bigcap _{\genfrac{}{}{0pt}{}{k\ge j}{k\in 
J_{i}}}C_{k,\,0}.$$
Suppose now that 
$x\in Q_{i}\cap\Omega $: there exists a $j\ge 0$ such that $x\in Q_{i}\cap \Omega _{j}$. Hence
$$F(x)\in\bigcup_{j\ge 0 }\bigcap _{\genfrac{}{}{0pt}{}{k\ge j}{k\in 
J_{i}}}C_{k,\,0}= \Gamma 
_{i,\,0 }.$$
In exactly the same way, if $x\in(X\setminus Q_{i})\cap\Omega $
then $F(x)\in \Gamma_{i,\,1} $. Since $\Gamma_{i,\,0}\cap \Gamma_{i,\,1}=
\varnothing$ and $\mu (\Omega )=1$, we have proved that
\[
F^{-1}(\Gamma _{i,\,0})=Q_{i},\quad\textrm{and}\quad F^{-1}(\Gamma 
_{i,\,1})=X\setminus Q_{i}
\]
up to sets of $\mu $-measure zero. The proof is now nearly finished. Let
\begin{align*}
 \mathcal{Q}=\bigl\{Q\in\bb;\ \textrm{there exists a Borel subset}\ 
C\ \textrm{of}\ \C\textrm{ such that } \mu(Q\vartriangle F^{-1}(C))=0\bigr\}\cdot
\end{align*}
Then $\mathcal{Q}$ is a $\sigma $-algebra which contains all the sets 
$Q_{i}$, $i\ge 0$, and these sets generate $\bb$. Thus $\mathcal{Q}=\bb$. This 
finishes the proof that $\Phi $ is an isomorphism of dynamical systems, 
and Theorem \ref{Theo1} is proved.

\subsection{Proof of Theorem \ref{Theo1bis}}\label{SousSec2d} The proof of 
Theorem 
\ref{Theo1bis} is exactly similar in spirit, but some technical points 
need to be adjusted. We briefly list the most important ones. If $(X,
\bb,\mu ;T)$ is an ergodic dynamical system, which is not necessarily 
invertible, the maps $\Phi _{f}$, $f\in\LL{\infty }(X,\bb,\mu)$ are 
defined as $\Phi _{f}(x)=\sum_{k\ge 1}f(\expo{T}{k}x)\,\expo{A}{-k+r}z_{0}$ 
for $\mu $-almost every $x\in X$, where $r\in\Z$ is such that $A^{r}z_{0}=0$. Then 
\[
\Phi _{f}(Tx)=\sum_{k\in\Z}f(\expo{T}{k}x)\,\expo{A}{-k+1+r}z_{0}=\sum_{k\ge 
1}f(\expo{T}{k}x)\,\expo{A}{-k+1+r}z_{0}
\]
since $A^{r}z_{0}=0$. Hence $\Phi _{f}(Tx)=A\,\Phi _{f}(x)$ $\mu $-almost 
everywhere. As in the proof of Theorem \ref{Theo1}, we can suppose by shifting the sequence $(z_{n})_{n\in\Z}$ that $0\in F$ and that there exists a non-zero functional $z_{0}^{*}\in Z^{*}$ such that $c_{n}=\pss{z_{0}^{*}}{A^{-n+r}z_{0}}$ is non-zero \ifff\ $n\in F$.  The vectors $u_{n}$ are defined as $$u_{n}=\sum_{k=0}^{d_{n}}
a_{k}^{(n)}\,\expo{A}{-k+r}z_{0}, \quad a_{k}^{(n)}\in\C.$$ The set $E$  and the integer $N$ are 
chosen in such a way that $\mu (E)>0$, $\mu \bigl(\bigcup_{k=0}^{N}T^{-k}E \bigr)<\eta $,
and the sets $\expo{T}{-k}E$ are pairwise disjoint. This entails a 
modification of the definition of the function $g_{n+1}$:
\[
g_{n+1}(x)=
\begin{cases}
 a_{k}^{\np}&\quad\textrm{if}\ x\in\expo{T}{-N+k}E,\quad 0\le k\le d_{n+1}
\\
0&\quad\textrm{if}\ x\in\expo{T}{-N+k}E,\quad d_{n+1}<k\le N\\
f_{n}(x)&\quad\textrm{if}\ x\in X\setminus\bigcup_{k=0}^{N}\expo{T}{-k}E.
\end{cases}
\]

The set $E'$ and the integer $M$ are chosen so that $\mu \bigl(
\bigcup_{k=0}^{M-d}\expo{T}{-k}E'\bigr)<1-\eta $ and the sets 
$\expo{T}{-k}E'$, 
$0\le k\le M$, are pairwise disjoint (where $d=\max |F|$). The set $G_{n+1}^{\np}$ needs to be 
defined as $G_{n+1}^{\np}=\expo{T}{-N}E$. Up to these modifications, the 
proof is the same, and we leave the details to the reader.

\subsection{Some remarks}\label{SousSec2e}
Now we have carried out the proofs of these two results, a few remarks are 
in order.

\begin{remark}\label{Remarque10} We have seen during the proofs of 
Theorem \ref{Theo1} and \ref{Theo1bis} that the measures $\nu $ 
constructed on $Z$ admit a second order moment. Thus operators satisfying 
the assumptions of these theorems are $2$-universal for (invertible) 
ergodic systems in the sense of Definition \ref{Definition10} (see Section \ref{Sec5}). It is not difficult to check that one can ensure that all the measures $\nu $ 
actually admit moments of all orders.
\end{remark}

\begin{remark}\label{Remarque11} If $A$ is an invertible operator on 
$Z$, then it cannot be universal for all ergodic systems: indeed, suppose 
that $(X,\bb,\mu ;T)$ is a dynamical system which is isomorphic to
$(Z,\bb_{Z},\nu ;A)$ for some $A$-invariant measure $\nu $ with full 
support. Then $\nu $ is $A^{-1}$-invariant too. If 
$\Phi :X\longrightarrow Z$ is an isomorphism between $(X,\bb,\mu ;T)$
and $(Z,\bb_{Z},\nu ;A)$ then set $S=\Phi ^{-1}A^{-1}\Phi $. The measure 
$\mu $ is $S$-invariant, and since $T=\Phi ^{-1}A\Phi $, we have 
$TS=ST=\textrm{id}_{X}$, so that $T$ is invertible. Thus an invertible 
operator may only be universal for invertible ergodic systems.
\par\smallskip
If $A$ is a universal operator for ergodic systems and $(X,\bb,\mu ;T)$ is an invertible ergodic system, there exists a probability measure $\nu $ on $Z$ such that $(X,\bb,\mu ;T)$ and $(Z,\bb_{Z},\nu ;A)$ are isomorphic. Hence $A$ is invertible as a \mea-preserving transformation of $(Z,\bb_{Z},\nu)$, although, as observed above, it is not invertible as a topological map from $Z$ into itself. There is no contradiction in this: there indeed exists a Borel subset $M$ of $Z$ with $A(M)= M$ and $\nu (M)=1$ such that $A\,:\, M\To M$ is an invertible \mea-preserving transformation, i.e. there exists a map $S\,:\, M\To M$ which is measurable and \mea-preserving such that $AS(x)=SA(x)=x$ for every $x\in M$. But $M$ being distinct from the whole space $X$, and $S$ not necessarily uniformly continuous on $M$, there is no reason why $S$ should extend into a topological inverse of $A$ on $Z$.
\end{remark}

\begin{remark}\label{Remarque12} The ergodicity of the system $(X,
\bb,\mu ;T)$ has been used in the proofs of Theorems \ref{Theo1} and 
\ref{Theo1bis} only via the Rokhlin lemma. So we could have supposed 
instead that the system was aperiodic, and the operators $A$ of Theorems 
\ref{Theo1} and \ref{Theo1bis} are in fact universal for (invertible) aperiodic systems.
\end{remark}
\par\smallskip 
We finish this section with a generalization of Theorems \ref{Theo1} and 
\ref{Theo1bis}, in which assumption (a) is relaxed. This generalization is 
useful for proving the universality of some operators (see Section 
\ref{Sec4}) for which the (bi)cyclicity assumption (a) is not easily seen 
to hold true.
\begin{theorem}\label{Theo1-3}
 Let $A\in\bb(Z)$ be a bounded operator on $Z$. Suppose that there exist 
sequences $\bigl(z_{n}^{(\iota)} \bigr)_{n\in\Z}$ of vectors of $Z$, $\iota\in I$,
where either $I=\en{N}$ for some integer $N\ge 0$, or $I=\Z^{+}$, such 
that $Az_{n}^{(\iota)}=z_{n+1}^{(\iota)}$ for every $\iota\in I$ and $n\in\Z$. Suppose that 
the following three properties hold true:
\par\smallskip
\begin{enumerate}
 \item [\emph{(a')}] $\overline{\vphantom{(}\emph{span}}\,\bigl[
 A^{-n}z_{0}^{(\iota)};\ n\in\Z,\ \iota\in I\bigr]=Z$;
 \par\smallskip
 \item[\emph{(b')}] there exists a finite subset $F$ of $\Z$ such that
 \[
\overline{\vphantom{(}\emph{span}}\,\bigl[
 \{A^{-n}z_{0}^{(0)};\ n\in \Z\setminus 
F\}\cup\{A^{-n}z_{0}^{(\iota)};\ n\in\Z,\ \iota\in I\setminus\{0\}\}\bigr]\neq 
Z;
\]
\par\smallskip
\item[\emph{(c')}] the series $\sum_{n\in\Z}A^{-n}z_{0}^{(\iota)}$, $\iota\in I$, 
are unconditionally convergent in $Z$.
\end{enumerate}
\par\smallskip
Then $A$ is universal for invertible ergodic systems.
\end{theorem}

The analogue statement for universality for ergodic systems is

\begin{theorem}\label{Theo1-4}
 Let $A\in\bb(Z)$ be a bounded operator on $Z$. 
With the notations of Theorem \ref{Theo1-3},
 suppose that there exist 
sequences $\bigl(z_{n}^{(\iota)} \bigr)_{n\in \Z}$, $\iota\in I$, of vectors of $Z$  and integers $r^{(\iota)}$, $\iota \in I$,
where either $I=\en{N}$ for some integer $N\ge 0$, or $I=\Z^{+}$, such 
that $Az_{n}^{(\iota)}=z_{n+1}^{(\iota)}$ for every $n\in \Z$ and $A^{r^{(\iota )}}z_{0}^{(\iota)}=0$, 
for every $\iota \in I$. 

If assumptions (a'), (b') and (c') of Theorem \ref{Theo1-3} hold true, $A$ is universal for ergodic systems.
\end{theorem}

\begin{proof}
 The proofs are almost the same as those of Theorems \ref{Theo1} and 
\ref{Theo1bis}, with the additional complication that several functions 
$f^{(\iota)}\in\LL{2}(X,\bb,\mu )$, $\iota\in I$, have to be considered, and 
several sequences $(f_{n}^{(\iota)})_{n\ge \iota}$ of functions of $\LL{\infty }
(X,\bb,\mu )$ introduced. We sketch here the proof of Theorem \ref{Theo1-3}, and leave the proof of Theorem \ref{Theo1-4} to the reader.
\par\smallskip 
Let $(u_{n})_{n\ge 0}$ be a sequence of vectors of $Z$ which is dense in 
$Z$ and has the following property: there exist for each $n\ge 0$ an 
integer $d_{n}$ and complex coefficients $a_{k}^{(\iota ,\,n)}$, 
$k\in\en{n+1}$, $\iota \in\en{n}$, such that 
$$u_{n}=\sum_{\iota =0}^{n}\sum_{k=0}^{d_{n}}a_{k}^{(\iota ,\,n)}\expo{A}{-n}z_{0}
^{(\iota)}.$$ Let $(r_{n})_{n\ge 0}$ be a decreasing sequence of positive radii such 
that the balls $U_{n}=B(u_{n},r_{n})$, $n\ge 0$, form a basis of the 
topology of $Z$.  Let also $z_{0}^{*}$ be a non-zero functional and $F$ a finite 
subset of $\Z$ such that $\pss{z_{0}^{*}}{\expo{A}{-n}z_{0}^{(\iota)}}
=0$ for every $\iota \in I\setminus \{0\}$ and every  $n\in\Z$, 
$c_{n}=\pss{z_{0}^{*}}{\expo{A}{-n}z_{0}^{(0)}}
=0$ for every $n\in\Z\setminus F$, and 
$c_{n}=\pss{z_{0}^{*}}{\expo{A}{-n}z_{0}^{(0)}}$ is non-zero for every 
$n\in F$ (the set $F$ can be assumed to contain $0$).
\par\smallskip 
Then we construct sequences $\bigl(f_{n}^{(\iota)} \bigr)_{n\ge \iota }$, $\iota \in I$, 
as in the proof of Theorem \ref{Theo1}: at each step $n$ we construct the 
functions $f_{n}^{(\iota)}$ for $\iota \in\en{n}$: for each such $\iota$, 
the open set used in the construction is $U_{n}^{(\iota)}=
B(u_{n}^{(\iota)},\expo{2}{-(\iota +1)}r_{n})$, where 
$$u_{n}^{(\iota)}=\sum_{k=0}^{d_{n}}a_{k}^{(\iota ,\,n)}\expo{A}{-n}z_{0}^{(\iota)}$$ if 
$\iota \in\en{n}$, $u_{n}^{(\iota)}=0$ otherwise (if $I$ is finite, $\iota \in\en{n}$ 
means that $\iota \in\en{\min(|I|,n)}$). Hence $u_{n}=\sum_{\iota \in I}u_{n}^{(\iota)}$ for every $n\ge 0$.
\par\smallskip
We carry out the construction in such 
a way that properties (1)--(6) hold true for $f_{n}^{(0)}$, 
and properties (4)--(6) hold true for $f_{n}^{(\iota)}$,
$\iota \in\{1,\dots,n\}$. The important points in this construction, when 
compared to the proof of Theorem \ref{Theo1} are the following:
\par\smallskip
-- the initial functions $f_{\iota }^{(\iota)}$, $\iota \in I$, are such 
that $||f_{\iota }^{(\iota)}||_{\LL{\infty}(X,\bb,\mu )}<2^{-(\iota +1)}\frac{r^{(\iota)}}{C_{\iota }}$, so that
$$||\Phi_{\iota }^{(\iota)}||_{\LL{2}(X,\bb,\mu ;Z)}\le C_{\iota }
||f_{\iota }^{(\iota)}||_{\LL{\infty}(X,\bb,\mu )}
<2^{-(\iota +1)}{r^{(\iota )}},$$ where $C_{\iota }$ is a positive constant such that 
$$\left|\left|\sum_{k\in\Z} a_{k}^{(\iota ) } A^{-k} z_{0}^{(\iota )} \right|\right|\le C_{\iota }\sup _{k\in\Z}|a_{k}^{(\iota) }|$$ for every $(a_{k}^{\iota })_{k\in\Z}\in\ell_{\infty}(\Z)$.
\par\smallskip
-- at a given step of the construction, the sets $E$ and 
$E'$ and the parameters $\gamma $ and $\eta $ can be chosen independently of the 
index $\iota \in\en{n}$. It follows that the sets $B_{n}$ and 
$G_{i}^{(n)}$, which may a priori depend on the index $\iota $, can be 
constructed so as not to depend on it, as well as the parameters $\gamma 
_{n}$, $\eta _{n}$, and $\delta _{n}$ (the other parameters in the 
construction, as well as the sets $C_{n}$, $D_{i,\,0}^{(n)}$, 
$D_{i,\,1}^{(n)}$, $F_{i,\,0}^{(n)}$, 
$F_{i,\,1}^{(n)}$, and $H_{i}^{(n)}$ are involved only for the index 
$\iota =0$).
More precisely, Step 11 in the proof of the construction becomes: we know that for every $i\in\en{n}$ and every $x\in G_{i}^{(n)}$, we have
$$||\Phi_{f_{n}^{(\iota )}}-u_{i}^{(\iota )}||<2^{-(\iota +1)}r_{i}\quad \textrm{for every }\iota \in \en{n}.$$ Let $\kappa $ be so small that for every $i\in\en{n}$,
\[
\mu \bigl(\bigl\{x\in G_{i}^{(n)} \textrm{ ; } \textrm{for every }\iota \in \en{n},\, ||\Phi 
_{f_{n}^{(\iota )}}(x)-u_{i}^{(\iota )}||<2^{-(\iota +1)}(r_{i}-\kappa) \bigr\}\bigr)\ge\delta 
_{i}\bigl(1+\frac{3}{4} 2^{-n} \bigr) .
\]
We set for $i\in\en{n}$
\[
G_{i}^{\np}=\bigl\{x\in G_{i}^{(n)} \textrm{ ; }\textrm{for every }\iota \in \en{n},\, ||\Phi 
_{f_{n}^{(\iota )}}(x)-u_{i}^{(\iota )}||<2^{-(\iota +1)}(r_{i}-\kappa)\bigr\}\cap B_{n+1}
\]
and $G_{n+1}^{\np}=E$.
Properties (6a) and (6b) are clearly true. It remains to prove property (6c). It is not difficult to see that for every $i\in\en{n}$ and every $\iota \in\en{n}$, $\Phi 
_{f_{n+1}^{(\iota )}}(x)\in U_{i}^{(\iota )}$ for every $x\in G_{i}^{(n+1)}$. We then have to check that for every $i \in\en{n}$
and every $x\in G_{i}^{(n+1)}$, $\Phi 
_{f_{n+1}^{(n+1 )}}(x)\in U_{i}^{(n+1)}$, i.e. that 
$$||\Phi 
_{f_{n+1}^{(n+1)}}(x)-u_{i}^{(n+1)}||<2^{-(n+2)}r_{i}.$$
But $u_{i}^{(n+1)}=0$ and $$||\Phi 
_{f_{n+1}^{(n+1)}}||_{\LL{\infty}(X,\bb,\mu )}<2^{-(n +2)}r_{n+1}<2^{-(n+2)}r_{i},$$ so we do have that $\Phi 
_{f_{n+1}^{(n+1 )}}(x)\in U_{i}^{(n+1)}$. The last item in the proof of property (6c) is to show that $\Phi 
_{f_{n+1}^{(\iota )}}(x)\in U_{n+1}^{(\iota )}$ for every $x\in G_{n+1}^{(n+1)}=E$ and every $\iota \in \en{n}$, and here the proof is again exactly the same.
\par\smallskip
We thus have $$||f_{n}^{(\iota)}-f_{n-1}^{(\iota)}||_{\LL{2}(X,\bb,\mu )}<2^{-n}
\quad
\textrm{and} \quad ||\Phi _{f_{n}}^{(\iota)}-\Phi _{f_{n-1}}^{(\iota)}||_{\LL{2}(X,\bb,\mu ;Z
)}<2^{-n}$$ for every $n\ge 1$ and $\iota\in\en{n}$ and $$||\Phi 
_{f_{n}^{(\iota)}}(x)-u_{n}^{(\iota)}||<2^{-(\iota+1)}r_{n}$$ for every $x\in 
G_{i}^{(n)}$, $i\in\en{n}$, $\iota \in\en{n}$. It follows that the sequence 
$\bigl(f_{n}^{(\iota)} \bigr)_{n\ge \iota }$ converges in $\LL{2}(X,\bb,\mu )$ to a 
function $f^{(\iota)}$ which satisfies
\begin{align*}
||f^{(\iota)}||_{\LL{2}(X,\bb,\mu )}&\le ||f_{\iota}^{(\iota)}||_{\LL{2}(X,\bb,\mu )}+\sum_{n\ge \iota}
||f_{n+1}^{(\iota)}-f_{n}^{(\iota)}||_{\LL{2}(X,\bb,\mu )}\\
&< 2^{-(\iota+1)}+\sum_{n\ge \iota}2^{-(n+1)}\le 
2^{-(\iota-1)}.
\end{align*}
In the same way, $\bigl(\Phi _{f_{n}}^{(\iota)} \bigr)_{n\ge \iota}$ converges in 
$\LL{2}(X,\bb,\mu ;Z)$ to a function $\Phi ^{(\iota)}$ which satisfies
$||\Phi ^{(\iota)}||_{\LL{2}(X,\bb,\mu ;Z)}<2^{-(\iota-1)}$. In particular, the 
series $\sum_{\iota\in I}||\Phi ^{(\iota)}||_{\LL{2}(X,\bb,\mu;Z )}$ is convergent.
\par\smallskip 
We now consider the map $\Phi :X\longrightarrow Z$ defined as 
$$\Phi =\sum_{\iota\in I}\Phi ^{(\iota)}.$$ Since
\[
\int_{X}\sum_{\iota\in I}||\Phi ^{(\iota)}(x)||\,d\mu (x)\le\sum_{\iota\in I}\Bigl( 
\int_{X}||\Phi ^{(\iota)}(x)||^{2}d\mu (x)\Bigr)^{1/2}\le
\sum_{\iota\in I}2^{-(\iota -1)}
<+\infty, 
\]
$\Phi (x)$ is defined $\mu $-almost everywhere.
\par\smallskip 
If $\nu $ is the measure defined on $Z$ by setting $\nu (B)=\mu (\Phi 
^{-1}(B))$ for every $B\in\bb_{Z}$, then it is clear that 
$\Phi :(X,\bb,\mu ;T)\longrightarrow (Z,\bb_{Z}, \nu ;A)$ is a factor map. 
It is an isomorphism of dynamical systems for the same reason as in the 
proof of Theorem \ref{Theo1}: if $Q$ is any Borel subset of $X$, we know 
that there exists a Borel subset $C$ of $\C$ such that
\[
\bigl\{x\in X;\ \sum_{p\in F}c_{p}f^{(0)}(\expo{T}{p}x)\in C\bigr\}=Q.
\]
Let now $B=\bigl\{z\in Z;\ \pss{z_{0}^{*}}{z}\in C\bigr\}$. We have
\[
\Phi ^{-1}(B)=\bigl\{x\in X;\ \pss{z_{0}^{*}}{\Phi (x)}=
\sum_{p\in F}c_{p}\,f^{(0)}(\expo{T}{p}x)\in C\bigr\}=Q.
\]
So $\Phi $ is an isomorphism between $(X,\bb,\mu ;T)$ and $ 
(Z,\bb_{Z}, \nu ;A)$. It remains to prove that the measure $\nu $ has full 
support: for every $i\ge 0$, every $n\ge i$, every $\iota \in\en{n}$, and every 
 $x\in G_{i}=\bigcap_{n\ge i}G_{i}^{(n)}$, $||\Phi _{f_{n}^{(\iota)}}-u_{n}^{(\iota)}||<2^{-(\iota +1)}r_{i}$. We 
also know that for every $k\ge 0$, every $\iota \in\en{k}$ and every $x\in B_{k}$, 
$$||\Phi _{f_{k}^{(\iota)}}(x)-\Phi _{f_{k-1}^{(\iota)}}(x)||<\gamma _{k}.$$ So, if 
$x\in G_{i}\cap\bigl(\bigcap_{k\ge n }B_{k} \bigr)$ and $\iota \in\en{n}$,
\begin{align*}
 ||\Phi _{f^{(\iota)}}(x)-u_{i}^{(\iota)}||&\le\sum_{k\ge n+1}
 ||\Phi _{f_{k}^{(\iota)}}(x)-\Phi _{f_{k-1}^{(\iota)}}(x)||+
 ||\Phi _{f_{n}^{(\iota)}}(x)-u_{i}^{(\iota)}||\\
 &\le\sum_{k\ge n+1}\gamma _{k}+2^{-(\iota +1)}r_{i}<\gamma 
_{n}+2^{-(\iota +1)}r_{i}.
\end{align*}
Let us now fix $i\ge 0$.
Let $(n_{\iota })_{\iota \in I}$ be an increasing (finite or infinite) sequence of integers such that $\gamma _{n_{\iota }}<2^{-(\iota +1)}r_{i}$ and 
$\mu \bigl(G_{i}\cap\bigl(\bigcap_{k\ge n_{0}}B_{k} \bigr) \bigr)>0$. If $x\in G_{i}\cap\bigl(\bigcap_{k\ge n_{0}}B_{k} \bigr)$, then $x\in G_{i}\cap\bigl(\bigcap_{k\ge n_{\iota  }}B_{k} \bigr)$ for every $\iota \in I$, so that $$||\Phi _{f^{(\iota)}}(x)-u_{i}^{(\iota)}
||<2^{-\iota }r_{i}\quad \textrm{ for every } \iota \in I.$$ Hence
$$||\sum_{\iota \in I}\Phi _{f^{(\iota)}}(x)-\sum_{\iota \in I}u_{i}^{(\iota)}||<2r_{i}\quad \textrm{ for every }x\in G_{i}\cap\bigl(\bigcap_{k\ge n_{0 }}B_{k} \bigr),$$  
i.e.\ $\Phi (x)\in B(u_{i},2r_{i})$. This being true for every $i\ge 0 $, 
it follows that the measure $\nu $ has full support.
\end{proof}

\section{Universal unilateral and bilateral weighted shifts}\label{Sec3}
The easiest class of operators to which the criteria of Theorems 
\ref{Theo1} and \ref{Theo1bis} can be applied is the class of weighted 
shifts on $\lp{p}{\N}$ or $\lp{p}{\Z}$, $1\le p<\infty $, or on 
$\co{\N}$ or $\co{\Z}$. We first give the proof of Theorem \ref{Theo2}.

\subsection{Proof of Theorem \ref{Theo2}}\label{SousSec3a}
Suppose that $\bw $ is a unilateral weighted shift with weights
$(w_{n})_{n\ge 1}$ on one of the spaces $\lp{p}{\N}$, $1\le p<\infty $, or 
$\co{\N}$, with $\bw e_{0}=0$ and $\bw e_{n}=w_{n}e_{n-1}$ for every $n\ge 
1$. Let us set $z_{0}=e_{0}$ and $z_{-n}=1/(\wn{n})\,e_{n}$ for every 
$n\ge 1$. Then 
\[
\bw z_{-n}=\dfrac{1}{\wn{n}}\,w_{n}e_{n-1}=z_{-(n-1)}\quad\textrm{for every}\ n\ge 1,
\ \textrm{and}\ \bw z_{0}=0.
\]
If we write $z_{-n}=\bw ^{-n}e_{0}$ for every $n\ge 1$, then the vectors 
$\bw^{-n}z_{0}$, $n\ge 0$, span a dense subset of $\lp{p}{\N}$ (or $\co{\N}$). Also, 
the linear span of the vectors $\bw^{-n}z_{0}$, $n\ge 1$, is not dense, so 
that  assumption (b) of Theorem \ref{Theo1bis} always  holds true with $F
=\{0\}$.
\par\smallskip 
The series $\sum_{n\ge 0}\bw^{-n}e_{0}$ is unconditionally convergent in 
$\lp{p}{\N}$ (resp.\ $\co{\N}$) if and only if the series $\sum_{n\ge 1}
1/|\wn{n}|^{p}$ is convergent (resp.\ if and only if 
$|\wn{n}|\longrightarrow +\infty $ as $n\longrightarrow +\infty $). So if 
this last condition is satisfied, $\bw$ is universal on $\lp{p}{\N}$ 
(resp.\ 
$\co{\N}$) for ergodic systems. 
\par\smallskip 
This condition is also necessary for $\bw$ to be universal for either ergodic or invertible ergodic systems, but the 
arguments are different depending on whether $\bw$ acts on $\lp{p}{\N}$ 
for some $1\le p<+\infty$  or on $\co{\N}$.
\par\smallskip 
-- if $\bw$ acts on $\lp{p}{\N}$, $1\le p< +\infty $, and is universal for 
ergodic systems, then $\bw$ is necessarily frequently hypercyclic, and by the characterization of frequently hypercyclic weighted shifts on $\ell_{p}(\N)$ of 
\cite{BR}, the series $\sum_{n\ge 1}1/|\wn{n}|^{p}$ is convergent;
\par\smallskip 
-- if $\bw$ acts on $\co{\N}$, then the same argument does not apply since 
the condition $|\wn{n}|\longrightarrow +\infty $ as $n\longrightarrow 
+\infty $ does not characterize frequently hypercyclic backward  weighted shifts on  $\co{\N}$
(see \cite{BG2} and
\cite{BR} for details). But if $\bw$ is universal for (invertible) ergodic systems, then it is necessarily strongly
mixing with respect to some invariant measure with full support, hence in 
particular topologically mixing. So, for every $\varepsilon >0$, there 
exists an integer $n_{\varepsilon }$ such that, for every $n\ge 
n_{\varepsilon }$, there exists a vector $x^{(n,\,\varepsilon )}\in c_{0}(\N)$ with
$||x^{(n,\,\varepsilon )}||_{\infty}<\varepsilon $ such that 
$||\bw^{n}\,x^{(n,\,\varepsilon )}-e_{0}||_{\infty}<1/2$. In particular,
$|x_{n}^{(n,\,\varepsilon )}\,\wn{n}-1|<1/2$, so that 
$|x_{n}^{(n,\,\varepsilon )}\,\wn{n}|>1/2$. Thus $\varepsilon |\wn{n}|>1/2$ 
for every $n\ge n_{\varepsilon }$. It follows that 
$|\wn{n}|\longrightarrow +\infty $ as $n\longrightarrow +\infty $.
\par\smallskip 
The arguments for  bilateral weighted shifts are exactly the same and we leave 
them to the reader.

\subsection{Proof of Theorem \ref{Theo4}}\label{SousSec3b}
The proof of Theorem \ref{Theo4} relies on the following simple idea: 
suppose that $Z$ is a Banach space admitting a biorthogonal system 
$(u_{n},u_{n}^{*})_{n\ge 0}$ having the following property:
\par\smallskip 
\emph{there exists a bounded sequence $(\omega _{n})_{n\ge 1}$ of non-zero 
weights such that $B_{\omega}$ defined by $B_{\omega}u_{n}=\omega _{n}u_{n-1}$ for every 
$n\ge 1$ and $B_{\omega} u_{0}=0$ is a bounded operator on $Z$ and 
$$\sum_{n\ge 1}\frac{||u_{n}||}{|\omega _{1}\ldots\omega _{n}|}<+\infty. $$}
\par\smallskip 
Then $B_{\omega}$ is a universal operator for ergodic systems on $Z$. The proof 
of this statement is exactly similar to that of Theorem \ref{Theo2}: we 
set $z_{0}=u_{0}$ and $z_{-n}=(1/|\omega _{1}\ldots\omega _{n}|)\,u_{n}$. Then
$\overline{\vphantom{(}\textrm{span}}\,\bigl[z_{-n};\ n\ge 0 \bigr]=Z$, 
$\pss{u_{0}^{*}}{z_{-n}}=0$ for every $n\ge 1$ so that 
$\overline{\vphantom{(}\textrm{span}}\,\bigl[z_{-n};\ n\ge 1 \bigr]\neq Z$,
and lastly the series $\sum_{n\ge 0}z_{-n}$ is unconditionally convergent 
since $\sum_{n\ge 0}||z_{-n}||<+\infty $. So Theorem \ref{Theo1bis} applies.
\par\smallskip 
Suppose that $Z$ can be decomposed as a topological sum $Z=E\oplus Y$ 
where $E$ has a sub-symmetric basis $(e_{n})_{n\ge 0}$ (i.e.\ the basis 
 $(e_{n})_{n\ge 0}$ is unconditional and equivalent to each of its 
subsequences). Let $(e_{n}^{*})_{n\ge 0}$ denote the family of biorthogonal 
functionals on $E$, which we extend to $Z$ by setting $\pss{e_{n}^{*}}{y}=0$ for 
every $y\in Y$. Let also $(y_{n},y_{n}^{*})_{n\ge 0}$ be a bounded 
biorthogonal system for $Y$, where each $y_{n}^{*}$ is extended to $Z$ by setting $\pss{y_{n}^{*}}{e}=0$ for every $e\in E$. We denote by $P_{E}$ and $P_{Y}$ the 
projections of $Z$ onto $E$ and $Y$ respectively, associated to the 
decomposition $Z=E\oplus Y$, and we let $M=\max(||P_{E}||,||P_{Y}||)$. 
Since the basis $(e_{n})_{n\ge 0}$ is sub-symmetric, $B_{\alpha }$ defined by 
$B_{\alpha } e_{n}=\alpha_{n}e_{n-1}$ for every $n\ge 1$ and $B_{\alpha } e_{0}=0$ is a bounded 
operator on $E$ for any bounded sequence of weights $(\alpha _{n})_{n\ge 1}$. Let $(n_{k})_{k\ge 0}$ be a strictly 
increasing sequence of integers with $n_{0}=0$. Consider the biorthogonal 
system $(u_{n},u_{n}^{*})$ of $Z$ defined by setting
\begin{align*}
 u_{n}&=
 \begin{cases}
  y_{k}&\quad\textrm{if}\ n=n_{k}\ \textrm{for some}\ k\ge 0\\
  e_{n-k-1}&\quad\textrm{if}\ n\in\{n_{k}+1,\dots,n_{k+1}-1\}\ \textrm{for 
some}\ k\ge 0
 \end{cases}\\
 u_{n}^{*}&=
 \begin{cases}
  y_{k}^{*}&\quad\textrm{if}\ n=n_{k}\ \textrm{for some}\ k\ge 0\\
  e_{n-k-1}^{*}&\quad\textrm{if}\ n\in\{n_{k}+1,\dots,n_{k+1}-1\}\ 
\textrm{for 
some}\ k\ge 0.
 \end{cases}
\end{align*}
With this definition, $\{u_{n} ;\ n_{k}<n<n_{k+1}\}=\{e_{n};\ n_{k}-k
\le n<n_{k+1}-(k+1)\}$, and
$\{u_{n};\ n\in\{n_{k};\ k\ge 0\}\}=\{y_{n};\ n\ge 0\}$, so that
$(u_{n},u_{n}^{*})_{n\ge 0}$ is indeed a bounded biorthogonal system of $Z$.
\par\smallskip 
We define an operator $A$ on $Z$ by setting for every $z\in Z$
\[
  Az=\sum_{k\ge 1}u_{n_{k}}^{*}(z)\,w_{k}u_{n_{k}-1}+2\sum_{k\ge 
0}\,\,\Bigl(
\sum_{n=n_{k}+2}^{n_{k+1}-1}u_{n}^{*}(z)\,u_{n-1}\Bigr)+\sum_{k\ge 0}
u_{n_{k}+1}^{*}(z)\,w'_{k}u_{n_{k}}
\]
where the weights $w_{k}$ and $w'_{k}$ are defined by 
$$w_{k}=\frac{2^{-k}}{||y_{k}^{*}||}\quad \textrm{and} \quad w'_{k}=\frac{2^{-k}}{||y_{k}||},\quad k\ge 0.$$ 
Observe that these weights do not depend on the sequence $(n_{k})_{k\ge 
0}$.
\par\smallskip 
This operator is a backward weighted shift with respect to the 
biorthogonal system $(u_{n},u_{n}^{*})_{n\ge 0}$: $Au_{n}=\omega 
_{n}u_{n-1}$ where
\[
\omega _{n}=
\begin{cases}
 w_{k}&\quad\textrm{if}\ n=n_{k}\ \textrm{for some}\ k\ge 1\\
 2&\quad\textrm{if}\ n\in\{n_{k}+2,\dots,n_{k+1}-1\}\ \textrm{for some}\ 
k\ge 0\\
w'_{k}&\quad\textrm{if}\ n=n_{k}+1\ \textrm{for some}\ k\ge 1.\\
\end{cases}
\]
Let us first check that $A$ is bounded.
We have for every $z\in Z$
\[
Az=\sum_{k\ge 1}y_{k}^{*}(z)\,w_{k}e_{n_{k-1}-k}+\sum_{k\ge 
0}\,\,\left(
\sum_{n=n_{k}+2}^{n_{k+1}-1}e_{n-k-1}^{*}(z)\,2\,e_{n-k-2}\right)+\sum_{
k\ge 0}
e^{*}_{n_{k}-k}(z)\,w'_{k}y_{k}.
\]
Since the basis $(e_{n})_{n\ge 0}$ is sub-symmetric, there exists a positive constant $C$ such that for every $z\in Z$,
$$\left|\left|\sum_{k\ge 0}\left(\sum_{n=n_{k}+2}^{n_{k+1}-1} e^{*}_{n-k-1}(z)\,e_{n-k-2}\right)\right|\right|\le C||z||.$$
Hence 
\[
||Az||\le||z||\sum_{k\ge 1}||y_{k}^{*}||\,w_{k}||e_{n_{k-1}-k}||+2C||z||+
||z||\sum_{k\ge 0}||e^{*}_{n_{k}-k}||\,w'_{k}||y_{k}||.
\]
Since $\sup_{n\ge 0}||e_{n}||$ and $\sup_{n\ge 0}||e_{n}^{*}||$
are finite, the conditions on the weights $w_{k}$ and $w'_{k}$ imply that 
$A$ is bounded.
\par\smallskip 
In order to show that $A$ is universal for ergodic systems, it remains to 
choose the sequence $(n_{k})_{k\ge 0}$ in such a way that the series 
$$\sum_{n\ge 1}\frac{||u_{n}||}{\omega _{1}\dots\omega _{n} }$$ is 
convergent. We have
\begin{align*}
\sum_{n\ge 1}\dfrac{||u_{n}||}{\omega _{1}\dots\omega _{n} }&=
\sum_{k\ge 0} 
\sum_{n=n_{k}+1}^{n_{k+1}}\dfrac{||u_{n}||}{\omega 
_{1}\dots\omega _{n} }\\
&\!\!\!\!\!\!=\sum_{k\ge 0}\dfrac{1}{w_{1}\dots w_{k}\,w'_{1}\dots 
w'_{k}\,2^{n_{k}-2k}}
\left(\sum_{n=n_{k}+1}^{n_{k+1}-1}\dfrac{||u_{n}||}{2^{n-(n_{k}+1)}}
+\dfrac{||u_{n_{k+1}}||}{2^{(n_{k+1}-1)-(n_{k}+1)}w_{k+1}}\right).
\end{align*}
If we write $C_{k}=\max\{||u_{n}||;\ n_{k}+1\le n\le n_{k+1}\}, $ this 
yields
\[
\sum_{n\ge 1}\dfrac{||u_{n}||}{\omega _{1}\dots\omega _{n} }\le
\sum_{k\ge 0}\dfrac{C_{k}}{w_{1}\dots w_{k}\,w'_{1}\dots 
w'_{k}\,2^{n_{k}-2k}}\Biggl(  
2+\dfrac{1}{2^{n_{k+1}-n_{k}-2}\,w_{k+1}}\Biggr).
\]
Since the weights $w_{k}$ and $w_{k}'$ are defined independently of the 
sequence $(n_{k})_{k\ge 0}$, we can choose this sequence growing so fast 
that 
\[
2^{n_{k}}>\max\Biggl(\dfrac{2^{n_{k-1}+2}}{w_{k}},\,
\dfrac{2^{3k}\,C_{k}}{w_{1}\dots w_{k}\,w'_{1}\dots 
w'_{k}}\Biggr)\quad\textrm{for every}\ k\ge 1.
\]
Then $$\sum_{n\ge 1}\displaystyle \dfrac{||u_{n}||}{\omega _{1}\dots\omega _{n} 
}\le  \sum_{k\ge 0} 3\,.\, 2^{-k},$$ from which it 
follows that the series $$\sum_{n\ge 1}\frac{||u_{n}||}{\omega 
_{1}\dots\omega _{n} }$$ is convergent. This terminates the proof of 
Theorem \ref{Theo4}.

\section{Unimodular eigenvectors and universality}\label{Sec4}
We begin this section with the proof of Theorem \ref{Theo5}, which gives a 
straightforward criterion in  terms of unimodular eigenvectors for an operator to be 
universal for (invertible) ergodic systems.

\subsection{Proof of Theorem \ref{Theo5}}\label{SousSec4a} Let $A$ be 
a bounded operator on the complex separable infinite-dimensional Banach 
space $Z$, admitting a unimodular eigenvectorfield $E$ satisfying 
assumptions (i), (ii), and (iii) of Theorem \ref{Theo5}. Set, for each 
$n\in\Z$, $z_{n}=\widehat{E}(-n)$. Then $A{z_{n}}=z_{n+1}$ for every $n\in 
\Z$. The vectors $z_{n}$, $n\in\Z$, span a dense  subspace of $Z$. 
Indeed, suppose that $z^{*}\in Z^{*}$ is such that $\pss{z^{*}}{z_{n}}=0$ 
for every
$n\in\Z$. Then the function $\pss{z^{*}}{E(\,.\,)}$ is zero almost 
everywhere, and by (i) it follows that $z^{*}=0$. So $z_{0}$ is bicyclic for $A$. Let  
now $F$ be the spectrum of the polynomial $q$ defined by $q(e^{i\theta })=\sum_{n\in\Z} \widehat{p}(-n)e^{in\theta }$. For every $n\in\Z\setminus
F$, we have 
\[
\pss{z_{0}^{*}}{z_{n}}=\int_{\T}\lambda ^{n}\,\pss{z_{0}^{*}}{E
(\lambda )}\,d\lambda =\widehat{q}(n)=0.
\]
Since $z_{0}^{*}$ is non-zero, it follows that the linear span of the 
vectors $z_{n}$, $n\in\Z\setminus F$, is not dense in $Z$. Lastly, 
assumption (iii) of Theorem \ref{Theo5} states that the series 
$\sum_{n\in\Z }z_{n}$ is unconditionally convergent. So the hypotheses of 
Theorem \ref{Theo1} are satisfied, and $A$ is universal for invertible 
ergodic systems. If $\widehat{E}(-r)=0$ for some integer $r\in\Z$, then $A^{r}z_{0}=0$, and Theorem 
\ref{Theo1bis} applies.
\par\smallskip 
Theorem \ref{Theo6} can now be obtained as a consequence of Theorem 
\ref{Theo5}.
\subsection{Proof of Theorem \ref{Theo6}}\label{SousSec4b}
Suppose that $A\in\bb(Z)$ admits a unimodular eigenvectorfield $E$ which 
is analytic in a neighborhood $\Omega $ of $\T$, and such that the 
vectors $E(\lambda )$, $\lambda \in\T$, span a dense subspace of $Z$. Then 
the restriction of $E$ to $\T$ is $\sigma $-spanning for any measure 
$\sigma $ with infinite support. In particular, it is $d\lambda 
$-spanning, where $d\lambda $ is the normalized Lebesgue measure on $\T$.
\par\smallskip 
Let $z_{0}^{*}$ be any non-zero element of $Z^{*}$. The function $\varphi 
$ defined on $\Omega $ by setting $$\varphi 
(\lambda)=\pss{z_{0}^{*}}{E(\lambda )},\quad \lambda \in\T,$$ is analytic on 
$\Omega $, and not identically zero since the span of the vectors 
$E(\lambda )$, 
$\lambda \in\T$, is dense in $Z$. Thus $\varphi $ admits only finitely many 
zeroes $z_{1},\dots,z_{r}$ on $\T$, with respective multiplicities 
$d_{1},\dots,d_{r}$. Let $p(z)=\prod_{j=1}^{r}(z-z_{j})^{d_{j}}$. There 
exists a function $\psi $, which is analytic on $\Omega $ and does not 
vanish on a neighborhood $\Omega '$ of $\T$ such that 
$\varphi (z)=p(z)\psi (z)$ for every $z\in\Omega '$.
\par\smallskip 
Consider the eigenvectorfield $F:\Omega '\longrightarrow Z$ defined by 
$F(\lambda )=E(\lambda )/\psi (\lambda )$. The span of the vectors 
$F(\lambda )$, $\lambda \in\T$, is dense in $Z$. Moreover, for every 
$\lambda \in \T$, 
\[
\pss{z_{0}^{*}}{F(\lambda )}=\dfrac{1}{\psi (\lambda )}\,\pss{z_{0}^{*}}
{E(\lambda )}=p(\lambda ).
\]
Lastly, since $F$ is analytic on a neighborhood of $\T$, there exists 
$a\in(0,1)$ such that $$||\widehat{F}(n)||=\textrm{O}(a^{|n|})\quad \textrm{as }|n|
\longrightarrow +\infty .$$ Hence the series $\sum_{n\in\Z}\widehat{F}(n)$ 
is unconditionally convergent. The assumptions of Theorem \ref{Theo5} are 
thus satisfied, and $A$ is universal for invertible ergodic systems.
\par\smallskip
If $E$ is analytic in a neighborhood of $\overline{\D}$, the same kind of 
reasoning applies: if $z_{1},\dots,z_{r}$ are the zeroes of 
$\varphi (\lambda )=\pss{z_{0}^{*}}{E(\lambda )}$ on $\overline{\D}$ with multiplicities 
$d_{1},\dots,d_{r}$, and if $p(z)=\prod_{j=1}^{r}(z-z_{j})^{d_{j}}$, then 
$\varphi (z)=p(z)\psi (z)$ on a neighborhood of $\overline{\D}$, where 
$\psi $ is an analytic function on this neighborhood which does not vanish. 
If we consider again the eigenvectorfield $F$ defined by $F(\lambda )=
E(\lambda )/\psi (\lambda )$, then $F$ is analytic in a neighborhood of 
$\overline{\D}$ so that $\widehat{F}(n)=0$ for every $n<0$. So Theorem
\ref{Theo5} applies again and $A$ is universal for ergodic systems.
\par\smallskip 
Let us mention here that, using Theorems \ref{Theo1-3} and \ref{Theo1-4}, Theorems \ref{Theo5} and \ref{Theo6} can be 
generalized to the case where $A$ admits several eigenvectorfields.

\subsection{Examples and applications}\label{SousSec4c}
We present in this section several examples of operators which can be 
shown to be universal thanks to Theorems \ref{Theo5} and \ref{Theo6}. 
First, one can easily retrieve the universality of some of the weighted 
shifts considered in Theorem \ref{Theo2} above.

\begin{example}\label{Example1}
Let $\alpha B$, $|\alpha |>1$, be a multiple of the unweighted backward 
shift on $\lp{p}{\N}$, $1\le p<+\infty $, or $\co{\N}$. Then $\alpha B$ 
admits an eigenvectorfield $E$ defined on the disk of radius $|\alpha |$ by
$$E(\lambda )=\sum_{n\ge 0}\left(\frac{\lambda }{\alpha }\right)^{n-1}e_{n}.$$ So $E$ is analytic 
on $\D(0,|\alpha |)$, and it is easy to check that the eigenvectors 
$E(\lambda )$, $\lambda \in\T$, span a dense subspace of $\lp{p}{\N}$ or 
$\co{\N}$. Thus Theorem \ref{Theo6} applies, and $\alpha B$ is universal 
for ergodic systems. The same argument applies for instance to the 
weighted shift $S_{w}$ on $\lp{p}{\Z}$ or $\co{\Z}$, where the weight $w$ is given by 
$w_{n}=2$ if $n\ge 1$ and $w_{n}=1/2$ if $n\le 0$: $S_{w}$ is universal 
for invertible ergodic systems.
\end{example}
\par\smallskip 
Our second class of examples is given by adjoints of multipliers 
$M_{\varphi }^{*}$ on $H^{2}(\D)$. This is a natural class
 of operators to consider here, since their dynamical properties 
(hypercyclicity, frequent hypercyclicity, ergodicity) are rather well understood.
See for instance \cite{BM} for details.

\begin{example}\label{Example2}
Denoting by $\D$ the open unit disk, let $\varphi :\D\longrightarrow\C$ be an analytic map belonging to 
$H^{\infty }(\D)$, and consider the adjoint $M_{\varphi }^{*}$ of the 
multiplier $M_{\varphi }$ defined on the Hardy space $H^{2}(\D)$ by 
setting 
$M_{\varphi }\,f=\varphi f$ for every $f\in H^{2}(\D)$. If $\T\subseteq\varphi (\D)$, then 
$M^{*}_{\varphi }$ is universal for invertible ergodic systems.
\end{example}

 \begin{proof}[Proof of Example \ref{Example2}] 
Suppose that the analytic map $\varphi :\D\To  \C$ is such that $\varphi (\D)$ contains the unit circle. Since $\varphi $ is open, $K=\varphi ^{-1}(\T)$ is a compact subset of $\D$, and there exists $\rho  \in(0,1)$ such that $K\subseteq D(0,\rho  )$. The derivative $\varphi '$ of $\varphi $ can vanish only finitely many times on $K$, and we denote by $z_{1},\dots,z_{r}$ the distinct zeroes of $\varphi '$ on $K$, with respective multiplicities $m_{1},\dots,m_{r}$. There exists for every $j\in\{1,\dots,r\}$ a disk $D(z_{j},\varepsilon _{j})$ with $\varepsilon _{j}>0$ and two holomorphic functions $\psi _{j}$ and $\sigma _{j}$ on $D(z_{j},\varepsilon _{j})$ which do not vanish here  such that for every $z\in D(z_{j},\varepsilon _{j})$, $$\varphi (z)=\lambda _{j}+(z-z_{j})^{d_{j}}\psi _{j}(z)\quad \textrm{and}\quad\varphi '(z)=(z-z_{j})^{m_{j}}\sigma _{j}(z)$$ where $d_{j}=m_{j}+1$ and $\lambda _{j}=\varphi (z_{j})$.
Also, there exists a holomorphic function $\beta _{j}$ on $D(z_{j},\varepsilon _{j})$ such that $\psi _{j}=\beta _{j}^{d_{j}}$. If we set $\alpha _{j}(z)=(z-z_{j})\beta _{j}(z)$ for $z\in D(z_{j},\varepsilon _{j})$, we can assume that $\alpha _{j}$ is a biholomorphism from a certain open neighborhood $V_{j}$ of $z_{j}$ contained in $D(z_{j},\varepsilon _{j})$ onto an open disk $D(0,\delta _{j})$.
Let 
\[
\Omega _{j,\,0} =D(0,\delta _{j}^{d_{j}})\setminus [0,\delta _{j}^{d_{j}}) 
\]
and let $\gamma _{j,\,0} $ be an holomorphic determination of the $d_{j}$-th root of $z$ on $\Omega _{j,\,0} $: 
$\gamma _{j,\,0} (z)^{d_{j}}=z$ for every $z\in\Omega _{j,\,0} $. 
We also set $V_{j,\,0} =\bigl\{z\in V_{j};\ \alpha _{j}(z)^{d_{j}}\in\Omega _{j,\,0} \bigr\}$, $U_{j,\,0} =\lambda _{j}+\Omega _{j,\,0}$ and $U_{j}=D(\lambda_{j} ,\delta_{j} ^{d_{j}})$. Observe that $U_{j,\,0}\cap\T$ contains a set of the form $\Gamma_j\setminus\{\lambda_j\}$, where $\Gamma_j$ is an open sub-arc of $\T$ containing the point $\lambda_j$.

We now claim that  $\varphi $ is a biholomorphism from $V_{j,\,0} $ onto $U_{j,\,0} $. Let us first check that $\varphi (V_{j,\,0} )=U_{j,\,0} $: $z\in V_{j}$ belongs to $V_{j,\,0} $ if and only if $\alpha _{j}(z)^{d_{j}}=\varphi (z)-\lambda _{j}$ belongs to $\Omega _{j,\,0}$, i.e.\ if and only if $\varphi (z)$ belongs to $U_{j,\,0}$.
Let us set, for $z\in U_{j,\,0}$, $$\varphi ^{-1}_{j,\,0}(z)=\alpha ^{-1}_{j}\circ\gamma _{j,\,0}\,(z-\lambda _{j}).$$ This definition makes sense: if $z\in U_{j,\,0}$, $z-\lambda _{j}\in\Omega _{j,\,0}$, so that $\gamma _{j,\,0}(z-\lambda _{j})\in D(0,\delta _{j})$, and $\alpha _{j}$ is a biholomorphism from $V_{j}$ onto $D(0,\delta _{j})$. This function $\varphi^{-1}_{j,\,0}$ is thus well-defined and holomorphic on $U_{j,\,0}$, and we have for every $z\in U_{j,\,0}$
\[
\varphi \bigl( \varphi ^{-1}_{j,\,0}(z)\bigr)=\lambda _{j}+\alpha _{j}\bigl( \varphi ^{-1}_{j,\,0}(z)\bigr)^{d_{j}}=\lambda _{j}+\bigl( \gamma _{j,\,0}(z-\lambda _{j})\bigr)^{d_{j}}=\lambda _{j}+z-\lambda _{j}=z
\]
since $z-\lambda _{j}$ belongs to $\Omega _{j,\,0}$ and $\gamma _{j,\,0}(z)^{d_{j}}=z$ for every $z\in\Omega _{j,\,0}$. It follows from this that $\varphi :V_{j,\,0}\To  U_{j,\,0}$ is a biholomorphism, the inverse of which is $\varphi _{j,\,0}^{-1}$. Restricting the sets $V_{j,\,0}$ and $U_{j,\,0}$, we can and do assume that, for every $j\in\{1,\ldots,r\}$,
$U_{j,\,0}\cap\T=\Gamma_j\setminus\{\lambda_j\}$, where $\Gamma_j$ is an open sub-arc of $\T$ containing the point $\lambda_j$, and that the arcs $\Gamma_j$, $j\in\{1,\ldots, r\}$, do not intersect.
\par\medskip 
If $z\in K$ is such that $\varphi (z)\not \in\{\lambda_{1},\dots,\lambda_{r}\}$ and $\varphi '(z)\neq 0$, $\varphi $ is a biholomorphism in a neighborhood of $z$. For every such $z$, let $V_{z}$ be an open neighborhood of $z$, and $U_{z}$ a disk centered at $\varphi (z)$ of radius $\rho _{z}>0$ such that $\varphi :V_{z}\To  U_{z}$ is a biholomorphism,
$\varphi'$ does not vanish on $V_{z}$, and the closure of the set
$U_{z}$ contains none of the points $\lambda_{j}$, $j\in\{1,\ldots,r\}$.
The disks $U_{z}$, $z\in \varphi^{-1}(\T\setminus \{\lambda_{1},\dots,\lambda_{r}\})$ and $U_{j}$, $j\in\{1,\ldots,r\}$, form an open covering of $\T$ (remember our assumption that $\T\subseteq \varphi(\D)$), so one can extract from it a finite covering of the form $U_{\xi _{1}},\dots,U_{\xi _{s}},U_{1},\dots,U_{r}$, $s\ge 1$. 
Denote by $\bigl( \Omega_{l}\bigr)_{l\in\Lambda } $ the finite family of open subsets of $\D$ consisting of the sets $U_{\xi _{i}}$, $1\le i\le s$,  and $U_{j,\,0}$, $1\le j\le r$. For every $l\in \Lambda $, $\Omega'_l=\Omega _{l}\cap \T$ is either an open subarc of $\T$ whose closure is contained in $\T\setminus\{\lambda _{1},\dots,\lambda _{r}\}$ (when $\Omega_{l}=U_{\xi _{i}}$ for some $i\in\{1,\ldots, s\}$) or 
an open subarc minus one point (when $\Omega_{l}=U_{j,\,0}$ for some $j\in\{1,\ldots, r\}$, in which case 
$\Omega'_{l}=\Gamma_j\setminus\{\lambda_j\}$).
We have $$\bigcup_{l\in\Lambda }\Omega' _{l}\cap\T=\T\setminus\{\lambda _{1},\dots,\lambda _{r}\}.$$ 
Observe that for every $j\in\{1,\ldots, r\}$, there is a unique index $l\in \Lambda $ such that $\Omega'_{l}$ contains the point $\lambda_j$ in its closure. Writing this index as $l_{j}$, we have $\Omega _{l_{j}}=U_{j,\,0}$ and $\Omega'_{l_{j}}=\Omega _{l_{j}}\cap\T=\Gamma_j\setminus\{\lambda_j\}$. For every $l\in\Lambda $, we denote by $\varphi ^{-1}_{l}$ the inverse of $\varphi $ defined on the set $\Omega _{l}$. 
\par\smallskip
Let $\bigl( v_{l}\bigr)_{l\in\Lambda }$ be a finite $\mathcal{C}^{1}$-partition of the unity associated to the finite covering $\bigl( \Omega' _{l}\bigr)_{l\in\Lambda }$ of $\T\setminus\{\lambda _{1},\dots,\lambda _{r}\}$, with $v_{l}$ supported on $\Omega'_l$ for every $l\in\Lambda$.
If $\Omega'_l=U_{\xi_i}\cap\T$ for some $i\in\{1,\ldots, s\}$, the function $v_{l}$, which is supported on $\Omega '_{l}$, obsviously extend into a $\mathcal{C}^{1}$ function on the whole circle $\T$ by setting $v_{l}(\lambda _{j})=0$ for every $j\in\{1,\ldots,r\}$.
If 
$\Omega'_l=\Omega '_{l_{j}}$ for some $j\in\{1,\ldots, r\}$, the observation above shows that
there exists an open sub-arc $\Gamma '_{j}$ of $\Gamma  _{j}$, containing $\lambda _{j}$, such that the only index $l\in \Lambda $ for which $v_{l}$ does not identically vanish on $\Gamma  '_{j}\setminus\{\lambda _{j}\}$ is $l_{j}$.
Hence $v_{l_{j}}$ is equal to $1$ on $\Gamma  '_{j}\setminus\{\lambda_j\}$, and it follows that $v_{l_{j}}$ can be extended into a $\mathcal{C}^{1}$ function on the whole circle $\T$ by setting $v_{l_{j}}(\lambda _{j})=1$. We can thus assume that all the functions $v_{l}$, $l\in\Lambda$, are defined and of class $\mathcal{C}^{1}$ on $\T$. 
\par\smallskip
For every $z\in\D$, let $k_{z}$ be the reproducing kernel of the space 
$H^{2}(\D)$ at the point $z$:
\[
k_{z}(\xi)= \sum_{n\ge 0}\overline{\vphantom{t}z}^{\,n}\xi 
^{n}=\dfrac{1}{1-\overline{\vphantom{t}z}\,\xi  },\quad \xi \in\D,
\]
and $k_{z}$ is characterized by the property that $\pss{f}{k_{z}}=f(z)$ 
for every $f\in H^{2}(\D)$.
We have $M_{\varphi}^{*}k_{z}=\overline{\varphi(z)}k_{z}$
for every $z\in\D$.
We then introduce the polynomial $$p(z)=\prod_{j=1}^{r}(z-\lambda _{j})^{2}$$ and the maps $F:\T\setminus\{\lambda _{1},\dots,\lambda _{r}\}\To  H^{2}(\D)$ and 
$E:\T\setminus\{\ba{\lambda }_{1},\dots,\ba{\lambda }_{r}\}\To H^{2}(\D)$ defined by 
\begin{align*}
 \forall\,\lambda \in\T\setminus\{\lambda _{1},\dots,\lambda _{r}\},\qquad &F(\lambda )=\ba{p(\lambda )}\,\sum_{l\in\Lambda }v_{l}(\lambda )\,k_{\varphi ^{-1}_{l}(\lambda )}
\intertext{and}
 \forall\,\lambda \in\T\setminus\{\ba{\lambda} _{1},\dots,\ba{\lambda} _{r}\},\qquad &E(\lambda )=F(\ba\lambda ).
\end{align*}
Observe that the quantity $k_{\varphi ^{-1}_{l}(\lambda )}$ in the expression of $F(\lambda)$ above only makes sense when $\lambda$ belongs to $\Omega'_l$. But since $v_l$ is supported on $\Omega'_l$, the function $\lambda\mapsto v_l(\lambda)k_{\varphi ^{-1}_{l}(\lambda )} $ extends into a $\mathcal{C}^{\infty}$ function on $\T\setminus\{{\lambda} _{1},\dots,{\lambda} _{r}\}$ by defining it to be zero outside the set $\Omega'_l$, so that $F$ is a well-defined $\mathcal{C}^{\infty}$ function on $\T\setminus\{{\lambda} _{1},\dots,{\lambda} _{r}\}$.
\par\smallskip
For every $\lambda \in\T\setminus\{\lambda _{1},\dots,\lambda _{r}\}$ we have
$$M^{*}_{\varphi} F(\lambda )=\ba{p(\lambda )}\,\displaystyle\sum_{l\in\Lambda }v_{l}(\lambda )\ba{\varphi \bigl( \varphi ^{-1}_{l}(\lambda )\bigr)}\,k_{\varphi ^{-1}_{l}(\lambda )}$$ so that $M^{*}_{\varphi }F(\lambda )=\ba{\lambda }\,F(\lambda )$. It follows that 
$E:\T\setminus\{\ba{\lambda }_{1},\dots,\ba{\lambda }_{r}\}\To H^{2}(\D)$ is a unimodular eigenvectorfield for $M^{*}_{\varphi }$. 
\par\smallskip
Our aim is to show that the assumptions of Theorem \ref{Theo5} are satisfied. Assumptions (i) and (ii) are easy to check: suppose that $f\in H^{2}(\D)$ is such that $\pss{f}{E(\lambda )}=0$ for every $\lambda \in B$ where $B\subseteq\T\setminus\{\ba{\lambda} _{1},\dots,\ba{\lambda} _{r}\}$ is a Borel subset of $\T$ of full Lebesgue measure. Then  by continuity $$\ba{p(\lambda )}\,\displaystyle\sum_{l\in\Lambda }v_{l}(\lambda )\,f\bigl( \varphi ^{-1}_{l}(\lambda )\bigr)=0 \quad \textrm{ for every } \lambda \in\T\setminus\{\lambda_{1},\dots,\lambda_{r}\},$$ so that $$\displaystyle\sum_{l\in\Lambda }v_{l}(\lambda )\,f\bigl( \varphi ^{-1}_{l}(\lambda )\bigr)=0 \quad \textrm{ for every } \lambda \in\T\setminus\{\lambda_{1},\dots,\lambda_{r}\}.$$ 
It follows that for every $j\in\{1,\ldots, r\}$ and every $\lambda \in\Gamma'_{j}\setminus\{\lambda _{j}\} $, $f\bigl( \varphi ^{-1}_{l_{j}}(\lambda )\bigr)=0$. Since $\varphi ^{-1}_{l_{j}}(\Gamma'_{j}\setminus\{\lambda _{j}\}  )$ has accumulation points in $\D$, $f=0$, and $\textrm{span}\,\{ E(\lambda );\ \lambda \in B\}$ is dense in $H^{2}(\D)$. Assumption (i) is 
thus satisfied. Assumption (ii) clearly holds true: if $f\equiv 1$, then $\pss{f}{E(\lambda )}=\ba{p(\ba{\lambda })}=q(\lambda )$ for every $\lambda \in\T\setminus\{\ba{\lambda} _{1},\dots,\ba{\lambda} _{r}\}$, where $q({e^{i\theta }})=\sum_{n\in\Z}\ba{\hat{p}(n)}e^{in\theta }$. The main difficulty is to check that the series $\sum_{n\in\Z}\widehat{E}(n)$, or equivalently the series $\sum_{n\in\Z}\widehat{F}(n)$, is unconditionally convergent. Since $H^{2}(\D)$ does not contain a copy of $c_{0}$, it suffices to prove that for every $f\in H^{2}(\D)$, the series $\sum_{n\in\Z}|\pss{f}{\widehat{F}(n)}|$ is convergent. 
\par\smallskip
We are going to show that for every $f\in H^{2}(\D)$, the function $\phi _{f}$ defined on 
$\T\setminus\{\lambda _{1},\dots,\lambda _{r}\}$ by $$\phi _{f}(\lambda )=\pss{f}{F(\lambda )}=\ba{p(\lambda )}\,\displaystyle\sum_{l\in\Lambda }v_{l}(\lambda )\,f\bigl( \varphi ^{-1}_{l}(\lambda )\bigr)
$$ extends into a function of class $\mathcal{C}^{1}$ on $\T$. Bernstein's Theorem will then imply that the series
$\sum_{n\in\Z}|\pss{f}{\widehat{F}(n)}|$ is convergent.
\par\smallskip
 We have seen that the function $F$ is of class $\mathcal{C}^{\infty }$ on $\T\setminus\{\lambda _{1},\dots,\lambda _{r}\}$. Since, for every $l\in\Lambda $, $\sup\bigl\{|\varphi ^{-1}_{l}(\lambda )|;\ \lambda \in\Omega'_{l} \bigr\}\le \rho<1$, the quantity $\sup\bigl\{\bigl| \bigl| k_{\varphi ^{-1}_{l}(\lambda )}\bigr|\bigr|;\ \lambda \in\Omega' _{l}\bigr\}$ is finite for every $l\in\Lambda $. Now $p(\lambda _{j})=0$ for every $j\in\{1,\dots,r\}$, and since each function $v_{l}$, $l\in\Lambda$, is uniformly bounded on $\T$ it follows that $F$ can be extended into a continuous map on $\T$ by setting $F(\lambda _{j})=0$ for every $j\in\{1,\dots,r\}$. So $\phi _{f}$ is actually continuous on $\T$ for every $f\in H^{2}(\D)$. Let us now compute the derivative of $\phi _{f}$. Writing, for $0\le \theta < 2\pi $, $\lambda =e^{i\theta }$ and 
\begin{align*}
 \phi _{f}( e^{i\theta })=&\ba{\vphantom{\bigl( }p( e^{i\theta })}\,\sum_{l\in\Lambda }v_{l}( e^{i\theta })\,f\bigl(\varphi ^{-1}_{l}( e^{i\theta }) \bigr)
\intertext{we have}
\dfrac{d\phi _{f}}{d\theta }(e^{i\theta })=&-i\,e^{-i\theta }\,\ba{\vphantom{\bigl( }p'( e^{i\theta })}\,\sum_{l\in\Lambda }v_{l}( e^{i\theta })\,f\bigl(\varphi ^{-1}_{l}( e^{i\theta }) \bigr)\\&+
\ba{\vphantom{\bigl( }p( e^{i\theta })}\,\sum_{l\in\Lambda }ie^{i\theta }\,\dfrac{dv_{l}}{d\theta }(e^{i\theta })\,f\bigl(\varphi ^{-1}_{l}( e^{i\theta }) \bigr)\\
&+\ba{\vphantom{\bigl( }p( e^{i\theta })}\,\sum_{l\in\Lambda }v_{l}( e^{i\theta })\,(i\,e^{i\theta })^{2}\dfrac{1}{\varphi '\bigl( \varphi ^{-1}_{l}(e^{i\theta })\bigr)}\,f'\bigl( \varphi ^{-1}_{l}(e^{i\theta })\bigr)
\end{align*}
(the notation $\dfrac{d}{d\theta }$ is used for the derivative of a function on $\T$ with respect to the real variable $\theta $, while the sign ' is used for the complex derivative of a holomorphic function).
\par\medskip 
The same argument as above, using the facts that $\Lambda $ is finite, that the functions $v_{l}$ are of class $\mathcal{C}^{1}$ on $\T$ (and hence have uniformly bounded derivatives on $\T$), and that $p(\lambda _{j})=p'(\lambda _{j})=0$ for every $j\in\{1,\dots,r\}$, shows that the first two terms in this expression tend to $0$ as $\lambda =e^{i\theta }\in\T\setminus\{\lambda _{1},\ldots, \lambda _{r}\}$ tends to $\lambda _{j}$, $j\in\{1,\dots,r\}$. So it remains to deal with the last term. 
If $\lambda \in\T\setminus\{\lambda _{1},\ldots, \lambda _{r}\}$ tends to $\lambda _{j}$ for some $j\in\{\lambda _{1},\ldots, \lambda _{r}\}$, we can suppose without loss of generality that $\lambda $ belongs to $\Omega '_{l_{j}}$. The third term in the expression above is then equal to
$$\ba{p(\lambda )}\dfrac{1}{\varphi '\bigl( \varphi ^{-1}_{l_{j}}(\lambda )\bigr)}f'(\varphi ^{-1}_{l_{j}}(\lambda )), \quad \textrm{where } \lambda =e^{i\theta }.$$
As $\sup\bigl\{\bigl|f'\bigl( \varphi ^{-1}_{l_{j}}(\lambda )\bigr)\bigr|; \ \lambda\in\Omega' _{l_{j}} \bigr\}$ is finite,
it suffices to show that 
$$\dfrac{\ba{p(\lambda )}}{\varphi '\bigl( \varphi ^{-1}_{l_{j}}(\lambda )\bigr)}\to 0\quad \textrm{ as } \lambda \to\lambda _{j},\; \lambda \in\Omega '_{l_{j}}.$$
We have seen that for every $z\in D(z_{j}\varepsilon _{j})$, $\varphi '(z)=(z-z_{j})^{m_{j}}\sigma _{j}(z)$, where $\sigma _{j}$ is an holomorphic function which does not vanish on $D(z_{j},\varepsilon _{j})$. Hence the quantity
\[
\sup\Bigl\{\dfrac{|z-z_{j}|^{m_{j}}}{|\varphi '(z)|};\ z\in D(z_{j},\varepsilon _{j})\setminus\{z_{j}\}\Bigr\}
\]
is finite. So there exists a positive constant $C$ such that
\[
\sup\Bigl\{\dfrac{\bigl|\varphi ^{-1}_{l_{j}}(\lambda )-z_{j}\bigr|^{m_{j}}}{\bigl|\varphi '\bigl( \varphi ^{-1}_{l_{j}}(\lambda )\bigr)\bigr|};\ \lambda \in\Omega _{l_{j}}\Bigr\}\le C.
\]
We have
$$p(\lambda )=p\bigl( \varphi \bigl( \varphi ^{-1}_{l_{j}}(\lambda )\bigr)\bigr)=p\bigl( \lambda _{j}+\bigl( \varphi ^{-1}_{l_{j}}(\lambda )-z_{j}\bigr)^{d_{j}}\psi _{j}\bigl( \varphi ^{-1}_{l_{j}}(\lambda )\bigr)\bigr)$$ for every $\lambda \in\Omega _{l_{j}}$, and thus
$p(\lambda )=\bigl( \varphi ^{-1}_{l_{j}}(\lambda )-z_{j}\bigr)^{2d_{j}}q_{j}\bigl( \varphi ^{-1}_{l_{j}}(\lambda )\bigr)$ for every $\lambda \in\Omega _{l_{j}}$, where $q_{j}$ is an holomorphic function which is bounded on $\varphi ^{-1}_{l_{j}}(\Omega _{l_{j}})$. There exists hence a positive constant $C'$ such that 
\begin{align*}
 \biggl|\dfrac{p(\lambda )}{\varphi '\bigl( \varphi ^{-1}_{l_{j}}(\lambda )\bigr)}\biggr|&=
\biggl|\dfrac{\bigl( \varphi ^{-1}_{l_{j}}(\lambda )-z_{j}\bigr)^{2d_{j}}}{\bigl( \varphi ^{-1}_{l_{j}}(\lambda )-z_{j}\bigr)^{m_{j}}}\biggr|\,.\, \bigl|q_{j}\bigl( \varphi ^{-1}_{l_{j}}(\lambda )\bigr)\bigr|\,.\,
\biggl|\dfrac{\bigl( \varphi ^{-1}_{l_{j}}(\lambda )-z_{j}\bigr)^{m_{j}}}{\varphi '\bigl( \varphi ^{-1}_{l_{j}}(\lambda )\bigr)}\biggr|\\
&\le C'\,.\,\bigl|\varphi ^{-1}_{l_{j}}(\lambda )-z_{j}\bigr|^{2d_{j}-m_{j}}\qquad \textrm{for every } \lambda \in\Omega _{l_{j}}.
\end{align*}
The righthand bound
tends to $0$ as $\varphi ^{-1}_{l_{j}}(\lambda )$ tends to $z_{j}$ since $2d_{j}=2(m_{j}+1)>m_{j}$. We now claim that if $\lambda\in \Omega _{l_{j}} $ tends to $\lambda _{j}$, then $\varphi ^{-1}_{l_{j}}(\lambda )$ tends to $z_{j}$. Indeed, since $\Omega _{l_{j}}=U_{j,\,0}$, $\varphi ^{-1}_{l_{j}}(\lambda )=\varphi _{j,\,0}^{-1}(\lambda )=
\alpha _{j}^{-1}\bigl( \gamma _{j,\,0}(\lambda -\lambda _{j})\bigr)$. If $\lambda \to \lambda _{j}$, then 
$\gamma_{j,\,0}(\lambda -\lambda _{j})\to 0$. Now the map $\alpha _{j}:D(z_{j},\varepsilon _{j})\To D(0,\delta _{j})$ is a biholomorphism such that $\alpha _{j}(z_{j})=0$. Thus $\alpha _{j}^{-1}\bigl( \gamma_{j,\,0}(\lambda -\lambda _{j})\bigr)\to z_{j}$ as $\lambda \to\lambda _{j}$, $\lambda \in\Omega _{l_{j}}$, and this proves our claim.
\par\medskip 
So we have proved that if $\lambda \in\Omega _{l_{j}}$ tends to $\lambda _{j}$, then 
\[
\biggl|\dfrac{\ba{\vphantom{\bigl( }p(\lambda )}} {\varphi '\bigl( \varphi ^{-1}_{l_{j}}(\lambda )\bigr)}\biggr|=\biggl|\dfrac{p(\lambda )} {\varphi '\bigl( \varphi ^{-1}_{l_{j}}(\lambda )\bigr)}\biggr|\quad \textrm{tends to}\ 0.
\]
As explained above,  this shows that the third term in the expression of $\dfrac{d\phi _{f}}{d \theta }(e^{i\theta })$ given above tends to $0$ as $d\bigl( e^{i\theta } ,\{\lambda _{1},\dots,\lambda _{r}\}\bigr)$ tends to $ 0$ with $e^{i\theta }\in\T\setminus\{\lambda _{1},\ldots,\lambda _{r}\}$. So $\phi _{f}$ is a map of class $\mathcal{C}^{1}$ on $\T$, and this finishes the proof of Example \ref{Example2}. 
\end{proof}
\par\smallskip
\begin{example}\label{Example3}
Using the notation of Example \ref{Example1}, the operator $\alpha B^{2}$ 
is universal on $\lp{p}{\N}$ or $\co{\N}$ for any $\alpha $ with $|\alpha |>1$. 
\par\smallskip
This can be 
proved in several ways. One of these is to observe that $\alpha B^{2}$ 
admits two eigenvectorfields $E_{1}$ and $E_{2}$ which are analytic on 
$D(0,|\alpha |)$:
\[
E_{1}(\lambda )=\sum_{n\ge 0}\biggl(\dfrac{\lambda }{\alpha } 
\biggr)^{\!\!n}
e_{2n+1}\quad\textrm{and}\quad
E_{2}(\lambda )=\sum_{n\ge 0}\biggl(\dfrac{\lambda }{\alpha } 
\biggr)^{\!\!n}
e_{2n}\cdot 
\]
We have $\overline{\vphantom{'}\textrm{span}}\,\bigl[E_{1}(\lambda ),
E_{2}(\lambda );\ \lambda \in\T\bigr]=Z$. Using the generalization of 
Theorem \ref{Theo6} following from Theorem \ref{Theo1-4}, we obtain that 
$\alpha B^{2}$ is universal for ergodic systems. 
\par\smallskip 
If one is interested only in the universality of $\alpha B^{2}$ for 
invertible ergodic systems, one can use simply Theorem \ref{Theo5} and 
the following argument: let $\varphi $ be a 
function of class $\mathcal{C}^{\infty }$ on $\T$ such that $0\le \varphi \le 1$ and the support of $\varphi 
$ is a non-trivial closed sub-arc $\Gamma $ of $\T$. Consider the 
$\mathcal{C}^{\infty }$ eigenvectorfield of $\alpha B^{2}$ defined by
$$E(\lambda )=\varphi (\lambda )\,E_{1}(\lambda )+(1-\varphi (\lambda ))\,
E_{2}(\lambda ),\quad \lambda \in\T.$$ Let us show that 
$\overline{\vphantom{'}\textrm{span}}\,\bigl[E(\lambda );\ \lambda 
\in\T\bigr]=Z$: if $x^{*}$ is a functional such that $\pss{x^{*}}{E(\lambda 
)}=0$ for every $\lambda \in\T$, then 
$(1-\varphi (\lambda ))\,\pss{x^{*}}{E_{2}(\lambda )}=\pss{x^{*}}{E_{2}(\lambda )}=0$ for every $\lambda 
\in\T\setminus\Gamma  $. Hence by analyticity of $E_{2}$,
$\pss{x^{*}}{E_{2}(\lambda )}=0$ for every $\lambda \in\T$. It follows that 
$\varphi (\lambda )\,\pss{x^{*}}{E_{1}(\lambda )}=0$ for every $\lambda 
\in\Gamma $, and the same argument shows that $\pss{x^{*}}{E_{1}(\lambda 
)}=0$ for every $\lambda \in\T$. Since 
$\overline{\vphantom{'}\textrm{span}}\,\bigl[E_{1}(\lambda ),E_{2}(\lambda 
);\ \lambda 
\in\T\bigr]=Z$, $x^{*}=0$, and thus 
$\overline{\vphantom{'}\textrm{span}}\,\bigl[E(\lambda );\ \lambda 
\in\T\bigr]=Z$. Moreover, $\pss{e_{0}^{*}+e_{1}^{*}}{E(\lambda )}=1$ for 
every $\lambda \in\T$ (where $e_{i}^{*}$ is the functional on $\ell_{p}(\N)$ or $c_{0}(\N)$ mapping a vector $x$ of the space on its $i^{th}$ coordinate). Lastly, $E$ being of class $\mathcal{C}^{\infty }$ 
on $\T$, the series $$\sum_{n\in\Z}||\widehat{E}(n)||$$ is convergent. So 
Theorem \ref{Theo5} applies. If we consider $\alpha B^{2}$ as acting on 
$\lp{2}{\N}$, Example \ref{Example2} applies directly since $\alpha B^{2}$ 
is unitarily similar to $M_{\varphi }^{*}$ where $\varphi (z)=\overline{\alpha} 
z^{2}$,
$z\in\D$.
\end{example}
\par\smallskip 
Many other examples can be obtained along these lines (such as $\bigoplus
_{\ell_{p}}\alpha B$ on the infinite direct sum $\bigoplus_{\ell_{p}}\ell_{p}$, $1\le 
p<+\infty $, $|\alpha| >1$,\,\dots).
\par\smallskip 
\par\smallskip
One can observe that all the universal operators presented until now admit 
 eigenvectorfields which are analytic in a neighborhood of some points of $\T$, and in particular have a rather large spectrum. So one 
may naturally wonder whether this condition is necessary for $A$ to be 
universal. Our last example, which is rather unexpected, shows that it is 
not the case. It is to be found within the class of Kalish-type operator 
on $\LL{2}(\T)$ (see \cite{BM} for details).
\par\smallskip
\begin{example}\label{Example4} Let $A$ be the operator defined on $\LL{2}(\T)$
by setting for every $f\in\LL{2}(\T)$ and every $0\le\theta <2\pi $,
\[
Af(e^{i\theta })=e^{i\theta }f(e^{i\theta })-\int_{0}^{\theta 
}ie^{it}f(e^{it})dt.
\]
Then $A$ is universal for invertible ergodic systems.
\end{example}

\par\smallskip 
It is not difficult to check that for every $\lambda =e^{i\theta }\in\T$, 
$0\le\theta<2\pi  $, $\ker(A-\lambda )=\textrm{span}\bigl[E(\lambda ) 
\bigr]$ where $E(\lambda )=\chi _{(\lambda ,1)}$. Here
$\chi _{(\lambda ,1)}$ denotes the indicator function of the arc 
$\Gamma _{\lambda }=\bigl\{e^{i\tau };\ \theta <\tau <2\pi \bigr\}$. The 
eigenvectorfield $E$ is $1/2$-H\"{o}lderian on $\T$. Since the spectrum of 
$A$ coincides with $\T$, $A$ does not admit any eigenvectorfield  which is analytic in 
a neighborhood of some point of $\T$.
\par\smallskip 
\begin{proof}[Proof of Example \ref{Example4}] The most obvious idea is to try to apply Theorem 
 \ref{Theo5}. Setting 
 \[
y_{n}(e^{it})=\int_{0}^{2\pi }e^{-in\theta 
}E(e^{i\theta })(e^{it})\dfrac{d\theta }{2\pi }
\]
for every $n\in\Z$, we have
\[
y_{0}(e^{it})=\dfrac{t}{2\pi }\quad\textrm{and}\quad y_{n}(e^{it})
=\dfrac{1-e^{-int}}{2i\pi n}\quad \textrm{for every}\ 
n\in\Z\setminus\{0\}\cdot
\]
The series $\sum_{n\in\Z} y_{n}$ is obviously not unconditionally 
convergent in $\LL{2}(\T)$, and Theorem \ref{Theo5} cannot be applied this 
way. So we consider instead of $E$ the unimodular eigenvectorfield $F$ 
defined by $F(\lambda )=(1-\lambda )E(\lambda )$, and we set 
\[
z_{n}(e^{i\theta })=\int_{0}^{2\pi }e^{-in\theta }F(e^{i\theta })(e^{it })
\dfrac{d\theta }{2\pi }=
\int_{0}^{2\pi }\bigl(e^{-in\theta }-e^{-i(n-1)\theta } 
\bigr)E(e^{i\theta })(e^{it })\dfrac{d\theta }{2\pi }\cdot 
\]
Thus $z_{n}=y_{n}-y_{n-1}$ for every $n\in\Z$. We have for every 
$n\in\Z\setminus\{0,1\}$
\[
z_{n}(e^{it})=\dfrac{1}{2i\pi
}\biggl(\dfrac{e^{-i(n-1)t}}{n-1}-\dfrac{e^{-int}}{n}-\dfrac{1}{n(n-1)} 
\biggr)\cdot 
\]
Since the series $\sum_{|n|\ge 2}1/(n(n-1))$ is convergent, and the series
$\sum_{|n|\ge 2}e^{-int}/n$ is unconditionally convergent in $\LL{2}(\T)$, 
it follows that the series $\sum_{|n|\ge 2}z_{n}$ (and hence the series 
$\sum_{n\in\Z} z_{n}$) is unconditionally convergent in $\LL{2}(\T)$.
\par\smallskip 
Let us now check that the functions $z_{n}$, $n\in\Z$, span a dense 
subspace of $\LL{2}(\T)$. Suppose that $f\in\LL{2}(\T)$ is such that 
$\pss{f}{z_{n}}=0$ for every $n\in\Z$. Then 
\begin{align*}
\dfrac{\widehat{f}(n)}{n}&=\dfrac{\widehat{f}(n-1)}{n-1}+\widehat{f}
(0)\biggl( \dfrac{1}{n}-\dfrac{1}{n-1}\biggr)\quad \textrm{for every }n\not\in\{0,1\}.
\intertext{Summing these equalities for $n\ge 2$, we obtain}
\sum_{n\ge 2}\dfrac{\widehat{f}(n)}{n}&=\widehat{f}(1)+\sum_{n\ge 2}
\dfrac{\widehat{f}(n)}{n}-\widehat{f}(0), \textrm{ so that }
\widehat{f}(1)=\widehat{f}(0).
\intertext{As}
\widehat{f}(n)&=\dfrac{n}{n-1}\widehat{f}(n-1)-\dfrac{1}{n-1}\widehat{f}(0)
\end{align*}
 for every $n\ge 2 $,
 $\widehat{f}(2)=2\widehat{f}(1)-\widehat{f}(0)=\widehat{f}(0)$. By 
induction $\widehat{f}(n)=\widehat{f}(0)$ for every $n\ge 2$, so that 
$\widehat{f}(n)=0$ for every $n\ge 0$. So $$\widehat{f}(n-1)=\frac{n-1}{n}
\widehat{f}(n)\quad \textrm{ for every }
n\ge -1,$$  which implies that $\widehat{f}(-n)=n\widehat{f}(-1)$ for every $n\ge 
1$. Hence $f=0$. So $\overline{\vphantom{'}\textrm{span}}\,\bigl[z_{n};\ 
n\in\Z
\bigr]=\LL{2}(\T)$ (we have actually proved that 
$\overline{\vphantom{'}\textrm{span}}\,\bigl[z_{n};\ 
n\in\Z\setminus\{0,1\}
\bigr]=\LL{2}(\T)$
\!). 
\par\smallskip 
Lastly, observe that if we set $f_{0}(e^{i\theta })=e^{i\theta }$, 
then $\pss{f_{0}}{z_{n}}=0$ for every $n$ with $|n|\ge 2$, 
$\pss{f_{0}}{z_{1}}=-\pss{f_{0}}{z_{-1}}=-1/2i\pi $
and a simple computation shows that $\pss{f_{0}}{z_{0}}=0$. So 
$\pss{f_{0}}{F(e^{i\theta })}=\bigl(e^{i\theta }-e^{-i\theta } \bigr)/2i\pi $, and all the 
assumptions of Theorem \ref{Theo5} (or, directly, Theorem \ref{Theo1}) are 
satisfied. So $A$ is universal for invertible ergodic systems.
\end{proof}

\section{Miscellaneous results and comments}\label{Sec5}

We begin this section by investigating necessary conditions for an operator to be 
universal. We have already seen some such necessary conditions: a 
universal operator must be frequently hypercyclic, and topologically 
mixing. If it is universal for all ergodic systems, it cannot be 
invertible. Without any additional assumption, it seems difficult to say 
more. But looking at the examples of universal operators presented in 
Sections \ref{Sec3} and \ref{Sec4}, we observe that all of them admit 
continuous unimodular eigenvectorfields, and that every $\lambda 
\in\T$ is an eigenvalue. This is not completely a coincidence: if $A\in
\bb(Z)$ satisfies the assumptions of Theorem \ref{Theo1}, then 
$A$ admits a continuous unimodular eigenvectorfield defined as 
$$E(\lambda )=\sum_{n\in\Z}\lambda ^{n}A^{-n}z_{0}.$$ If $A$ satisfies the 
assumptions of Theorem \ref{Theo1bis} with $r=1$ for instance, then $A$ admits a continuous 
eigenvectorfield on $\overline{\D}$ defined as $E(\lambda 
)=\sum_{n\in\N}\lambda ^{n}A^{-n}z_{0}$. In both cases assumption (b) implies that 
all $\lambda \in\T$ except possibly finitely many are eigenvectors of $A$. 
Indeed, there exists a finite subset $F$ of $\Z$ and a \nz\ functional $z_{0}^{*}\in Z^{*}$ such that
$$\pss{z_{0}^{*}}{E(\lambda )}=\sum_{n\in F}\lambda ^{n}\pss{z_{0}^{*}}{A^{-n}z_{0}}.$$
The function $\pss{z_{0}^{*}}{E(\,.\,)}$ is a \nz\ trigonometric polynomial, so it has only finitely many zeroes on $\T$. This implies that $E(\lambda )$ is \nz\ for all $\lambda \in\T$ except possibly finitely many.
\par\smallskip
So it comes as a natural question to ask whether a universal operator 
necessarily admits some (or many) unimodular eigenvectors. It is possible 
to answer this question in the affirmative under two additional 
assumptions: first that the operator lives on a Hilbert space, and, second, 
that we add in the definition of the universality the requirement that all 
the measures $\nu $ on $Z$ involved in the definition have a moment of 
order $2$.
\begin{definition}\label{Definition10}
 A bounded operator $A$ on $Z$ is said to be $2$-universal for 
(invertible) ergodic systems if for every (invertible) ergodic dynamical 
system $(X,\bb,\mu ;T)$ there exists a Borel probability measure $\nu $ on 
$Z$ which is $A$-invariant, has full support, has a moment of order $2$
(i.e.\ $\int_{Z}||z||^{2}d\nu (z)<+\infty$), and is such that the two dynamical 
systems $(X,\bb,\mu ;T)$ and $(Z,\bb_{Z},\nu ;A)$ are isomorphic.
\end{definition}
We have seen in Section \ref{Sec2} that operators satisfying the 
assumptions of Theorems \ref{Theo1}, \ref{Theo1bis}, \ref{Theo1-3} or 
\ref{Theo1-4} are $2$-universal. Thus all the universal operators 
presented in Sections \ref{Sec3} and \ref{Sec4} above are $2$-universal. A 
first necessary condition for an operator on a Hilbert space to be 
$2$-universal is 
\begin{proposition}\label{Proposition11}
 Let $H$ be a complex separable infinite-dimensional Hilbert space, and 
let $A$ be a $2$-universal operator on $H$. Then $A$ admits a perfectly 
spanning (and even a $\mathcal{U}_{0}$-perfectly spanning) set of unimodular eigenvectors.
\end{proposition}
\begin{proof}
 The proof is an easy consequence of some results of \cite{BG2} and \cite{BM2}
concerning the ergodic theory of linear dynamical systems (see also 
\cite{BM}). Since $A$ is $2$-universal, it is weakly mixing with respect 
to some probability measure $\nu $ with full support which has a moment of 
order $2$. Consider the centered Gaussian probability measure $m$ on $H$ whose 
covariance operator $S$ is given by
\[
\pss{Sx}{y}=\int_{H}\pss{x}{z}\,\overline{\pss{y}{z}}\,d\nu 
(z)\quad\textrm{for every}\ x,y\in H.
\]
Then
\[
\int_{H}\pss{x}{z}\,\overline{\pss{y}{z}}\,dm(z)=\pss{Sx}{y}=\int_{H}\pss{x
}{z}\,\overline{\pss{y}{z}}\,d\nu(z)
\]
for every $x,y\in H$. Hence $m$ is $A$-invariant and has full support. 
Moreover, since $A$ is weakly mixing with respect to $\nu $,
\[
\dfrac{1}{N}\sum_{n=0}^{N-1}\,\,\biggl|\int_{H}\pss{x}{A^{n}z}\,\overline{
\pss { y }{z } }\,d\nu (z)\biggr|^{2}=\dfrac{1}{N}\sum_{n=0}^{N-1}\,\,\biggl|\int_{H}\pss{x}{A^{n}z}\,\overline{
\pss { y }{z } }\,dm (z)\biggr|^{2}
\]
tends to $0$ as $N$ tends to infinity for every $x,y\in H$,
and this implies (see \cite{BG1} or \cite{BM} for details) that $A$ is 
weakly mixing with respect to $m$. Since it is proved in \cite{BG2} that 
any operator on a space of cotype $2$ which is weakly mixing  with respect 
to some Gaussian measure with full support has perfectly spanning 
unimodular eigenvectors, the first part of the result follows. The second part is proved in exactly the same way: $A$ is necessarily strongly mixing \wrt\ some probability measure on $Z$ with full support, and is hence
strongly mixing \wrt\ some Gaussian probability measure on $Z$ with full support.
It is proved in \cite{BM2} that this implies that the unimodular eigenvectors of $A$ are $\mathcal{U}_{0}$-perfectly spanning (i.e. for any Borel set $B\subseteq \T$ which is a set of extended uniqueneness,
$\overline{\vphantom{'}\textrm{span}}\,\bigl[\ker(T-\lambda );\, \lambda \in\T\setminus B\bigr]=H$. See \cite{BM2} for more about these questions).
\end{proof}
\par\smallskip
If $A$ is supposed to be $2$-universal for ergodic systems (and not only 
for invertible ones), we can moreover prove that the unimodular point 
spectrum of $A$ is a subset of $\T$ of full Lebesgue measure. These are 
the contents of Theorem \ref{Theo7}, which we now prove.
\par\smallskip 
\begin{proof}[Proof of Theorem \ref{Theo7}]
Let $A$ be a $2$-universal 
operator on $H$ for ergodic systems. Consider the dynamical system $T$ 
defined on $\bigl([0,1],\bb_{[0,1]},dx \bigr)$, where $dx$ is the Lebesgue 
measure on $[0,1]$, by $Tx=2x\mod 1$. Then $T$ is strongly mixing, and 
has the following property of decay of correlations: for any $f,g\in\LL
{2}\bigl([0,1] \bigr)$ and $n\ge 0$, define the $n$-th correlation between 
 $f$ and $g$ as 
\[
\mathcal{C}_{n}(f,g)=\int_{0}^{1}f(T^{n}x)\,\overline{\vphantom{'}g(x)}\,dx-
\biggl(\int_{0}^{1}f \biggr)\,.\,\overline{\biggl(\int_{0}^{1}g \biggr)}.
\]
Then we have for every $f\in\LL{2}\bigl([0,1] \bigr)$,
$g\in\mathcal{C}^{1}\bigl([0,1] \bigr)$, and $n\ge 0$
\[
|\,\mathcal{C}_{n}(f,g)|\le 2^{-n}\,\dfrac{||f||_{2}\,.\,||g'||_{\infty 
}}{\sqrt{3}}\cdot
\]
Thus the correlations decay exponentially fast provided one of the two 
functions $f$ and $g$ is sufficiently smooth.
\par\smallskip 
Since $A$ is $2$-universal for ergodic systems, there exists an 
$A$-invariant measure $\nu $ on $H$ with full support and with a moment of 
order $2$, for which there exists an isomorphism $\Phi $ between the two 
dynamical systems $\bigl([0,1],\bb_{[0,1]},dx;T \bigr)$ and $\bigl(H,
\bb_{H},\nu ;A\bigr)$. We denote by $\LL{2}_{0}([0,1])$ (resp. 
$\LL{2}_{0}(H,\bb_{H},\nu )$) the set of functions $f\in 
\LL{2}([0,1])$ (resp. $F\in \LL{2}(H,\bb_{H},\nu )$) such that 
$\int_{0}^{1}f(x)\,dx=0$ (resp. $\int_{H}F(z)\,d\nu (z)=0$). For every 
functions $F,G\in\LL{2}_{0}(H,\bb_{H},\nu )$ such that $G=g\circ\Phi 
^{-1}$ for some function $g\in\mathcal{C}^{1}\bigl([0,1] \bigr)
\cap\LL{2}_{0}\bigl([0,1] \bigr)$, there exists a positive constant $c(F,G)$ 
such that if we denote for every $n\ge 0$ by $C_{n}(F,G)$ the correlation
\[
C_{n}(F,G)=\int_{Z}F(A^{n}z)\,\overline{\vphantom{'}G(z)}\,d\nu (z),
\] 
then
$\bigl|C_{n}(F,G) \bigr|\le c(F,G)\,2^{-n}$. Since the measure $\nu $ has 
a moment of order $2$, one can consider its covariance operator $S$ on $H$ defined as:
\[
\pss{Sx}{y}=\int_{Z}\pss{x}{z}\overline{\vphantom{'}\pss{y}{z}}\,
d\nu (z) 
\]
for every $x,y\in H$. The operator $S$ is self-adjoint, positive, and of 
trace class. Since $\nu $ has full support, $S$ has dense range. Hence 
there exist an orthonormal basis $(e_{l})_{l\ge 1}$ of $H$ and a sequence
$(\sigma ^{2}_{l})_{l\ge 1}$ of positive numbers with $\sum_{l\ge 1} \sigma  _{l}^{2}<+\infty$ such that $Se_{l}
=2\sigma ^{2}_{l}e_{l}$ for every $l\ge 1$. 
Also it follows from the orthogonality of the vectors $e_{l}$ in $H$ that 
the functions $\pss{e_{l}}{\,.\,}$ are orthogonal in $\LL{2}(H,\mathcal{B}_{H},\nu )$.
Let $\mathcal{E}=\overline{\vphantom{'}\textrm{span}}\,\bigl[
\pss{e_{l}}{\,.\,};\ l\ge 1\bigr]$, where the closed linear span is taken in 
$\LL{2}(H,\mathcal{B}_{H},\nu )$: $\mathcal{E}$ is a closed subspace of 
$\LL{2}(H,\mathcal{B}_{H},\nu )$ which consists of all functions $F\in\mathcal{E}$ which can be written as a 
convergent series in $\LL{2}(H,\mathcal{B}_{H},\nu )$ of the form
\[
F=\sum_{l\ge 1} a_{l}\,\pss{e_{l}}{\,.\,}, \quad \textrm{where}\quad 
\sum_{l\ge 1}
|a_{l}|^{2}\,\sigma _{l}^{2}<+\infty .
\]
Remark that the function $\pss{x}{\,.\,}$ belongs to $\mathcal{E}$ for 
every $x\in H$. We denote by $\iota $ the injection operator $\iota :H\To \mathcal{E}$ defined by $\iota (x)=\pss{x}{\,.\,}$. If $U_{A}$ denotes the Koopman operator associated to 
$(H,\bb_{H},\nu ;A)$, then $U_{A}(\mathcal{E})\subseteq \mathcal{E}$.
Proceeding as in the proof of \cite[Th. 4.1]{BG2}, we apply the spectral 
decomposition theorem to $U_{A}$, which is an isometry on $\LL{2}
(H,\bb_{H},\nu )$: there exists a finite or countable family 
$(H_{i})_{i\in I}$ of Hilbert spaces, with either $H_{i}=H^{2}(\T)$ or
$H_{i}=\LL{2}(\T,\sigma _{i})$ for some probability measure on $\T$, and 
an invertible isometry $J:\bigoplus_{i\in I}H_{i}\longrightarrow \LL{2}
(H, \mathcal{B}_{H},\nu )$ such that $U_{A}J=JM$. Here $M$ acts on $\bigoplus_{i\in 
I}H_{i}$ as $M=\bigoplus_{i\in I}M_{i}$, where $M_{i}$ is the 
multiplication operator by $\lambda $ on $H_{i}$: $(M_{i}f_{i})(\lambda )=
\lambda f_{i}(\lambda )$ for every $f_{i}\in H_{i}$.
\par\smallskip
Let now $K:\bigoplus_{i\in I}H_{i}\longrightarrow H$ be the operator defined as $K=\iota ^{*}J$. 
For 
every $x\in H$, we have
\[
\pss{Sx}{x}=\int_{H}|\pss{x}{z}|^{2}\,d\nu (z)=||\iota (x)||^{2}=||K^{*}x||^{2}.
\]
It follows that $K^{*}$ is a Hilbert-Schmidt operator, and so there exists 
for every $i\in I $ a unimodular eigenvectorfield
$E_{i}\in \LL{2}(\T,\sigma _{i};H)$ such that $K^{*}x=\bigoplus_{i\in I}
\pss{x}{E_{i}(\,.\,)}$ (see \cite{VTC}, \cite{BG2} or \cite{BM} for more details).
We have thus for every $x,y\in H$,
\[
\pss{Sx}{y}=\pss{K^{*}x}{K^{*}y}=\sum_{i\in 
I}\int_{\T}\pss{x}{E_{i}(\lambda )}\,\overline{\vphantom{'}\pss{y}{E_{i}
(\lambda )}}\, d\sigma _{i}(\lambda ).
\]
Let now $G\in\LL{2}_{0}(H,\mathcal{B}_{H},\nu )$ be a function of the form $G=g\circ \Phi 
 ^{-1}$ with $g\in  \LL{2}_{0}\bigl([0,1] \bigr)\cap \mathcal{C}^{1}
 \bigl([0,1] \bigr)$, which is such that its orthogonal projection on 
$\mathcal{E}$, which we denote by $F$, is non-zero. Such a function $G$ 
does exist because $\bigl\{g\circ\Phi ^{-1};\ g\in \LL{2}_{0}\bigl([0,1] 
\bigr)\cap\mathcal{C}^{1}
 \bigl([0,1] \bigr)\bigr\}$ is dense in $\LL{2}(H,\mathcal{B}_{H},\nu )$. For every 
$n\ge 0$ we have
\[
\mathcal{C}_{n}(F,G)=\int_{Z}F(A^{n}z)\,\overline{\vphantom{'}G(z)}\,d\nu 
(z)=\pss{U_{A}^{n}F}{G}
=\pss{U_{A}^{n}F}{F}+\pss{U_{A}^{n}F}{G-F}.
\]
Since $U_{A}(\mathcal{E})\subseteq \mathcal{E}$ and $G-F$ is orthogonal to 
 $\mathcal{E}$, the second term vanishes and 
 \[
\mathcal{C}_{n}(F,G)=\mathcal{C}_{n}(F,F)=\int_{Z}F(A^{n}z)\,\overline{
\vphantom{'}F(z)}\,d\nu 
(z)\quad \textrm{for every}\ n\ge 0.
\]
Hence there exists a positive constant $C$ such that for every $n\ge 0$, 
$|\mathcal{C}_{n}(F,F)|\le C.2^{-n}$. Writing $F$ as $F=\sum_{l\ge 
1}a_{l}\,\pss{e_{l}}{\,.\,}$, where $\sum_{l\ge 1}|a_{l}|\,\sigma 
_{l}^{2}<+\infty $ and one at least of the coefficients $a_{l}$ is 
non-zero, we can write $\mathcal{C}_{n}(F,F)$ as
\begin{align*}
 \mathcal{C}_{n}(F,F)&=\sum_{k,l\ge 1}a_{k}\overline{\vphantom{t}a_{l}}\,
 \pss{SA^{*n}\,e_{k}}{e_{l}}\\&=\sum_{k,l\ge 
  1}a_{k}\overline{\vphantom{t}a_{l}}\,\sum_{i\in I}\int_{\T}
  \pss{A^{*n}e_{k}}{E_{i}(\lambda )}\,\overline{\vphantom{'}\pss{e_{l}}
  {E_{i}(\lambda )}}\,d\sigma _{i}(\lambda )\\
  &=\int_{\T}\,\lambda ^{n}\,\sum_{i\in I}\Bigl(
  \sum_{k,l\ge 
1}a_{k}\overline{\vphantom{t}a_{l}}\,\pss{e_{k}}{E_{i}(\lambda )}\,
\overline{\vphantom{'}\pss{e_{l}}{E_{i}(\lambda )}}\Bigr)\,d\sigma 
_{i}(\lambda )\\
&=\int_{\T}\,\,\lambda ^{n}\,\sum_{i\in I}\,\,\,\Bigl|\, 
\sum_{k\ge 1}
a_{k}\,\pss{e_{k}}{E_{i}(\lambda )}\,
\Bigr|^{2}\,d\sigma _{i}(\lambda ).
\end{align*}
All these computations make sense because
$$\int_{\T}\,\sum_{i\in I}\,\Bigl|\lambda ^{n}\, 
\sum_{k\ge 1}
a_{k}\,\pss{e_{k}}{E_{i}(\lambda )}
\Bigr|^{2}\,d\sigma _{i}(\lambda )=\int_{Z}|F(z)|^{2}d\nu (z)<+\infty.$$
Let us denote by $\sigma $ the positive finite measure 
$$\sigma  =\sum_{i\in I}\,\,\,\Bigl|\, 
\sum_{k\ge 1}
a_{k}\,\pss{e_{k}}{E_{i}(\,.\,)}\,
\Bigr|^{2}\sigma _{i}.$$ Then $|\widehat{\sigma }(n)|\le C.2^{-n}$ for 
every $n\ge 0$, and since $\sigma $ is a positive measure, 
$\widehat{\sigma }(-n)=\ba{\widehat{\sigma }(n)}$ for every $n\ge 0$, so that 
$|\widehat{\sigma }(n)|\le C.2^{-|n|}$ for every $n\in\Z$. Hence there 
exists a function $\varphi $ which is analytic in $\mathbb{A}_{1/2}=
\bigl\{\lambda \in\C;\ 1/2<|\lambda |<2\bigr\}$ such that $d\sigma =
\varphi \,d\lambda $, where $d\lambda $ is the normalized Lebesgue measure 
on $\T$. Since $\mathcal{C}_{0}(F,F)=\widehat{\sigma }(0)>0$ (recall that the function $F$ is \nz), the function 
$\varphi $ cannot be identically zero. If $E$ is a subset of $\T$ of 
positive Lebesgue measure, it is thus impossible that $E_{i}(\lambda )=0$ 
for every $\lambda \in E$ and every $i\in I$. So the unimodular point 
spectrum of $A$ has full Lebesgue measure in $\T$. This proves Theorem 
\ref{Theo7}.
\end{proof}
\par\smallskip 
The proof of Theorem \ref{Theo7} does not extend to operators which are 
$2$-universal for inver\-tible ergodic systems: the proof uses in a 
crucial way that the correlations $\mathcal{C}_{n}(f,g)$ of the system 
$x\mapsto 2x \mod 1$ on $[0,1]$ decay exponentially fast for \emph{all} 
$f\in\LL{2}
\bigl([0,1] \bigr)$ and sufficiently smooth $g\in\LL{2}\bigl([0,1] 
\bigr)$. This system is not invertible, and this seems to be in the nature 
of things that for an invertible system, the correlations decay 
exponentially fast only for sufficiently smooth functions $f$ \emph{and} 
$g$ (see \cite{B} for more on these questions). So the following question remains open:
\begin{question}\label{Question12}
If $A$ is a $2$-universal operator for invertible ergodic systems on a 
Hilbert space, is it true that the unimodular point spectrum of $A$ has 
full Lebesgue measure?
\end{question}

We do not know any example of a $2$-universal operator whose unimodular point spectrum is not the whole unit circle:
\begin{question}\label{Question12bis}
If $A$ is a $2$-universal operator for (invertible) ergodic systems on a 
Hilbert space, is it true that the unimodular point spectrum of $A$ is equal to $\T$?
\end{question}

If an affirmative answer to this question could be obtained, it would 
be a first step towards a characterization of symbols $\varphi \in H^{\infty }(\D)$ such 
that $M^{*}_{\varphi} $ acting on $H^{2}(\D)$ is $2$-universal for invertible 
ergodic systems.
\begin{question}\label{Question13}
 Let $\varphi \in H^{\infty }(\D)$. Is it true that $M^{*}_{\varphi} \in
 \bb(H^{2}(\D))$ is $2$-universal for invertible ergodic systems if and 
only if $\T\subseteq\varphi (\D)$?
\end{question}
It would also be interesting to obtain a characterization of adjoints of 
multipliers on $H^{2}(\D)$ which are universal for ergodic systems.
\begin{question}\label{Question14}
 Let $\varphi \in H^{\infty }(\D)$. If $\overline{\D}\subseteq \varphi 
(\D)$ (where $\overline{\D}$ denotes the closure of the unit disk $\D$), 
is $M^{*}_{\varphi} $ universal for ergodic systems? Is it true that 
$M^{*}_{\varphi}$ is $2$-universal for ergodic systems if and only if 
$\overline{\D}\subseteq\varphi (\D)$?
\end{question}
Of course things would be simpler if we knew that a universal operator is 
necessarily $2$-universal, but this does not seem easy to prove. It is not 
even known whether a frequently hypercyclic operator on a reflexive space 
admits an invariant measure with full support having a moment of order 
$2$, although it is known that it admits invariant measures with full 
support (see \cite{GM}).
\begin{question}\label{Question15}
 Does there exist a universal operator which is not $2$-universal?
\end{question}
This brings us back to questions about the existence of unimodular 
eigenvectors for universal operators.
\begin{question}\label{Question16}
 Does there exist universal (or $2$-universal) operators admitting no 
unimodular eigenvalue? What about universal operators on a Hilbert space?
\end{question}
The second half of this question seems hard, again because we do not know 
whether a frequently hypercyclic operator on a Hilbert space  
necessarily has some unimodular eigenvalue. The first half of Question 
\ref{Question16} may be more tractable, and a potential example would be 
the Kalish-type operator $A$ of Example \ref{Example4} acting on the space
$\mathcal{C}_{0}([0,2\pi])$ of continuous functions on $[0,2\pi]$ vanishing at $0$. It 
is proved in \cite{BG2} that although this operator has no unimodular 
eigenvalue, it admits a Gaussian invariant measure with full support with 
respect to which it is strongly mixing. We have seen that $A$ acting on 
$\LL{2}(\T)$ is universal for invertible ergodic systems, so one may 
naturally wonder about the following:
\begin{question}\label{Question17}
 Let $A$ be the bounded operator on $\mathcal{C}_{0}([0,2\pi])$ defined by 
setting, for every $f\in\mathcal{C}_{0}([0,2\pi])$ and every ${\theta }\in [0,2\pi]$,
\[
Af({\theta })=e^{i\theta }f({\theta 
})-\int_{0}^{\theta}ie^{it}f(e^{it})\,dt.
\]
Is $A$ a universal (or $2$-universal) operator for invertible ergodic 
systems on $\mathcal{C}_{0}([0,2\pi])$?
\end{question}
A positive answer to Question \ref{Question17} cannot be obtained via an 
application of Theorem \ref{Theo1}, or any variant of it, since we have 
seen that assumption (c) of Theorem \ref{Theo1} for instance implies that 
$A$ admits a continuous unimodular eigenvectorfield $E$ which is such that
$\overline{\vphantom{'}\textrm{span}}\bigl[E(\lambda );\ \lambda \in\T 
\bigr]=Z$.
So we finish the paper with this last question:

\begin{question}
Does there exist any universal operator for invertible ergodic systems which does not satisfy the assumptions of Theorem \ref{Theo1}? 
\end{question}

\end{document}